\documentclass[11pt]{amsart}
\usepackage[a4paper,hmargin={2.4cm,2.4cm},vmargin={2.4cm,2.4cm}]{geometry}
\usepackage[utf8]{inputenc}
\usepackage[T1]{fontenc}
\usepackage{euscript,
mathrsfs,color,latexsym,marginnote}
\usepackage{graphicx}
\usepackage{amsmath,amssymb,amsthm,amsfonts,mathabx}
\usepackage{amssymb}
\usepackage[active]{srcltx}
\usepackage{geometry}
\usepackage{hyperref}
\usepackage{comment}
\usepackage{color}
\newtheorem{thm}{Theorem}[section]

\newtheorem{lem}[thm]{Lemma}
\newtheorem{prop}[thm]{Proposition}
\newtheorem{rem}[thm]{Remark}

\newcommand{\R}{\mathbb R}

\newcommand{\N}{\mathbb N}

\newcommand{\D}{\hbox{D}}

\begin{document}
	
	\title[Local bounds for nonlinear higher-order vector]{Local bounds for nonlinear higher-order vector fields for the p-Laplace equation}
	
	\author{Felice Iandoli, Giuseppe Spadaro, Domenico Vuono}

	\email[Felice Iandoli]{felice.iandoli@unical.it}
    \email[Giuseppe Spadaro]{giuseppe.spadaro@unical.it}
	\email[Domenico Vuono]{domenico.vuono@unical.it}
	\address{Dipartimento di Matematica e Informatica, Università della Calabria,
		Ponte Pietro Bucci 31B, 87036 Arcavacata di Rende, Cosenza, Italy}
	

	\keywords{$p$-Laplacian, regularity, Calderòn-Zygmund estimates, Moser's iteration}
	
	\subjclass[2020]{35J60, 35B65}

	\begin{abstract}
		We study higher regularity for weak solutions of the $p$-Laplace equation $-\Delta_p u = f$ in a domain $\Omega \subset \mathbb{R}^n$ for $p$ sufficiently close to 2. For $m \ge 3$, assuming that $f$ satisfies suitable Sobolev and H\"older regularity conditions, we prove that the nonlinear quantity $|\nabla u|^{m-2}\nabla u$ belongs to $W^{m-1,q}_{{loc}}(\Omega)$,  and that $|\nabla u|^{m-2} D^2u$ belongs to $W^{m-2,q}_{{loc}}(\Omega)$, for any $q\ge 2$. Furthermore, we obtain uniform $L^\infty$ bounds for the weighted $(m-1)$-th and $(m-2)$-th derivatives of $|\nabla u|^{m-2}\nabla u$ and $|\nabla u|^{m-2} D^2u$ respectively, providing quantitative control even near critical points of $\nabla u$.
	\end{abstract}

 	\maketitle
  
	\section{Introduction}
	
	\noindent This paper is concerned  with the regularity of higher order derivatives  of weak solutions to:
	\begin{equation} \label{eq:problema}
		- \Delta_{p} u = f(x)  \quad \text{in } \Omega\,,
	\end{equation}
	where $\Delta_{p} u :=  \operatorname{div} (|\nabla u|^{p-2} \nabla u)$, for $p>1$, is the $p$-Laplace operator and $\Omega $ is a domain of $\R^{n}$ with $n \ge 2$. 
	For $u\in W^{1,p}_{loc}(\Omega)$, the weak formulation of \eqref{eq:problema} is the following
	\begin{equation}\label{soluzionedebole}
		\int_\Omega |\nabla u|^{p-2}(\nabla u,\nabla \psi)\,dx=\int_\Omega f \psi \,dx \quad \forall \, \psi\in C^{\infty}_c(\Omega).
	\end{equation}
    
	\noindent We have two main results, in the first one we provide $L^{\infty}$-type bounds for the derivatives of order $m$ of the vector field $|\nabla u|^{m-2}\nabla u$.
	 \begin{thm}\label{INFINITO}
	    Let $\Omega$ be a domain of $\R^n$, with $n\geq 2$. Let $\boldsymbol{\alpha}$ be a multi-index of order $m-1$, with $m\ge 3$. Let $u$ be a weak solution of~\eqref{eq:problema}. Assume that $f\in  W^{m,l}_{loc}(\Omega)$, with $l>n/2$. For any $k>0$, there exists  $\mathfrak{C}:=\mathfrak{C}(k,l,n,m)>0$ small enough such that if $|p-2|<\mathfrak{C}$, then  
     $$|\nabla u|^kD^{\boldsymbol{\alpha}} \left(|\nabla u|^{m-2}\nabla u\right)\in L^\infty_{loc}(\Omega)\quad\text{and}\quad|\nabla u|^kD^{\boldsymbol{\alpha}-1} \left(|\nabla u|^{m-2}D^2u\right)\in L^\infty_{loc}(\Omega).$$
     \end{thm}
\begin{rem}\label{rem1}
The following example shows that the presence of the weight $|\nabla u|^k$ is necessary in order to establish Theorem \ref{INFINITO}.
Consider  
$u(x_1,\ldots,x_n):=|x_1|^{p'}/p'$,  $p'=\frac{p}{p-1}$. This function satisfies $\Delta_p u=1$ in $\mathbb{R}^n$, for any $n\geq 1$. 
An easy computation shows that $|\partial_{x_1^{m-1}}(|\partial_{x_1} u|^{m-2}\partial_{x_1}u)|\approx 
|x_1|^{\frac{2-p}{p-1}(m-1)},$ which of course is not in $L^{\infty}(\Omega)$, for any $\Omega$   open set such that $\Omega\cap \{x_1=0\}\neq \emptyset$ and for any $p>2$.
On the other hand 
$|\nabla u|^k\partial_{x_1^{m-1}}(|\partial_{x_1} u|^{m-2}\partial_{x_1}u)$ is locally in  $L^{\infty}$ if $p\le 2+\frac{k}{m-1}$. Note that in this example, when $k$ goes to zero, we obtain $p\leq 2$. The exact dependence between $p$ and $k$ is unknown in general, as it relies on a non-explicit constant in the Calderón–Zygmund inequality.
\end{rem}
\noindent In the second main result we prove  that the vector field  $|\nabla u|^{m-2}\nabla u$ belongs to Sobolev space $W^{m-1,q}_{loc}$ for any $q>1$. First, we define the following space:
\begin{equation}\label{S_loc}
        S_{loc}(\Omega):=\begin{cases}
            W^{m-3,s}_{loc}(\Omega),\ \forall s \geq 1 & \text{ if $m \geq 4$}\\
            C^{0,\beta}_{loc}(\Omega), \ \ \ \quad \beta<1 & \text{ if $m = 3$}.
        \end{cases}
    \end{equation}
	\begin{thm}\label{teoprinc}
    Let $\Omega$ be a domain of $\mathbb{R}^n$ and $u$ be a weak solution of \eqref{eq:problema}. We assume that $f \in W^{m-2,q}_{loc}(\Omega)\cap S_{loc}(\Omega)$, with $m \geq 3$, $q\geq 2$ and $S_{loc}(\Omega)$ defined in \eqref{S_loc}.
    There exists $\mathcal{C}:=\mathcal{C}(n,m,q)>0$ small enough such that if $|p-2|<\mathcal{C}$, then
    \begin{equation*}
        |\nabla u|^{m-2}\nabla u \in W^{m-1,q}_{loc}(\Omega) \quad \text{and} \quad |\nabla u|^{m-2}D^2u \in W^{m-2,q}_{loc}(\Omega).
    \end{equation*}
    As a consequence, for any $\gamma \in (0,1)$, there exists $\widetilde{\mathcal{C}}$ such that if $|p-2|<\widetilde{\mathcal{C}}$, then $|\nabla u|^{m-2}\nabla u$ belongs to $C^{m-2,\gamma}_{loc}(\Omega)$ and $|\nabla u|^{m-2}D^2u$ belongs to  $C^{m-3,\gamma}_{loc}(\Omega)$.
\end{thm}
\begin{rem}
    The example given in Remark \ref{rem1} shows that the presence of the weight $|\nabla u|^{m-2}$ is necessary to establish Theorem \ref{teoprinc}.
Indeed, let $u(x_1,\ldots,x_n):=|x_1|^{p'}/p'$ be the solution to $\Delta_p u=1$ in $\mathbb{R}^n$ and let $\Omega$ be an open set such that $\Omega\cap \{x_1=0\}\neq \emptyset$, a straightforward computation shows that $\partial_{x_1^m} u \in L^q(\Omega)$ if and only if $p-1 < \frac{q}{q(m-1)-1}$. In particular, this implies that, if $p > 5/3$ then $\partial_{x_1^m} u \notin L^q(\Omega)$, for any $m\geq 3$ and any $q\geq 2$.
On the other hand 
$\partial_{x_1^{m-1}}(|\partial_{x_1} u|^{m-2}\partial_{x_1}u) \in L^q(\Omega)$ if and only if $p-2 < \frac{1}{q(m-1)-1}$, namely in a neighborhood of $p=2$.
\end{rem}
\noindent The results in the theorems above substantially strengthen and broaden those in Iandoli-Vuono \cite{IV}, and in Baratta-Sciunzi-Vuono \cite{beni}; in particular, the results therein correspond to the case $m=2$ of our statements. To the best of our knowledge, this is the first result concerning higher-order regularity for solutions to the $p$-Laplace equation.

\noindent Let us now discuss the state of the art of the regularity results concerning the equation \eqref{eq:problema}. 
\noindent In general, solutions of \eqref{eq:problema} are not classical. It is well known (see \cite{DB,GT,LIB,T}) that for every $p>1$ and under suitable assumptions on the source term,  solutions of \eqref{eq:problema} are of class $C^{1,\beta}_{loc}(\Omega)\cap C^2(\Omega\setminus Z_u)$, where $Z_u$ denotes the set where the gradient vanishes.\\
The study of higher-order derivatives of solutions to quasilinear elliptic equations, with particular focus on second-order derivatives, has been the subject of extensive research. In \cite{Cma}, the authors proved the $W^{1,2}$ regularity of the stress field under minimal assumptions on the source term, namely $f\in L^2(\Omega)$. In this direction, similar results are available in \cite{antonuovo,HS,Lou}.\\
When the source term $f$ has higher regularity, stronger second-order estimates can be obtained. In particular, in \cite{MMS} the authors established sharp second-order estimates up to the boundary for the stress field $|\nabla u|^{\alpha-1}\nabla u$, with $\alpha > (p-1)/2$, thus providing a global counterpart to the local results proved in \cite{DS}. We emphasize that the analysis in \cite{MMS} is carried out under the assumption that the domain is $C^3$-smooth. As a corollary, the authors proved that the solution $u\in W^{2,2}(\Omega)$, for $1<p<3$. Via similar techniques, analogous results have been proved for more general elliptic operators, see \cite{cellina,Dong,S1,S2}.\\
The anisotropic counterpart of these results, which involves significant technical challenges, has been studied in \cite{Anto1, Anto2, BMV, CRS}. Additionally, we highlight recent advances in regularity theory for the vectorial case, presented in \cite{BaCiDiMa, Cmavett, Minchievic, MMSV, SSV}.\\
For $p$ close to $2$, more general regularity estimates, in the spirit of a natural Calderón--Zygmund theory, were obtained in \cite{beni,GuMo,IV,MRS}. We also point out the existence of an $L^{\infty}$ estimate for second derivatives in very specific cases, such as for $p$-harmonic functions in the planar case $n=2$, see \cite{IM}.
 Within this framework, a general theory has been developed in \cite{AKM,2,4,245,DM,DM2,DM3,12,13,KuuMin,16}.

\textbf{Strategy of the proof.}  
Let $x_0\in \Omega$ and  $B_{2R}:=B_{2R}(x_0)\subset\subset\Omega$, where $B_{2R}(x_0)$ denotes the open ball of radius $2R$ centered at $x_0$. First of all we consider the regularized problem
	\begin{equation} \label{eq:problregol}
		\begin{cases}
			-\operatorname{div}\left( (\varepsilon +|\nabla u_{\varepsilon}|^{2})^{\frac{p-2}{2}} \, \nabla u_{\varepsilon} \right) = f(x) & \text{in } B_{2R} \\
			u_{\varepsilon} = u & \text{on } \partial B_{2R},
		\end{cases}
	\end{equation}
	where $\varepsilon \in (0,1)$. The existence of  a weak solution of \eqref{eq:problregol} follows by a classical minimization procedure. Notice that by standard regularity results \cite{DiKaSc,GT},  the solution $ u_\varepsilon$ is regular.
	In \cite{DB,DiKaSc,L2} (see also \cite[Section 5]{Anto2} or \cite{Ant1}) it has been proven that for any compact set $K\subset \subset B_{2R}(x_0)$ the sequence ${u}_\varepsilon$ is uniformly bounded in $C^{1,\beta '}(K)$, for some $0<\beta'<1$. Moreover, it follows that
		${u}_\varepsilon\rightarrow  u$ in the norm $\|\cdot\|_{C^{1,\beta}(K)},$ for all $0\le \beta<\beta '$.

\textbf{Ideas of the proof of Theorem \ref{teoprinc}.}
The proof of the $L^q$ regularity result is based on a uniform analysis of the regularized version of the $p$-Laplace equation, see \eqref{eq:problregol}. We consider the  smooth solutions $u_{\varepsilon}$, which allows us to differentiate the equation and derive estimates at the level of classical solutions.
The main step consists in establishing uniform weighted $L^q$-bounds for derivatives of order $m$ of $u_{\varepsilon}$, where the weights are given by suitable powers of  $(\varepsilon+|\nabla u_{\varepsilon}|^2)$. These bounds are obtained by induction on the order of differentiation. The highest-order terms yield a weighted operator acting on the derivatives of order $m$, while the remaining contributions, arising from the nonlinear structure of the operator, are treated as lower-order terms. Their control relies on Calderón–Zygmund estimates (see Lemma \ref{Zygmund}), Hölder inequalities, and multivariate Faà di Bruno formulas (see Theorem \ref{FAAC}).

In our inductive scheme, each step requires applying the inductive hypothesis to derivatives of progressively lower order; however, this reduction is carried out through H\"older inequalities, which inevitably increase the integrability exponents involved. Consequently, the lower‑order derivatives of the source term $f$ must possess sufficiently high integrability, a requirement that naturally leads us to impose the condition $f\in S_{{loc}}$. Although in principle it would suffice to assume finite integrability with a single, sufficiently large exponent $s$, we prefer to require the condition for every $s$ in order to avoid further technical complications.

The assumption that $p$ sufficiently close to $2$ ensures that the coefficients generated by the differentiation of the  operator can be absorbed, allowing the induction argument to close. As a result, one obtains uniform weighted $L^q$-estimates that are stable as $\varepsilon$ goes to zero. Passing to the limit we prove the thesis of Theorem \ref{teoprinc}. Note that the base case of this induction argument,  is given by the result in \cite{beni}, see, in particular Remark \ref{regolarizziamo}.

\textbf{Ideas of the proof of Theorem \ref{INFINITO}.}
The proof of the $L^{\infty}$-estimates builds on the weighted $L^q$-regularity established for the regularized problem. In particular, we prove that the nonlinear quantity $g_{\varepsilon}:=(\varepsilon+|\nabla u_{\varepsilon}|^2)^{\frac{m-2}{2}}\nabla u_{\varepsilon}$ satisfies the following  family of localized reverse Sobolev-type inequalities of Moser type, with constants independent of the regularization parameter

\begin{align*}
    \begin{split}
    \|g_{\varepsilon}\|_{L^{2^*q}(B_{h'})}^q\leq \mathcal{C}\frac{q}{h-h'}\|g_{\varepsilon}\|_{L^{\zeta(q-\hat{s})}(B_h)}^{q-\hat{s}}, \quad 2<\zeta<2^*,\quad  q\geq {1}/{k}=:\hat{s}.
    \end{split}
\end{align*}
These inequalities provide a quantitative gain of integrability on concentric balls and constitute the fundamental step of an iteration procedure.
At this stage, the Moser iteration scheme developed by Iandoli–Vuono \cite{IV}, inspired by the classical works \cite{Moser,serrin}, can be applied as a black box.  Therefore, the main estimate to obtain, in order to prove Theorem \ref{INFINITO} is the one above, this is done in Prop. \ref{dinodino}.  A key step is performed in the preparation Lemma \ref{pantusa1}, where we get some preliminary estimates. Here we test the  regularized $m$-linearized equation against a well chosen test function, we exlpoit  the weighted high‑order $L^q$-bounds to control commutators and nonlinear remainders.

\textbf{Organization.} In Section \ref{notations}, we recall several results that shall be systematically used throughout the paper. In Section \ref{sec:Lq} we prove Theorem \ref{teoprinc} and in Section  \ref{sec:Linf} we eventually prove Theorem \ref{INFINITO}.
	
	\section{Notations and Preliminary results}\label{notations}
We say that $\Omega \subseteq \mathbb{R}^n$ is a domain if it is open and connected. Generic fixed and numerical constants will be denoted by $C$ and they will be allowed to vary within a single line or
formula. Moreover for typographic reasons we shall sometimes write $A\lesssim B$, meaning that there exists a constant $C$ such that $A\leq CB$, such a constant may change from line to line.

For $a,b \in \mathbb{R}^n$, we denote by $a \otimes b$ the matrix whose entries are $(a \otimes b)_{ij} = a_ib_j$. Given  $A,B$ be two matrices in $\mathbb{R}^{n\times n}$, $C$ be a third-order tensor in $\mathbb{R}^{n^3}$ and $v$ be a vector in $\mathbb{R}^n$, we use the following standard notation:
\begin{equation*}
    [A B]_{ij} = \sum_{k=1}^n a_{ik}b_{kj}, \quad 
    [C[v]]_{ij} = \sum_{k=1}^n c_{ijk}v_{k}, \quad i,j=1,...,n.
\end{equation*}
We denote by $\nabla u$ and by $D^2 u$ respectively the gradient vector and the Hessian matrix of $u$.\\
We recall  the Calder\'on-Zygmund inequality  (see e.g. \cite[Corollary 9.10]{GT}):
	
	\begin{lem}\label{Zygmund}
    Let $\Omega$ be a domain of $\R^n$, $n\geq 2$, and let $w \in W^{2,q}_{0}(\Omega)$. Then there exists a positive constant $C(n,q)$ such that 
		\begin{equation} \label{eq:CZ}
			\|D^{2}w\|_{L^{q}(\Omega)} \le C(n,q) \|\Delta w\|_{L^{q}(\Omega)}.
		\end{equation}
	\end{lem}

    The following theorem is very important for our purposes. 

\begin{thm}[\cite{beni,MRS}]\label{teoremacalderon}
		Let $\Omega$ be a domain of $\R^{n}$, $n\geq 2$, and $u\in C^{1,\beta}_{loc}(\Omega)$ be a weak solution to \eqref{eq:problema}, with $f(x)\in W^{1,1}_{loc}(\Omega)\cap C_{loc}^{0,\beta'}(\Omega).$ Let $q \ge 2$ and $p$ be such that $$2-\frac{1}{C(n,q)}<p < \min\left\{2+\frac{1}{q-1},2+\frac{1}{C(n,q)}\right\},$$ with $C(n,q)$ given by \eqref{eq:CZ}. Then we have $u \in W^{2,q}_{loc}(\Omega).$
		
\end{thm}
	We shall make repeated use  of the following remark.
\begin{rem}\label{regolarizziamo}
    Let $u,q$ and $p$ as in Theorem \ref{teoremacalderon} and let $u_\varepsilon$ be a solution of \eqref{eq:problregol}. In order to establish our results we shall also exploit the fact that for any ball $B_{2R}(x_0) \subset\subset \Omega$ there exists positive constant $\mathcal{C}$ depending on $n$, $q$, $R$, $p$, $ \| \nabla u\|_{L^{\infty}(B_{2R}(x_0))}$,$\|f\|_{C^{0,\beta'}(B_{2R}(x_0))},\|f\|_{W^{1,1}(B_{2R}(x_0))}$ such that 
	$\|D^{2}u_{\varepsilon}\|_{L^{q}(B_{R}(x_0))}\le \mathcal{C}.$
   This is a consequence of Theorem \ref{teoremacalderon}, see \cite[Proof of Theorem 1.3]{beni}.
\end{rem}
Throughout this paper, we employ a multivariate version of the Faà di Bruno formula. To simplify expressions, we  recall some standard multivariate notations. If $\boldsymbol{\nu} = (\nu_1,...,\nu_n) \in \mathbb{N}^n$, then
\begin{equation*}
\begin{split}
 &|\boldsymbol{\nu}|=\sum_{i=1}^n \nu_i,\quad \boldsymbol{\nu}!= \prod_{i=1}^n \nu_i !,\quad D^{\boldsymbol{\nu}} = \frac{\partial^{|\boldsymbol{\nu}|}}{\partial_{x_1}^{\nu_1}\cdots \partial_{x_n}^{\nu_n}}, \ \text{ for } |\boldsymbol{\nu}|>0.
\end{split}
\end{equation*}
Moreover, if $\boldsymbol{\gamma} = (\gamma_1,...,\gamma_n) \in \mathbb{N}^n$ we write $\boldsymbol{\gamma}\leq \boldsymbol{\nu}$ provided that $\gamma_i \leq \nu_i$ for $i=1,\dots,n$. W also set $\tbinom{\boldsymbol{\nu}}{\boldsymbol{\gamma}} = \prod_{i=1}^n \tbinom{\nu_i}{\gamma_i}$.
In addition, we introduce the following order on $\mathbb{N}^n$. If $\boldsymbol{l} = (l_1,...,l_n)$ and $\boldsymbol{\nu} = (\nu_1,...,\nu_n)$ are in $\mathbb{N}^n$, we write $\boldsymbol{l}\prec \boldsymbol{\nu}$ provided one of the following holds:
\begin{enumerate}
    \item[i)] $|\boldsymbol{l}| < |\boldsymbol{\nu}|$;
    \item[ii)] $|\boldsymbol{l}| = |\boldsymbol{\nu}|$ and $l_1 < \nu_1$;
    \item[iii)] $|\boldsymbol{l}| = |\boldsymbol{\nu}|$, $l_1 = \nu_1,\dots,l_k=\nu_k$ and $l_{k+1} < \nu_{k+1}$ for some $1\leq k <n$.
\end{enumerate}
\begin{thm}\cite[Corollary $2.10$]{FAA}\label{FAAC}.
    Let $|\boldsymbol{\nu}|=m\geq1$ and $h(x_1,...,x_n)=f[g(x_1,...,x_n)]$ with $g\in C^{\boldsymbol{\nu}}(\boldsymbol{x_0)}$ and $f\in C^m(y_0)$, where $y_0=g(\boldsymbol{x}_0)$. Then
    \begin{equation}\label{FAA_INTRO}
        D^{\boldsymbol{\nu}}h= \sum_{r=1}^m D^rf\sum_{p(\boldsymbol{\nu},r)}(\boldsymbol{\nu}!)\prod_{j=1}^m \frac{(D^{\boldsymbol{l}_j}g)^{k_j}}{(k_j!)(\boldsymbol{l}_j!)^{k_j}},
    \end{equation}
    where
\begin{equation*}
\begin{split}
p(\boldsymbol{\nu},r) = \{&(k_1,...,k_m;\boldsymbol{l}_1,...,\boldsymbol{l}_m): \text{for some }1\leq s \leq m,\ k_i=0 \text{ and } \boldsymbol{l}_i = \boldsymbol{0}\\
&\quad\text{ for } 1\leq i\leq m-s;\ k_i > 0 \text{ for } m - s+1 \leq i \leq m;\\
&\qquad \text{ and } 0 \prec \boldsymbol{l}_{m-s+1} \prec \cdots\prec \boldsymbol{l}_{m} \text{ are such that }\sum_{i=1}^{m} k_i = r, \sum_{i=1}^{m} k_i \boldsymbol{l}_i = \boldsymbol{\nu}\}.
\end{split}
\end{equation*}
\end{thm}
\begin{rem}
    Note that , see \cite[Remark 2.2]{FAA}, $p(\boldsymbol{\nu},r)$ can be identified with the union of $p_s(\boldsymbol{\nu},r)$ for $s=1,...,m$. Thus, we can rewrite \eqref{FAA_INTRO} as follows:
    \begin{equation}\label{FFA_s}
        D^{\boldsymbol{\nu}}h= \sum_{r=1}^m D^rf\sum_{s=1}^m\sum_{p_s(\boldsymbol{\nu},r)}(\boldsymbol{\nu}!)\prod_{j=1}^s \frac{(D^{\boldsymbol{l}_j}g)^{k_j}}{(k_j!)(\boldsymbol{l}_j!)^{k_j}},
    \end{equation}
    where
\begin{equation*}
\begin{split}
    p_s(\boldsymbol{\nu},{r})= \Big\{({k}_1,...,{k}_s;\boldsymbol{l}_1,...,\boldsymbol{l}_s):  0\prec \boldsymbol{l}_1\prec\cdots \prec\boldsymbol{l}_s;k_i>0;\sum_{i=1}^s k_i ={r}, \sum_{i=1}^{s} k_i \boldsymbol{l}_i = \boldsymbol{\nu}\Big\}.
\end{split}
\end{equation*}   
\end{rem}
The following approximation lemma shall be useful.
\begin{lem}\label{approssimazione}
Let $\Omega$ be a domain in $\mathbb{R}^n$, $n\geq 2$. Let $f \in W^{1,q}_{loc}(\Omega)\cap C^{0, \beta}_{loc}(\Omega)$, with $ \beta<1$ and $q \geq 1$. Then there  exists a sequence  of smooth functions $\{f_l\} \subset C^\infty(\Omega)$ such that $f_l$ converges to $f$ in $W^{1,q}_{loc}(\Omega)$.
Moreover, for any $K \subset \subset \Omega$, $\|f_l\|_{C^{0,\beta}(K)}\leq \|f\|_{C^{0,\beta}(K_\delta)}$, where  $K_\delta := \{x\in\Omega:\mathrm{dist}(x,K)<\delta\}\subset\subset\Omega$.
\end{lem}
\begin{proof}
  Let $K\subset\subset \Omega$. Then, there exists $\delta>0$ such that
    $K_\delta := \{x\in\Omega:\mathrm{dist}(x,K)<\delta\}\subset\subset\Omega.$
Let $\rho\in C_c^\infty(\mathbb{R}^n)$ be a standard mollifier with
$\rho\ge 0$, $\mathrm{supp}\,\rho\subset B_1(0)$, and $\int_{\mathbb{R}^n}\rho=1$, and define $\rho_l(x)=l^{-n}\rho(x/l)$, $f_l := f*\rho_l.$
For $l<\delta$, we have $f_l\in C^\infty(K)$ and, by standard approximation arguments
    $f_l \rightarrow f$ in  $ W^{1,q}(K).$
For $x,y\in K$, $x\neq y$ and $l<\delta$ it follows that $x-z$, $y-z \in K_\delta$. Thus,
\begin{equation*}
    \begin{split}
        |f_l(x)-f_l(y)|
&=\left|\int_{\mathbb{R}^n} (f(x-z)-f(y-z))\rho_l(z)\,dz\right|\\
&\le \int_{\mathbb{R}^n} |f(x-z)-f(y-z)|\rho_l(z)\,dz\le [f]_{C^{0,\beta}(K_\delta)}|x-y|^{ \beta}.
    \end{split}
\end{equation*}
Hence, $[f_l]_{C^{0,\beta}(K)} \le [f]_{C^{0,\beta}(K_\delta)}$.
Moreover, for any $x\in K$
\begin{equation*}
    |f_l(x)|  =\left|\int_{\mathbb{R}^n} f(x-z)\rho_l(z)\,dz\right| \leq \|f\|_{L^\infty(K_\delta)}.\qedhere
\end{equation*}
\end{proof}
\section{\texorpdfstring{Weighted $L^q$ estimates}{Weighted Lq estimates}}\label{sec:Lq}
As a preliminary step toward the proof of Theorem \ref{teoprinc}, we establish the following lemma providing a uniform weighted estimate for the $m$-th derivatives of solutions to the regularized problem \eqref{eq:problregol}.
\begin{lem}\label{sisso}
    Let $\Omega$ be a domain in $\mathbb{R}^n$, $n\geq 2$ and fix a point $x_0 \in \Omega$ and a radius $R>0$ such that $B_{2R}(x_0)\subset\subset \Omega$. Let $q\geq 2$ and $\boldsymbol{\nu}$ be a multi-index of order $m-2$, with $m\geq 2$.
    Assume that $u_\varepsilon$ is a solution of \eqref{eq:problregol}, with $f \in C^{m-2,\beta'}_{loc}(\Omega),$ for some $\beta'\in (0,1)$. There exists $\mathcal{C}:=\mathcal{C}(n,m,q)>0$ small enough such that if $|p-2|<\mathcal{C}$, then
    \begin{equation}\label{RESULT}
        \left\|(\varepsilon+|\nabla u_\varepsilon|^2)^{\frac{m-2}{2}}D^{\boldsymbol{\nu}}(D^2 u_\varepsilon)\right\|_{L^q(B_R(x_0))}\leq C,
    \end{equation}
     where $C=C(n,m,q,R,\|D^{\boldsymbol{\nu}}f\|_{L^q_{loc}(\Omega)},\|f\|_{S_{loc}(\Omega)},p,\|\nabla u\|_{L^\infty_{loc}(\Omega)})$ is a positive constant independent of $\varepsilon$.
\end{lem}
\begin{proof}
We proceed by induction. 
The base case consists in proving that \eqref{RESULT} is true for $m=2$. This is  a consequence of Remark \ref{regolarizziamo}.
Let us now assume that \eqref{RESULT} is true for any $2\leq k \leq m-1$, and any $\tilde \Omega\subset \subset B_{2R}$, namely:
\begin{equation}\label{Inductive_Step}
    \left\|(\varepsilon+|\nabla u_\varepsilon|^2)^\frac{k-2}{2}D^{\boldsymbol{\alpha}}(D^2 u_\varepsilon)\right\|_{L^q(\tilde \Omega)}\leq C,
\end{equation}
for any $k\leq m-1$. Note that in this case $|\boldsymbol{\alpha}|=k-2$ and $C$ is a positive constant independent of $\varepsilon$ and depending on $n,m,q, \tilde \Omega,\|f\|_{S_{loc}(\Omega)},p,\|\nabla u\|_{L^\infty_{loc}(\Omega)}$ . We remark that the assumption on $f$ ensures that the solution $u_\varepsilon$ is of class $C^m$.\\
%
%
Let $R'>0$ fixed, sufficiently small. Let $\varphi \in C_c^\infty(B_{2R'})$ such that $\varphi=1$ in $B_{R'}$, $\varphi=0$ in $(B_{2R'})^c$ and $|\nabla \varphi|\leq \frac{2}{R'}$ in $B_{2R'}\setminus B_{R'}$. For simplicity of notation, we will use $R$ instead of $R'$. Once the theorem has been proved for sufficiently small subsets, the general result follows by a covering argument, extending the estimate from radius $R'$ to $R$.\\
Let $m\geq3$ and a multi-index $\boldsymbol{\nu} = (\nu_1,...,\nu_n)$ such that $|\boldsymbol{\nu}|=m-2$. Since $(\varepsilon+|\nabla u_\varepsilon|^2)^\frac{m-2}{2} D^{\boldsymbol{\nu}} u_\varepsilon\varphi \in W^{2,q}_0(B_{2R})$, using Lemma \ref{Zygmund}, we have
\begin{equation}\label{C_Z_nu}
\begin{split}
&\left \|D^2((\varepsilon+|\nabla u_\varepsilon|^2)^\frac{m-2}{2} D^{\boldsymbol{\nu}} u_\varepsilon\varphi) \right \|_{L^q(B_{2R})}\leq C(n,q)\left \|\Delta ((\varepsilon+|\nabla u_\varepsilon|^2)^\frac{m-2}{2} D^{\boldsymbol{\nu}} u_\varepsilon\varphi)\right \|_{L^q(B_{2R})}.
\end{split}
\end{equation}
Developing the Laplacian on the right hand side of \eqref{C_Z_nu}, we get:
\begin{equation}\label{RHS_phi}
\begin{split}
    &\Delta ((\varepsilon+|\nabla u_\varepsilon|^2)^\frac{m-2}{2} D^{\boldsymbol{\nu}} u_\varepsilon\varphi)\\
    &\qquad=\Delta ((\varepsilon+|\nabla u_\varepsilon|^2)^\frac{m-2}{2} D^{\boldsymbol{\nu}} u_\varepsilon)\varphi + 2 \nabla ((\varepsilon+|\nabla u_\varepsilon|^2)^\frac{m-2}{2} D^{\boldsymbol{\nu}} u_\varepsilon)\cdot \nabla \varphi\\
    &\qquad \qquad + (\varepsilon+|\nabla u_\varepsilon|^2)^\frac{m-2}{2} D^{\boldsymbol{\nu}} u_\varepsilon\Delta\varphi
\end{split}
\end{equation}
Expanding the second order derivative on the left hand side of \eqref{C_Z_nu}, we obtain
\begin{equation}\label{LHS_phi}
    \begin{split}
        &D^2((\varepsilon+|\nabla u_\varepsilon|^2)^\frac{m-2}{2} D^{\boldsymbol{\nu}} u_\varepsilon\varphi)\\
        &\quad = D^2((\varepsilon+|\nabla u_\varepsilon|^2)^\frac{m-2}{2} D^{\boldsymbol{\nu}} u_\varepsilon)\varphi + \nabla ((\varepsilon+|\nabla u_\varepsilon|^2)^\frac{m-2}{2} D^{\boldsymbol{\nu}} u_\varepsilon) \otimes \nabla \varphi\\
        &\quad\qquad + \nabla \varphi \otimes \nabla ((\varepsilon+|\nabla u_\varepsilon|^2)^\frac{m-2}{2} D^{\boldsymbol{\nu}} u_\varepsilon) +(\varepsilon+|\nabla u_\varepsilon|^2)^\frac{m-2}{2} D^{\boldsymbol{\nu}} u_\varepsilon D^2\varphi
    \end{split}
\end{equation}
Moreover, expanding the Laplacian on the first term of the right-hand side of \eqref{RHS_phi}, we obtain
\begin{equation}\label{CZ_RHS}
    \Delta ((\varepsilon+|\nabla u_\varepsilon|^2)^\frac{m-2}{2} D^{\boldsymbol{\nu}} u_\varepsilon) = (\varepsilon+|\nabla u_\varepsilon|^2)^\frac{m-2}{2} D^{\boldsymbol{\nu}}(\Delta u_\varepsilon) + R^1_{\leq \nu+1},
\end{equation}
where,
\begin{equation}\label{LOT_1}
\begin{split}
R^1_{\leq \nu+1}&:= (m-2)(m-4) (\varepsilon+|\nabla u_\varepsilon|^2)^\frac{m-6}{2} |D^2u_\varepsilon\nabla u_\varepsilon|^2 D^{\boldsymbol{\nu}} u_\varepsilon\\
&\qquad +  (m-2) (\varepsilon+|\nabla u_\varepsilon|^2)^\frac{m-4}{2} |D^2u_\varepsilon|^2 D^{\boldsymbol{\nu}} u_\varepsilon\\
&\qquad + 2 (m-2) (\varepsilon+|\nabla u_\varepsilon|^2)^\frac{m-4}{2} (D^2u_\varepsilon\nabla u_\varepsilon, D^{\boldsymbol{\nu}}(\nabla u_\varepsilon))\\
&\qquad +(m-2) (\varepsilon+|\nabla u_\varepsilon|^2)^\frac{m-4}{2} (\nabla ( \Delta u_\varepsilon),\nabla u_\varepsilon) D^{\boldsymbol{\nu}} u_\varepsilon.
\end{split}
\end{equation}
Expanding the gradient on the second term of the right-hand side of \eqref{RHS_phi}, we get
\begin{equation}\label{CZ_RHS_2}
\begin{split}
R^2_{\leq \nu +1}&:=\nabla ((\varepsilon+|\nabla u_\varepsilon|^2)^\frac{m-2}{2} D^{\boldsymbol{\nu}} u_\varepsilon)\\
&\ = (\varepsilon+|\nabla u_\varepsilon|^2)^\frac{m-2}{2} D^{\boldsymbol{\nu}} (\nabla u_\varepsilon)\\
&\qquad +  (m-2) (\varepsilon+|\nabla u_\varepsilon|^2)^\frac{m-4}{2}(D^2u_\varepsilon\nabla u_\varepsilon) D^{\boldsymbol{\nu}} u_\varepsilon.
\end{split}
\end{equation}
Therefore, we can rewrite \eqref{RHS_phi} as follows
\begin{equation}\label{RHS_phi_LOT}
\begin{split}
    \Delta ((\varepsilon+|\nabla u_\varepsilon|^2)^\frac{m-2}{2} D^{\boldsymbol{\nu}} u_\varepsilon\varphi)
    &=(\varepsilon+|\nabla u_\varepsilon|^2)^\frac{m-2}{2} D^{\boldsymbol{\nu}}(\Delta u_\varepsilon)\varphi + R^1_{\leq \nu+1}\varphi\\
    &\qquad + 2R^2_{\leq \nu +1} \cdot \nabla \varphi+ R^3_{\leq \nu +1}\Delta\varphi,
\end{split}
\end{equation}
where,
    $R^3_{\leq \nu +1}:=(\varepsilon+|\nabla u_\varepsilon|^2)^\frac{m-2}{2} D^{\boldsymbol{\nu}} u_\varepsilon.$
Let us now deal with \eqref{LHS_phi}, expanding the second order derivative on the first term of the right hand side of \eqref{LHS_phi}, we get
\begin{equation}\label{CZ_LHS_LOT}
    D^2((\varepsilon+|\nabla u_\varepsilon|^2)^\frac{m-2}{2} D^{\boldsymbol{\nu}} u_\varepsilon) = (\varepsilon+|\nabla u_\varepsilon|^2)^\frac{m-2}{2} D^{\boldsymbol{\nu}}(D^2u_\varepsilon) + R^4_{\leq \nu +1},
\end{equation}
where,
\begin{equation}\label{LOT_4}
\begin{split}
R^4_{\leq \nu +1}
    &:= (m-2)(m-4)(\varepsilon+|\nabla u_\varepsilon|^2)^\frac{m-6}{2}(D^2u_\varepsilon\nabla u_\varepsilon)\otimes (D^2u_\varepsilon\nabla u_\varepsilon) D^{\boldsymbol{\nu}}u_\varepsilon\\
    &\qquad+ (m-2)(\varepsilon+|\nabla u_\varepsilon|^2)^\frac{m-4}{2}(D^2u_\varepsilon\nabla u_\varepsilon)\otimes D^{\boldsymbol{\nu}}(\nabla u_\varepsilon)\\
    &\qquad+ (m-2)(\varepsilon+|\nabla u_\varepsilon|^2)^\frac{m-4}{2}D^{\boldsymbol{\nu}}(\nabla u_\varepsilon)\otimes(D^2u_\varepsilon\nabla u_\varepsilon)\\
    &\qquad + (m-2)(\varepsilon+|\nabla u_\varepsilon|^2)^\frac{m-4}{2}(D^2u_\varepsilon D^2u_\varepsilon) D^{\boldsymbol{\nu}}u_\varepsilon\\
    &\qquad + (m-2)(\varepsilon+|\nabla u_\varepsilon|^2)^\frac{m-4}{2}(D^3u_\varepsilon[\nabla u_\varepsilon]) D^{\boldsymbol{\nu}}u_\varepsilon.
\end{split}
\end{equation}
Thus, \eqref{LHS_phi} can be rewritten as follows
\begin{equation}\label{LHS_phi_L}
    \begin{split}
        D^2((\varepsilon+|\nabla u_\varepsilon|^2)^\frac{m-2}{2} D^{\boldsymbol{\nu}} u_\varepsilon\varphi)
        & = (\varepsilon+|\nabla u_\varepsilon|^2)^\frac{m-2}{2} D^{\boldsymbol{\nu}}(D^2u_\varepsilon) \varphi + R^4_{\leq \nu +1}\varphi\\
        &\qquad + R^2_{\leq \nu+1} \otimes \nabla \varphi + \nabla \varphi \otimes R^2_{\leq \nu+1} + R^3_{\leq \nu+1} D^2\varphi.
    \end{split}
\end{equation}
Therefore, using a reverse triangle inequality, by \eqref{RHS_phi_LOT} and \eqref{LHS_phi_L}, the inequality \eqref{C_Z_nu} becomes
\begin{equation}\label{stima_G_N}
\begin{split}
&\left\|(\varepsilon+|\nabla u_\varepsilon|^2)^\frac{m-2}{2} D^{\boldsymbol{\nu}}(D^2u_\varepsilon)\varphi\right\|_{L^q(B_{2R})}\\
&\quad\leq C(n,q)\left\|(\varepsilon+|\nabla u_\varepsilon|^2)^\frac{m-2}{2}D^{\boldsymbol{\nu}}(\Delta u_\varepsilon)\varphi\right\|_{L^q(B_{2R})}\\
&\qquad + C(n,q,R) \left(\left\|R^1_{\leq \nu+1}\right\|_{L^q(B_{2R})} + \left\|R^2_{\leq \nu+1}\right\|_{L^q(B_{2R})}\right.\\
&\qquad\qquad\left.+ \left\|R^3_{\leq \nu+1}\right\|_{L^q(B_{2R})} + \left\|R^4_{\leq \nu+1}\right\|_{L^q(B_{2R})}\right),
\end{split}
\end{equation}
where $C(n,q,R)$ is a positive constant.\\
\textbf{Bound on lower-order terms.} First, we estimate the term $R^1_{\leq \nu+1}$.
\begin{equation}\label{SL1}
\begin{split}
   \left \|R^1_{\leq \nu+1}\right \|_{L^q(B_{2R})} &\leq C(m,n) \left(\left\||D^2u_\varepsilon|^2(\varepsilon+|\nabla u_\varepsilon|^2)^\frac{m-4}{2}D^{\boldsymbol{\nu}}u_\varepsilon\right\|_{L^q(B_{2R})} \right.\\
    &\qquad\left.+\left\||D^2u_\varepsilon|(\varepsilon+|\nabla u_\varepsilon|^2)^\frac{m-3}{2}D^{\boldsymbol{\nu}}(\nabla u_\varepsilon)\right\|_{L^q(B_{2R})}\right.\\
    &\qquad\left.+ \left\|(\varepsilon+|\nabla u_\varepsilon|^2)^\frac{1}{2}|D^3u_\varepsilon|(\varepsilon+|\nabla u_\varepsilon|^2)^\frac{m-4}{2}D^{\boldsymbol{\nu}}u_\varepsilon\right\|_{L^q(B_{2R})} \right).
\end{split}
\end{equation}
Since the three terms on the right-hand side of \eqref{SL1} can be treated in an equivalent way, we focus on the second one. Using H\"older inequalities  with exponents $(q_1,q_2)$ such that $1/q_1+1/q_2=1/q$, we obtain
\begin{equation*}
\begin{split}
    &\left\||D^2u_\varepsilon|(\varepsilon+|\nabla u_\varepsilon|^2)^\frac{m-3}{2}D^{\boldsymbol{\nu}}(\nabla u_\varepsilon)\right\|_{L^q(B_{2R})}\\
    &\qquad\leq \left\| D^2u_\varepsilon\right\|_{L^{q_1}(B_{2R})}\left\|(\varepsilon+|\nabla u_\varepsilon|^2)^\frac{m-3}{2}D^{\boldsymbol{\nu}}(\nabla u_\varepsilon)\right\|_{L^{q_2}(B_{2R})}.
\end{split}
\end{equation*}
Combining the base case with the inductive step, we obtain the existence of a range of $p$ close to $2$ such that
\begin{equation*}
    \left\||D^2u_\varepsilon|(\varepsilon+|\nabla u_\varepsilon|^2)^\frac{m-3}{2}D^{\boldsymbol{\nu}}(\nabla u_\varepsilon)\right\|_{L^q(B_{2R})} \leq C,
\end{equation*}
where $C=C(n,m,q,\|f\|_{S_{loc}(\Omega)},R,p,\|\nabla u\|_{L^\infty_{loc}(\Omega)})$ is a positive constant independent of $\varepsilon$.
Proceeding as before for the remaining two terms in \eqref{SL1}, we get
\begin{equation}\label{stima_LOT_1}
    \left \|R^1_{\leq \nu+1}\right \|_{L^q(B_{2R})} \leq \mathcal{C}(n,m,q,\|f\|_{S_{loc}(\Omega)},R,p,\|\nabla u\|_{L^\infty_{loc}(\Omega)}),
\end{equation}
where $\mathcal{C}(n,m,q,\|f\|_{S_{loc}(\Omega)},R,p,\|\nabla u\|_{L^\infty_{loc}(\Omega)})$ is a positive constant.\\
Note that the the term $R^4_{\leq \nu+1}$ is similarly estimated.
We proceed by estimating the term $R^2_{\leq \nu+1}$.
\begin{equation}\label{SL2}
\begin{split}
   \left \|R^2_{\leq \nu+1}\right \|_{L^q(B_{2R})} &\leq \left\|(\varepsilon+|\nabla u_\varepsilon|^2)^\frac{m-2}{2}D^{\boldsymbol{\nu}}(\nabla u_\varepsilon)\right\|_{L^q(B_{2R})}\\
    &\qquad + (m-2) \left\||D^2u_\varepsilon||\nabla u_\varepsilon|(\varepsilon+|\nabla u_\varepsilon|^2)^\frac{m-4}{2}D^{\boldsymbol{\nu}} u_\varepsilon\right\|_{L^q(B_{2R})}.
\end{split}
\end{equation}
Again, using H\"older inequalities with exponents $(q_1,q_2)$ such that $1/q_1+1/q_2=1/q$ together with the fact that $|\nabla u_\varepsilon|\leq C$ independent of $\varepsilon$, see \cite[Theorem 1.7]{L2} or \cite{Ant2} or \cite[Theorem $2.1$]{Cma2}, we conclude that
\begin{equation}\label{stima_LOT_2}
    \left \|R^2_{\leq \nu+1}\right \|_{L^q(B_{2R})} \leq \mathcal{C}(n,m,q,\|f\|_{S_{loc}(\Omega)},R,p,\|\nabla u\|_{L^\infty_{loc}(\Omega)}).
\end{equation}
The same bound holds true for $R^3_{\leq \nu+1}$.
Combining the estimates \eqref{stima_LOT_1} and \eqref{stima_LOT_2}, we rewrite \eqref{stima_G_N} in the following way
\begin{equation}\label{stima_N_LOT}
\begin{split}
&\left\|(\varepsilon+|\nabla u_\varepsilon|^2)^\frac{m-2}{2} D^{\boldsymbol{\nu}}(D^2u_\varepsilon)\varphi\right\|_{L^q(B_{2R})}\\
&\quad\leq C(n,q)\left\|(\varepsilon+|\nabla u_\varepsilon|^2)^\frac{m-2}{2}D^{\boldsymbol{\nu}}(\Delta u_\varepsilon)\varphi\right\|_{L^q(B_{2R})}\\
&\quad\qquad + \mathcal{C}(n,m,q,\|f\|_{S_{loc}(\Omega)},R,p,\|\nabla u\|_{L^\infty_{loc}(\Omega)}).
\end{split}
\end{equation}
\textbf{Dominant term.} 
We aim now to estimate the dominant term given by a weighted $m$-th order derivative of $u_\varepsilon$ solution of the regularized problem \eqref{eq:problregol}. Taking $D^\nu$ on both sides of \eqref{eq:problregol} we have
\begin{equation}\label{pEq}
- D^{\boldsymbol{\nu}} f 
=D^{\boldsymbol{\nu}} ((\varepsilon+|\nabla u_\varepsilon|^2)^\frac{p-2}{2}\Delta u_\varepsilon) + (p-2) D^{\boldsymbol{\nu}} [(\varepsilon+|\nabla u_\varepsilon|^2)^\frac{p-4}{2} (D^2u_\varepsilon \nabla u_\varepsilon,\nabla u_\varepsilon)].
\end{equation}
Using the Leibniz formula on the first term of the right-hand side of \eqref{pEq} we get
\begin{equation}\label{Acca}
\begin{split}
&D^{\boldsymbol{\nu}} ((\varepsilon+|\nabla u_\varepsilon|^2)^\frac{p-2}{2}\Delta u_\varepsilon)\\
&\quad= \sum_{\boldsymbol{\gamma}\leq \boldsymbol{\nu}} \tbinom{\boldsymbol{\nu}}{\boldsymbol{\gamma}} D^{\boldsymbol{\gamma}}\left((\varepsilon+|\nabla u_\varepsilon|^2)^\frac{p-2}{2}\right) D^{\boldsymbol{\nu}-\boldsymbol{\gamma}}(\Delta u_\varepsilon)\\
&\quad= (\varepsilon+|\nabla u_\varepsilon|^2)^\frac{p-2}{2} D^{\boldsymbol{\nu}}(\Delta u_\varepsilon) + \sum_{0<\boldsymbol{\gamma}\leq \boldsymbol{\nu}} \tbinom{\boldsymbol{\nu}}{\boldsymbol{\gamma}} D^{\boldsymbol{\gamma}}\left((\varepsilon+|\nabla u_\varepsilon|^2)^\frac{p-2}{2}\right) D^{\boldsymbol{\nu}-\boldsymbol{\gamma}}(\Delta u_\varepsilon),
\end{split}
\end{equation}
where the notation $\boldsymbol{\gamma} > 0$ stands for $|\boldsymbol{\gamma}|>0$.\\
Multiplying the last equation by $(\varepsilon+|\nabla u_\varepsilon|^2)^\frac{m-2-(p-2)}{2}$, we obtain
\begin{equation}\label{m_Der}
\begin{split}
&(\varepsilon+|\nabla u_\varepsilon|^2)^\frac{m-2}{2} D^{\boldsymbol{\nu}}(\Delta u_\varepsilon)\\
&\quad=(\varepsilon+|\nabla u_\varepsilon|^2)^\frac{m-p}{2} D^{\boldsymbol{\nu}} ((\varepsilon+|\nabla u_\varepsilon|^2)^\frac{p-2}{2}\Delta u_\varepsilon)\\
&\quad\qquad - (\varepsilon+|\nabla u_\varepsilon|^2)^\frac{m-p}{2}\sum_{0<\boldsymbol{\gamma}\leq \boldsymbol{\nu}} \tbinom{\boldsymbol{\nu}}{\boldsymbol{\gamma}}D^{\boldsymbol{\gamma}}\left((\varepsilon+|\nabla u_\varepsilon|^2)^\frac{p-2}{2}\right) D^{\boldsymbol{\nu}-\boldsymbol{\gamma}}(\Delta u_\varepsilon)\\
&\quad=(\varepsilon+|\nabla u_\varepsilon|^2)^\frac{m-p}{2} D^{\boldsymbol{\nu}} ((\varepsilon+|\nabla u_\varepsilon|^2)^\frac{p-2}{2}\Delta u_\varepsilon)\\
&\quad\qquad - \sum_{0<\boldsymbol{\gamma}\leq \boldsymbol{\nu}} \tbinom{\boldsymbol{\nu}}{\boldsymbol{\gamma}}\left ( (\varepsilon+|\nabla u_\varepsilon|^2)^\frac{2-p+|\boldsymbol{\gamma}|}{2}D^{\boldsymbol{\gamma}}\left((\varepsilon+|\nabla u_\varepsilon|^2)^\frac{p-2}{2}\right) \right )\\
&\qquad\qquad\quad\times(\varepsilon+|\nabla u_\varepsilon|^2)^\frac{m-2-|\boldsymbol{\gamma}|}{2}D^{\boldsymbol{\nu}-\boldsymbol{\gamma}}(\Delta u_\varepsilon)
\end{split}
\end{equation}
To proceed, we expand $D^{\boldsymbol{\gamma}}\left((\varepsilon+|\nabla u_\varepsilon|^2)^\frac{p-2}{2}\right)$ exploiting a multivariate version of the Faà di Bruno formula  (see Theorem \ref{FAAC}), applied to the composition  $g_1(g_2(x))$, where $g_1(t):=(\varepsilon+t)^{(p-2)/2}$ and $g_2(x):=|\nabla u_\varepsilon|^2$. In particular, we have 
\begin{equation}\label{Bruno_1}
\begin{split}
&(\varepsilon+|\nabla u_\varepsilon|^2)^\frac{2-p+|\boldsymbol{\gamma}|}{2}D^{\boldsymbol{\gamma}}\left((\varepsilon+|\nabla u_\varepsilon|^2)^\frac{p-2}{2}\right)\\
&\qquad=(\varepsilon+|\nabla u_\varepsilon|^2)^\frac{2-p+|\boldsymbol{\gamma}|}{2}\sum_{r=1}^{|\gamma|}c_1(\varepsilon+|\nabla u_\varepsilon|^2)^{\frac{p-2}{2}-r} \sum_{p(\boldsymbol{\gamma},r)}(\boldsymbol{\gamma}!) \prod_{j=1}^{|\boldsymbol{\gamma}|}\frac{\left( D^{\boldsymbol{l}_j}(|\nabla u_\varepsilon|^2)\right)^{k_j}}{c_2(j)},
\end{split}
\end{equation}
where $c_1=  \prod_{h=1}^r \frac{(p-2h)}{2}$, $c_2(j)= k_j! (\boldsymbol{l}_j !)^{k_j}$ and
\begin{equation*}
\begin{split}
p(\boldsymbol{\gamma},r) = \{&(k_1,...,k_{|\boldsymbol{\gamma}|};\boldsymbol{l}_1,...,\boldsymbol{l}_{|\boldsymbol{\gamma}|}): \text{for some }1\leq s \leq |\boldsymbol{\gamma}|,\ k_i=0 \text{ and } \boldsymbol{l}_i = \boldsymbol{0}\\
&\quad\text{ for } 1\leq i\leq |\boldsymbol{\gamma}|-s;\ k_i > 0 \text{ for } |\boldsymbol{\gamma}| - s+1 \leq i \leq |\boldsymbol{\gamma}|;\\
&\qquad \text{ and } 0 \prec \boldsymbol{l}_{|\boldsymbol{\gamma}|-s+1} \prec \cdots\prec \boldsymbol{l}_{|\boldsymbol{\gamma}|} \text{ are such that }\sum_{i=1}^{|\boldsymbol{\gamma}|} k_i = r, \sum_{i=1}^{|\boldsymbol{\gamma}|} k_i \boldsymbol{l}_i = \boldsymbol{\gamma}\}.
\end{split}
\end{equation*}
Using the constraints $\sum_{i=1}^{|\boldsymbol{\gamma}|} k_i = r$ and $\sum_{i=1}^{|\boldsymbol{\gamma}|} k_i \boldsymbol{l}_i = \boldsymbol{\gamma}$, we can rewrite \eqref{Bruno_1} as follows:
\begin{equation}\label{After_c_1}
\begin{split}
&(\varepsilon+|\nabla u_\varepsilon|^2)^\frac{2-p+|\boldsymbol{\gamma}|}{2}D^{\boldsymbol{\gamma}}\left((\varepsilon+|\nabla u_\varepsilon|^2)^\frac{p-2}{2}\right)\\
&\quad=\sum_{r=1}^{|\gamma|}c_1 \sum_{p(\boldsymbol{\gamma},r)}(\boldsymbol{\gamma}!)(\varepsilon+|\nabla u_\varepsilon|^2)^{\frac{|\boldsymbol{\gamma}|}{2}-\sum_{i=1}^{|\boldsymbol{\gamma}|} k_i}\prod_{j=1}^{|\boldsymbol{\gamma}|}\frac{\left( D^{\boldsymbol{l}_j}(|\nabla u_\varepsilon|^2)\right)^{k_j}}{c_2(j)}\\
&\quad=\sum_{r=1}^{|\gamma|}c_1 \sum_{p(\boldsymbol{\gamma},r)}(\boldsymbol{\gamma}!) (\varepsilon+|\nabla u_\varepsilon|^2)^{\frac{\sum_{i=1}^{|\boldsymbol{\gamma}|} k_i |\boldsymbol{l}_i|}{2}-\sum_{i=1}^{|\boldsymbol{\gamma}|} k_i}\prod_{j=1}^{|\boldsymbol{\gamma}|}\frac{\left( D^{\boldsymbol{l}_j}(|\nabla u_\varepsilon|^2)\right)^{k_j}}{c_2(j)}\\
&\quad=\sum_{r=1}^{|\gamma|}c_1 \sum_{p(\boldsymbol{\gamma},r)}(\boldsymbol{\gamma}!) \prod_{j=1}^{|\boldsymbol{\gamma}|}\frac{\left( (\varepsilon+|\nabla u_\varepsilon|^2)^{\frac{|\boldsymbol{l}_j|}{2}-1}D^{\boldsymbol{l}_j}(|\nabla u_\varepsilon|^2)\right)^{k_j}}{c_2(j)}.
\end{split}
\end{equation}
Moreover, using the Leibniz formula,
\begin{equation}\label{L_M}
\begin{split}
    D^{\boldsymbol{l}_j}(|\nabla u_\varepsilon|^2) &= D^{\boldsymbol{l}_j}((\nabla u_\varepsilon,\nabla u_\varepsilon))=\sum_{\boldsymbol{\beta}\leq \boldsymbol{l}_j} \tbinom{\boldsymbol{l}_j}{\boldsymbol{\beta}} \left(D^{\boldsymbol{\beta}}(\nabla u_\varepsilon),D^{\boldsymbol{l}_j-\boldsymbol{\beta}}(\nabla u_\varepsilon)\right).
\end{split}
\end{equation}
Therefore, using \eqref{L_M}, we can rewrite \eqref{After_c_1} as follows:
\begin{equation}\label{stima_F}
\begin{split}
&(\varepsilon+|\nabla u_\varepsilon|^2)^\frac{2-p+|\boldsymbol{\gamma}|}{2}D^{\boldsymbol{\gamma}}\left((\varepsilon+|\nabla u_\varepsilon|^2)^\frac{p-2}{2}\right)\\
&\quad=\sum_{r=1}^{|\gamma|}c_1 \sum_{p(\boldsymbol{\gamma},r)}(\boldsymbol{\gamma}!) \prod_{j=1}^{|\boldsymbol{\gamma}|}\frac{1}{c_2(j)}\\
&\qquad \quad\times\left(
\sum_{\boldsymbol{\beta}\leq \boldsymbol{l}_j} \tbinom{\boldsymbol{l}_j}{\boldsymbol{\beta}}
\left((\varepsilon+|\nabla u_\varepsilon|^2)^{\frac{|\boldsymbol{\beta}|-1}{2}}D^{\boldsymbol{\beta}}(\nabla u_\varepsilon),(\varepsilon+|\nabla u_\varepsilon|^2)^{\frac{|\boldsymbol{l}_j|-|\boldsymbol{\beta}|-1}{2}}D^{\boldsymbol{l}_j-\boldsymbol{\beta}}(\nabla u_\varepsilon)\right)\right)^{k_j}.
\end{split}
\end{equation}
In particular, using the previous equation, we can estimate \eqref{After_c_1} as follows
\begin{equation}\label{stima_F_M}
\begin{split}
&(\varepsilon+|\nabla u_\varepsilon|^2)^\frac{2-p+|\boldsymbol{\gamma}|}{2}\left|D^{\boldsymbol{\gamma}}\left((\varepsilon+|\nabla u_\varepsilon|^2)^\frac{p-2}{2}\right)\right|\\
&\quad\leq\sum_{r=1}^{|\gamma|}c_1 \sum_{p(\boldsymbol{\gamma},r)}(\boldsymbol{\gamma}!) \prod_{j=1}^{|\boldsymbol{\gamma}|}\frac{1}{c_2(j)}C(|\boldsymbol{l}_j|)\\
&\qquad \times
\sum_{\boldsymbol{\beta}\leq \boldsymbol{l}_j} \tbinom{\boldsymbol{l}_j}{\boldsymbol{\beta}}^{k_j}
\left|(\varepsilon+|\nabla u_\varepsilon|^2)^{\frac{|\boldsymbol{\beta}|-1}{2}}D^{\boldsymbol{\beta}}(\nabla u_\varepsilon)\right|^{k_j}\left|(\varepsilon+|\nabla u_\varepsilon|^2)^{\frac{|\boldsymbol{l}_j|-|\boldsymbol{\beta}|-1}{2}}D^{\boldsymbol{l}_j-\boldsymbol{\beta}}(\nabla u_\varepsilon)\right|^{k_j},
\end{split}
\end{equation}
where $C(|\boldsymbol{l}_j|)=\max\{1,|\boldsymbol{l}_j|^{k_j-1}\}$.\\
We now deal with the first term on the right hand side of \eqref{stima_N_LOT}. Using \eqref{pEq}, \eqref{Acca} and \eqref{m_Der} we obtain
\begin{equation}\label{m_Der_e_pEq}
\begin{split}
&\Big\|(\varepsilon+|\nabla u_\varepsilon|^2)^\frac{m-2}{2} D^{\boldsymbol{\nu}}(\Delta u_\varepsilon)\varphi\Big\|_{L^q(B_{2R})}\\
%
&\quad\leq \Big\|\varphi D^{\boldsymbol{\nu}}f(\varepsilon+|\nabla u_\varepsilon|^2)^\frac{m-p}{2} \Big\|_{L^q(B_{2R})}\\
&\quad\qquad+ |p-2| \big\| (\varepsilon+|\nabla u_\varepsilon|^2)^\frac{m-p}{2} D^{\boldsymbol{\nu}} [ (\varepsilon+|\nabla u_\varepsilon|^2)^\frac{p-4}{2} (D^2u_\varepsilon \nabla u_\varepsilon,\nabla u_\varepsilon)]\varphi\big\|_{L^q(B_{2R})}\\
&\quad\qquad + \sum_{0<\boldsymbol{\gamma}\leq \boldsymbol{\nu}} \tbinom{\boldsymbol{\nu}}{\boldsymbol{\gamma}}\Big\|\Big (  (\varepsilon+|\nabla u_\varepsilon|^2)^\frac{2-p+|\boldsymbol{\gamma}|}{2}D^{\boldsymbol{\gamma}}\Big((\varepsilon+|\nabla u_\varepsilon|^2)^\frac{p-2}{2}\Big) \Big )\\
&\qquad\qquad\qquad \times(\varepsilon+|\nabla u_\varepsilon|^2)^\frac{m-2-|\boldsymbol{\gamma}|}{2}D^{\boldsymbol{\nu}-\boldsymbol{\gamma}}(\Delta u_\varepsilon)\varphi\Big\|_{L^q(B_{2R})}\\
&\quad=: I_1 + |p-2|I_2 + \sum_{0<\boldsymbol{\gamma}\leq \boldsymbol{\nu}} \tbinom{\boldsymbol{\nu}}{\boldsymbol{\gamma}} I_3,
\end{split}
\end{equation}
where we used the standard triangular inequality.
Let us estimate the terms $I_1$, $I_2$, $I_3$ separately. We start with $I_3$, using the H\"older inequality with exponents $(q_1,q_2)$, $1/q_1+1/q_2=1/q$; we have
\begin{equation}\label{I_3}
\begin{split}
    I_3 &\leq \Big\| (\varepsilon+|\nabla u_\varepsilon|^2)^\frac{2-p+|\boldsymbol{\gamma}|}{2}D^{\boldsymbol{\gamma}}\left((\varepsilon+|\nabla u_\varepsilon|^2)^\frac{p-2}{2}\right) \Big\|_{L^{q_1}(B_{2R})} \\
    &\quad\qquad \times\Big \| (\varepsilon+|\nabla u_\varepsilon|^2)^\frac{m-2-|\boldsymbol{\gamma}|}{2}D^{\boldsymbol{\nu}-\boldsymbol{\gamma}}(\Delta u_\varepsilon)\Big\|_{L^{q_2}(B_{2R})}.
\end{split}
\end{equation}
Using \eqref{stima_F_M} and iterated H\"older inequalities we can estimate the first term on the right hand side of the previous equation, namely
\begin{equation}\label{Faa_F}
\begin{split}
&\Big\| (\varepsilon+|\nabla u_\varepsilon|^2)^\frac{2-p+|\boldsymbol{\gamma}|}{2}D^{\boldsymbol{\gamma}}\left((\varepsilon+|\nabla u_\varepsilon|^2)^\frac{p-2}{2}\right) \Big\|_{L^{q_1}(B_{2R})}\\  
&\ \leq \sum_{r=1}^{|\gamma|}c_1 \sum_{p(\boldsymbol{\gamma},r)}(\boldsymbol{\gamma}!) \Big\|\prod_{j=1}^{|\boldsymbol{\gamma}|}\frac{1}{c_2(j)}C(|\boldsymbol{l}_j|)\sum_{\boldsymbol{\beta}\leq \boldsymbol{l}_j} \tbinom{\boldsymbol{l}_j}{\boldsymbol{\beta}}^{k_j}\\
&\quad \times
\left|(\varepsilon+|\nabla u_\varepsilon|^2)^{\frac{|\boldsymbol{\beta}|-1}{2}}D^{\boldsymbol{\beta}}(\nabla u_\varepsilon)\right|^{k_j}\left|(\varepsilon+|\nabla u_\varepsilon|^2)^{\frac{|\boldsymbol{l}_j|-|\boldsymbol{\beta}|-1}{2}}D^{\boldsymbol{l}_j-\boldsymbol{\beta}}(\nabla u_\varepsilon)\right|^{k_j}\Big\|_{L^{q_1}(B_{2R})}\\
&\ \leq \sum_{r=1}^{|\gamma|}c_1 \sum_{p(\boldsymbol{\gamma},r)}(\boldsymbol{\gamma}!) \prod_{j=1}^{|\boldsymbol{\gamma}|}\frac{1}{c_2(j)}C(|\boldsymbol{l}_j|)\sum_{\boldsymbol{\beta}\leq \boldsymbol{l}_j} \tbinom{\boldsymbol{l}_j}{\boldsymbol{\beta}}^{k_j}\\
&\quad \times
\Big\|\left|(\varepsilon+|\nabla u_\varepsilon|^2)^{\frac{|\boldsymbol{\beta}|-1}{2}}D^{\boldsymbol{\beta}}(\nabla u_\varepsilon)\right|^{k_j}\left|(\varepsilon+|\nabla u_\varepsilon|^2)^{\frac{|\boldsymbol{l}_j|-|\boldsymbol{\beta}|-1}{2}}D^{\boldsymbol{l}_j-\boldsymbol{\beta}}(\nabla u_\varepsilon)\right|^{k_j}\Big\|_{L^{p_j}(B_{2R})}\\
&\ \leq \sum_{r=1}^{|\gamma|}c_1 \sum_{p(\boldsymbol{\gamma},r)}(\boldsymbol{\gamma}!) \prod_{j=1}^{|\boldsymbol{\gamma}|}\frac{1}{c_2(j)}C(|\boldsymbol{l}_j|)\sum_{\boldsymbol{\beta}\leq \boldsymbol{l}_j} \tbinom{\boldsymbol{l}_j}{\boldsymbol{\beta}}^{k_j}\\
&\quad \times
\Big\|(\varepsilon+|\nabla u_\varepsilon|^2)^{\frac{|\boldsymbol{\beta}|-1}{2}}D^{\boldsymbol{\beta}}(\nabla u_\varepsilon)\Big\|^{k_j}_{L^{k_jp_j^1}(B_{2R})} \Big\|(\varepsilon+|\nabla u_\varepsilon|^2)^{\frac{|\boldsymbol{l}_j|-|\boldsymbol{\beta}|-1}{2}}D^{\boldsymbol{l}_j-\boldsymbol{\beta}}(\nabla u_\varepsilon)\Big\|^{k_j}_{L^{k_jp_j^2}(B_{2R})},
\end{split}
\end{equation}
where $\{p_j\}_{j=1}^{|\boldsymbol{\gamma}|}$ and $(p_j^1,p_j^2)$ are such that
   $ \sum_{j=1}^{|\boldsymbol{\gamma}|} \frac{1}{p_j}=\frac{1}{q_1}$  and  $\frac{1}{p_j^1}+\frac{1}{p_j^2}=\frac{1}{p_j}.$ 
This allows us to conclude that the term $I_3$ consists entirely of products of lower-order terms. Thus, by the base and the inductive steps we obtain:
\begin{equation}\label{STIMA_I3}
    I_3 \leq C_3(n,m,q,\|f\|_{S_{loc}(\Omega)},p,\|\nabla u\|_{L^\infty_{loc}(\Omega)}),
\end{equation}
where $C_3(n,m,q,\|f\|_{S_{loc}(\Omega)},p,\|\nabla u\|_{L^\infty_{loc}(\Omega)})$ is a positive constant independent of $\varepsilon$.\\
Concerning $I_1$, by assumptions on $f$, we have
\begin{equation}\label{I_1}
\begin{split}
I_1 &= \left\|D^{\boldsymbol{\nu}}f (\varepsilon+|\nabla u_\varepsilon|^2)^\frac{m-p}{2}\varphi \right\|_{L^q(B_{2R})}\leq C(m,p,\|\nabla u\|_{L^\infty_{loc}(\Omega)} \|D^{\boldsymbol{\nu}}f \|_{L^q(B_{2R})}).
\end{split}
\end{equation}
The remaining term is $I_2$. We follow the same approach used for $I_3$, i.e. applying both the Leibniz and the multivariate version of the Faà di Bruno formulas. First of all, applying the Leibniz formula we obtain:
\begin{equation}\label{L_2}
    \begin{split}
        &D^{\boldsymbol{\nu}} [ (\varepsilon+|\nabla u_\varepsilon|^2)^\frac{p-4}{2} (D^2u_\varepsilon \nabla u_\varepsilon,\nabla u_\varepsilon)]= (\varepsilon+|\nabla u_\varepsilon|^2)^\frac{p-4}{2} D^{\boldsymbol{\nu}}(D^2u_\varepsilon \nabla u_\varepsilon,\nabla u_\varepsilon)\\
        &\quad\qquad+\sum_{0<\boldsymbol{\gamma}\leq \boldsymbol{\nu}} \tbinom{\boldsymbol{\nu}}{\boldsymbol{\gamma}} D^{\boldsymbol{\gamma}}\left((\varepsilon+|\nabla u_\varepsilon|^2)^\frac{p-4}{2}\right) D^{\boldsymbol{\nu}-\boldsymbol{\gamma}}(D^2u_\varepsilon\nabla u_\varepsilon, \nabla u_\varepsilon) = : J_1+J_2.
    \end{split}
\end{equation}
The second term of \eqref{L_2}, again using the Leibniz formula, can be rewritten as follows
\begin{equation}\label{J_2}
\begin{split}
J_2&=\sum_{0<\boldsymbol{\gamma}\leq \boldsymbol{\nu}} \tbinom{\boldsymbol{\nu}}{\boldsymbol{\gamma}} D^{\boldsymbol{\gamma}}\left((\varepsilon+|\nabla u_\varepsilon|^2)^\frac{p-4}{2}\right)  \sum_{\boldsymbol{\sigma}\leq \boldsymbol{\nu}-\boldsymbol{\gamma}} \tbinom{\boldsymbol{\nu}-\boldsymbol{\gamma}}{\boldsymbol{\sigma}}\left(D^{\boldsymbol{\sigma}}(D^2u_\varepsilon\nabla u_\varepsilon), D^{\boldsymbol{\nu}-\boldsymbol{\gamma}- \boldsymbol{\sigma}}(\nabla u_\varepsilon)\right)\\
& = \sum_{0<\boldsymbol{\gamma}\leq \boldsymbol{\nu}} \tbinom{\boldsymbol{\nu}}{\boldsymbol{\gamma}} D^{\boldsymbol{\gamma}}\left((\varepsilon+|\nabla u_\varepsilon|^2)^\frac{p-4}{2}\right) \sum_{\boldsymbol{\sigma}\leq \boldsymbol{\nu}-\boldsymbol{\gamma}} \sum_{\boldsymbol{\mu}\leq \boldsymbol{\sigma}} \mathcal{C}    \left(D^{\boldsymbol{\mu}}(D^2u_\varepsilon)D^{\boldsymbol{\sigma}-\boldsymbol{\mu}}(\nabla u_\varepsilon), D^{\boldsymbol{\nu}-\boldsymbol{\gamma}- \boldsymbol{\sigma}}(\nabla u_\varepsilon)\right),
\end{split}
\end{equation}
where $\mathcal{C} := \tbinom{\boldsymbol{\nu}-\boldsymbol{\gamma}}{\boldsymbol{\sigma}}\tbinom{\boldsymbol{\sigma}}{\boldsymbol{\mu}}$.We now deal with $J_1$,
\begin{equation}\label{Leading}
    \begin{split}
        J_1&= (\varepsilon+|\nabla u_\varepsilon|^2)^\frac{p-4}{2} \left(D^{\boldsymbol{\nu}}(D^2u_\varepsilon \nabla u_\varepsilon),\nabla u_\varepsilon\right)\\
        &\qquad+(\varepsilon+|\nabla u_\varepsilon|^2)^\frac{p-4}{2}\sum_{\boldsymbol{\sigma}<\boldsymbol{\nu}} \tbinom{\boldsymbol{\nu}}{\boldsymbol{\sigma}}  \left(D^{\boldsymbol{\sigma}}(D^2u_\varepsilon\nabla u_\varepsilon), D^{\boldsymbol{\nu}-\boldsymbol{\sigma}}(\nabla u_\varepsilon)\right)\\
        &=(\varepsilon+|\nabla u_\varepsilon|^2)^\frac{p-4}{2} \left(D^{\boldsymbol{\nu}}(D^2u_\varepsilon) \nabla u_\varepsilon,\nabla u_\varepsilon\right)\\
        &\qquad+(\varepsilon+|\nabla u_\varepsilon|^2)^\frac{p-4}{2}\sum_{\boldsymbol{\mu}<\boldsymbol{\nu}} \tbinom{\boldsymbol{\nu}}{\boldsymbol{\mu}}  \left(D^{\boldsymbol{\mu}}(D^2u_\varepsilon)D^{\boldsymbol{\nu}-\boldsymbol{\mu}}(\nabla u_\varepsilon), \nabla u_\varepsilon\right)\\
        &\qquad+ (\varepsilon+|\nabla u_\varepsilon|^2)^\frac{p-4}{2}\sum_{\boldsymbol{\sigma} < \boldsymbol{\nu}} \sum_{\boldsymbol{\mu} \leq \boldsymbol{\sigma}} \tbinom{\boldsymbol{\nu}}{\boldsymbol{\sigma}} \tbinom{\boldsymbol{\sigma}}{\boldsymbol{\mu}} \left(D^{\boldsymbol{\mu}}(D^2u_\varepsilon)D^{\boldsymbol{\sigma}-\boldsymbol{\mu}}(\nabla u_\varepsilon), D^{\boldsymbol{\nu}- \boldsymbol{\sigma}}(\nabla u_\varepsilon)\right).
    \end{split}
\end{equation}
Therefore we bound $I_2$ as follows
\begin{equation}\label{I_2_q}
\begin{split}
    I_2&= \Big\| (\varepsilon+|\nabla u_\varepsilon|^2)^\frac{m-p}{2} D^{\boldsymbol{\nu}} [ (\varepsilon+|\nabla u_\varepsilon|^2)^\frac{p-4}{2} (D^2u_\varepsilon \nabla u_\varepsilon,\nabla u_\varepsilon)]\varphi\Big\|_{L^q(B_{2R})}\\
    &\leq \Big\| (\varepsilon+|\nabla u_\varepsilon|^2)^\frac{m-p}{2} J_1 \varphi\Big\|_{L^q(B_{2R})} + \Big\| (\varepsilon+|\nabla u_\varepsilon|^2)^\frac{m-p}{2} J_2 \varphi\Big\|_{L^q(B_{2R})}
\end{split}
\end{equation}
Using \eqref{Leading}, the first term of the right-hand side of \eqref{I_2_q} becomes
\begin{equation*}
    \begin{split}
        &(\varepsilon+|\nabla u_\varepsilon|^2)^\frac{m-p}{2}J_1\\
        &\quad= \left((\varepsilon+|\nabla u_\varepsilon|^2)^\frac{m-2}{2}D^{\boldsymbol{\nu}}(D^2u_\varepsilon) \frac{\nabla u_\varepsilon}{(\varepsilon+|\nabla u_\varepsilon|^2)^\frac{1}{2}},\frac{\nabla u_\varepsilon}{(\varepsilon+|\nabla u_\varepsilon|^2)^\frac{1}{2}}\right)\\
        &\qquad+\sum_{\boldsymbol{\mu}<\boldsymbol{\nu}} \tbinom{\boldsymbol{\nu}}{\boldsymbol{\mu}}  \left((\varepsilon+|\nabla u_\varepsilon|^2)^\frac{|\boldsymbol{\mu}|}{2}D^{\boldsymbol{\mu}}(D^2u_\varepsilon)(\varepsilon+|\nabla u_\varepsilon|^2)^\frac{m-3-|\boldsymbol{\mu}|}{2}D^{\boldsymbol{\nu}-\boldsymbol{\mu}}(\nabla u_\varepsilon), \right.\\
        &\qquad\qquad\qquad\left. , \frac{\nabla u_\varepsilon}{(\varepsilon+|\nabla u_\varepsilon|^2)^\frac{1}{2}}\right)\\
        &\qquad+\sum_{\boldsymbol{\sigma} < \boldsymbol{\nu}} \sum_{\boldsymbol{\mu} \leq \boldsymbol{\sigma}} \tbinom{\boldsymbol{\nu}}{\boldsymbol{\sigma}} \tbinom{\boldsymbol{\sigma}}{\boldsymbol{\mu}}\left( (\varepsilon+|\nabla u_\varepsilon|^2)^\frac{|\boldsymbol{\mu}|}{2}D^{\boldsymbol{\mu}}(D^2u_\varepsilon)(\varepsilon+|\nabla u_\varepsilon|^2)^\frac{|\boldsymbol{\sigma}|-|\boldsymbol{\mu}|-1}{2}D^{\boldsymbol{\sigma}-\boldsymbol{\mu}}(\nabla u_\varepsilon),\right.\\
        &\qquad\qquad\qquad\left.,(\varepsilon+|\nabla u_\varepsilon|^2)^\frac{m-3-|\boldsymbol{\sigma}|}{2}D^{\boldsymbol{\nu}- \boldsymbol{\sigma}}(\nabla u_\varepsilon)\right).
    \end{split}
\end{equation*}
In particular, using triangular inequality, we get
\begin{equation}\label{LOT_5}
    \begin{split}
        \left\|(\varepsilon+|\nabla u_\varepsilon|^2)^\frac{m-p}{2}J_1\varphi\right\|_{L^q(B_{2R})}
        \leq \left\|(\varepsilon+|\nabla u_\varepsilon|^2)^\frac{m-2}{2}D^{\boldsymbol{\nu}}(D^2u_\varepsilon)\varphi\right\|_{L^q(B_{2R})} + R^5_{\leq \nu + 1},
    \end{split}
\end{equation}
where,
\begin{equation*}
\begin{split}
    R^5_{\leq \nu + 1} : = &\sum_{\boldsymbol{\mu}<\boldsymbol{\nu}} \tbinom{\boldsymbol{\nu}}{\boldsymbol{\mu}}  \left\|(\varepsilon+|\nabla u_\varepsilon|^2)^\frac{|\boldsymbol{\mu}|}{2}|D^{\boldsymbol{\mu}}(D^2u_\varepsilon)|(\varepsilon+|\nabla u_\varepsilon|^2)^\frac{m-3-|\boldsymbol{\mu}|}{2}|D^{\boldsymbol{\nu}-\boldsymbol{\mu}}(\nabla u_\varepsilon)|\right\|_{L^q(B_{2R})}\\
        &\quad+\sum_{\boldsymbol{\sigma} < \boldsymbol{\nu}} \sum_{\boldsymbol{\mu} \leq \boldsymbol{\sigma}} \tbinom{\boldsymbol{\nu}}{\boldsymbol{\sigma}} \tbinom{\boldsymbol{\sigma}}{\boldsymbol{\mu}}\\
        &\ \qquad\times\left\| (\varepsilon+|\nabla u_\varepsilon|^2)^\frac{|\boldsymbol{\mu}|}{2}|D^{\boldsymbol{\mu}}(D^2u_\varepsilon)|(\varepsilon+|\nabla u_\varepsilon|^2)^\frac{|\boldsymbol{\sigma}|-|\boldsymbol{\mu}|-1}{2}|D^{\boldsymbol{\sigma}-\boldsymbol{\mu}}(\nabla u_\varepsilon)|\right.\\
        &\qquad\quad\qquad \left. \times (\varepsilon+|\nabla u_\varepsilon|^2)^\frac{m-3-|\boldsymbol{\sigma}|}{2}|D^{\boldsymbol{\nu}- \boldsymbol{\sigma}}(\nabla u_\varepsilon)|\right\|_{L^q(B_{2R})}.
\end{split}
\end{equation*}
Using H\"older inequalities with exponents $1/{p_1}+1/{p_2}=1/q$ and $1/{r_1}+1/{r_2}+1/{r_3}=1/q$, we have
\begin{equation}\label{stima_Lot_5}
    \begin{split}
        R^5_{\leq \nu + 1} &\leq \sum_{\boldsymbol{\mu}<\boldsymbol{\nu}} \tbinom{\boldsymbol{\nu}}{\boldsymbol{\mu}}  \left\|(\varepsilon+|\nabla u_\varepsilon|^2)^\frac{|\boldsymbol{\mu}|}{2}D^{\boldsymbol{\mu}}(D^2u_\varepsilon)\right\|_{L^{p_1}(B_{2R})}\\
        &\qquad\qquad \times \left\|(\varepsilon+|\nabla u_\varepsilon|^2)^\frac{m-3-|\boldsymbol{\mu}|}{2}D^{\boldsymbol{\nu}-\boldsymbol{\mu}}(\nabla u_\varepsilon)\right\|_{L^{p_2}(B_{2R})}\\
        &\qquad+\sum_{\boldsymbol{\sigma} < \boldsymbol{\nu}} \sum_{\boldsymbol{\mu} \leq \boldsymbol{\sigma}} \tbinom{\boldsymbol{\nu}}{\boldsymbol{\sigma}} \tbinom{\boldsymbol{\sigma}}{\boldsymbol{\mu}}\left\|(\varepsilon+|\nabla u_\varepsilon|^2)^\frac{|\boldsymbol{\mu}|}{2}D^{\boldsymbol{\mu}}(D^2u_\varepsilon)\right \|_{L^{r_1}(B_{2R})}\\
        &\ \qquad\qquad\times\left\|(\varepsilon+|\nabla u_\varepsilon|^2)^\frac{|\boldsymbol{\sigma}|-|\boldsymbol{\mu}|-1}{2}D^{\boldsymbol{\sigma}-\boldsymbol{\mu}}(\nabla u_\varepsilon)\right\|_{L^{r_2}(B_{2R})}\\
        &\qquad\qquad\qquad \times\left\|(\varepsilon+|\nabla u_\varepsilon|^2)^\frac{m-3-|\boldsymbol{\sigma}|}{2}D^{\boldsymbol{\nu}- \boldsymbol{\sigma}}(\nabla u_\varepsilon)\right\|_{L^{r_3}(B_{2R})}.
    \end{split}
\end{equation}
Combining the base case with the inductive step \eqref{Inductive_Step}, we obtain
\begin{equation}\label{term1_I2}
\begin{split}
\left\|(\varepsilon+|\nabla u_\varepsilon|^2)^\frac{m-p}{2}J_1\varphi\right\|_{L^q(B_{2R})} &\leq \left\|(\varepsilon+|\nabla u_\varepsilon|^2)^\frac{m-2}{2}D^{\boldsymbol{\nu}}(D^2u_\varepsilon)\varphi\right\|_{L^q(B_{2R})}+ C_1,
\end{split}
\end{equation}
where $C_1=C_1(n,m,q,R,\|f\|_{S_{loc}(\Omega)},p,\|\nabla u\|_{L^\infty_{loc}(\Omega)})$ is a positive constant.\\
We now estimate the second term on the right-hand side of \eqref{I_2_q}. By \eqref{J_2}, recalling that $|\boldsymbol{\nu}|=m-2$, we get

\begin{equation*}
\begin{split}
&(\varepsilon+|\nabla u_\varepsilon|^2)^\frac{m-p}{2}J_2\\
&\  = \sum_{0<\boldsymbol{\gamma}\leq \boldsymbol{\nu}} \tbinom{\boldsymbol{\nu}}{\boldsymbol{\gamma}}\left( (\varepsilon+|\nabla u_\varepsilon|^2)^\frac{4-p+|\boldsymbol{\gamma}|}{2}D^{\boldsymbol{\gamma}}\left((\varepsilon+|\nabla u_\varepsilon|^2)^\frac{p-4}{2}\right) \right)\\
&\qquad \times \sum_{\boldsymbol{\sigma}\leq \boldsymbol{\nu}-\boldsymbol{\gamma}} \sum_{\boldsymbol{\mu}\leq \boldsymbol{\sigma}} \mathcal{C}  \Big((\varepsilon+|\nabla u_    \varepsilon|^2)^\frac{|\boldsymbol{\mu}|}{2}D^{\boldsymbol{\mu}}(D^2u_\varepsilon)(\varepsilon+|\nabla u_\varepsilon|^2)^\frac{|\boldsymbol{\sigma}|-|\boldsymbol{\mu}|-1}{2}D^{\boldsymbol{\sigma}-\boldsymbol{\mu}}(\nabla u_\varepsilon),\\
&\qquad\qquad ,(\varepsilon+|\nabla u_\varepsilon|^2)^\frac{|\boldsymbol{\nu}|-|\boldsymbol{\gamma}|-|\boldsymbol{\sigma}|-1}{2}D^{\boldsymbol{\nu}-\boldsymbol{\gamma}- \boldsymbol{\sigma}}(\nabla u_\varepsilon)\Big).
\end{split}
\end{equation*}
Using triangular inequality, we get
\begin{equation}\label{LOT_6}
        \left\|(\varepsilon+|\nabla u_\varepsilon|^2)^\frac{m-p}{2}J_2\varphi\right\|_{L^q(B_{2R})}
        \leq R^6_{\leq \nu + 1},
\end{equation}
where,
\begin{equation*}
\begin{split}
    R^6_{\leq \nu + 1} : = &\sum_{0<\boldsymbol{\gamma}\leq \boldsymbol{\nu}} \tbinom{\boldsymbol{\nu}}{\boldsymbol{\gamma}}\Big\| (\varepsilon+|\nabla u_\varepsilon|^2)^\frac{4-p+|\boldsymbol{\gamma}|}{2}D^{\boldsymbol{\gamma}}\left((\varepsilon+|\nabla u_\varepsilon|^2)^\frac{p-4}{2}\right)\\
&\qquad \times \sum_{\boldsymbol{\sigma}\leq \boldsymbol{\nu}-\boldsymbol{\gamma}} \sum_{\boldsymbol{\mu}\leq \boldsymbol{\sigma}} \mathcal{C} \Big((\varepsilon+|\nabla u_    \varepsilon|^2)^\frac{|\boldsymbol{\mu}|}{2}D^{\boldsymbol{\mu}}(D^2u_\varepsilon)(\varepsilon+|\nabla u_\varepsilon|^2)^\frac{|\boldsymbol{\sigma}|-|\boldsymbol{\mu}|-1}{2}D^{\boldsymbol{\sigma}-\boldsymbol{\mu}}(\nabla u_\varepsilon),\\
&\qquad\qquad ,(\varepsilon+|\nabla u_\varepsilon|^2)^\frac{|\boldsymbol{\nu}|-|\boldsymbol{\gamma}|-|\boldsymbol{\sigma}|-1}{2}D^{\boldsymbol{\nu}-\boldsymbol{\gamma}- \boldsymbol{\sigma}}(\nabla u_\varepsilon)\Big)\Big\|_{L^{q}(B_{2R})}.
\end{split}
\end{equation*}
Applying H\"older inequalities with exponents $1/q_1+1/q_2=1/q$ and $1/t_1+1/t_2+1/t_3=1/q_2$, we obtain
\begin{equation}\label{stima_Lot_6}
    \begin{split}
         R^6_{\leq \nu + 1} &\leq \sum_{0<\boldsymbol{\gamma}\leq \boldsymbol{\nu}} \tbinom{\boldsymbol{\nu}}{\boldsymbol{\gamma}}\left\| (\varepsilon+|\nabla u_\varepsilon|^2)^\frac{4-p+|\boldsymbol{\gamma}|}{2}D^{\boldsymbol{\gamma}}\left((\varepsilon+|\nabla u_\varepsilon|^2)^\frac{p-4}{2}\right) \right\|_{L^{q_1}(B_{2R})}\\
&\qquad \times \sum_{\boldsymbol{\sigma}\leq \boldsymbol{\nu}-\boldsymbol{\gamma}} \sum_{\boldsymbol{\mu}\leq \boldsymbol{\sigma}} \mathcal{C}  \Big\|(\varepsilon+|\nabla u_    \varepsilon|^2)^\frac{|\boldsymbol{\mu}|}{2}D^{\boldsymbol{\mu}}(D^2u_\varepsilon)\Big\|_{L^{t_1}(B_{2R})}\\
&\qquad\qquad\times\Big\|(\varepsilon+|\nabla u_\varepsilon|^2)^\frac{|\boldsymbol{\sigma}|-|\boldsymbol{\mu}|-1}{2}D^{\boldsymbol{\sigma}-\boldsymbol{\mu}}(\nabla u_\varepsilon)\Big\|_{L^{t_2}(B_{2R})}\\
&\qquad\qquad \times\Big\|(\varepsilon+|\nabla u_\varepsilon|^2)^\frac{|\boldsymbol{\nu}|-|\boldsymbol{\gamma}|-|\boldsymbol{\sigma}|-1}{2}D^{\boldsymbol{\nu}-\boldsymbol{\gamma}- \boldsymbol{\sigma}}(\nabla u_\varepsilon)\Big\|_{L^{t_3}(B_{2R})}.
    \end{split}
\end{equation}
Using the same computation as in \eqref{stima_F} and \eqref{Faa_F}, we obtain
\begin{equation}\label{stima_F2}
\begin{split}
&\Big\| (\varepsilon+|\nabla u_\varepsilon|^2)^\frac{4-p+|\boldsymbol{\gamma}|}{2}D^{\boldsymbol{\gamma}}\left((\varepsilon+|\nabla u_\varepsilon|^2)^\frac{p-4}{2}\right) \Big\|_{L^{q_1}(B_{2R})}\\ 
&\ \leq \sum_{r=1}^{|\gamma|}c_2 \sum_{p(\boldsymbol{\gamma},r)}(\boldsymbol{\gamma}!) \prod_{j=1}^{|\boldsymbol{\gamma}|}\frac{1}{c_2(j)}C(|\boldsymbol{l}_j|)\sum_{\boldsymbol{\beta}\leq \boldsymbol{l}_j} \tbinom{\boldsymbol{l}_j}{\boldsymbol{\beta}}^{k_j}\\
&\qquad \times
\Big\|(\varepsilon+|\nabla u_\varepsilon|^2)^{\frac{|\boldsymbol{\beta}|-1}{2}}D^{\boldsymbol{\beta}}(\nabla u_\varepsilon)\Big\|^{k_j}_{L^{k_jp_j^1}(B_{2R})}\\
&\qquad \times\Big\|(\varepsilon+|\nabla u_\varepsilon|^2)^{\frac{|\boldsymbol{l}_j|-|\boldsymbol{\beta}|-1}{2}}D^{\boldsymbol{l}_j-\boldsymbol{\beta}}(\nabla u_\varepsilon)\Big\|^{k_j}_{L^{k_jp_j^2}(B_{2R})},
\end{split}
\end{equation}
where $c_2=\prod_{h=1}^r\frac{p-2-2h}{2}$.\\
Moreover, using the base step and the inductive step \eqref{Inductive_Step} of the induction, we obtain:
\begin{equation}\label{estimate_J2}
    \Big\|(\varepsilon+|\nabla u_\varepsilon|^2)^\frac{m-p}{2}J_2\varphi\Big\|_{L^q(B_{2R})} \leq C_2,
\end{equation}
where $C_2=C_2(n,m,q,R,\|f\|_{S_{loc}(\Omega)},p,\|\nabla u\|_{L^\infty_{loc}(\Omega)})$ is a positive constant independent of $\varepsilon$.
Combining \eqref{term1_I2} and \eqref{estimate_J2}, we can rewrite \eqref{I_2_q} as follows
\begin{equation}\label{STIMA_I2}
\begin{split}
    I_2 &\leq \left\|(\varepsilon+|\nabla u_\varepsilon|^2)^\frac{m-2}{2}D^{\boldsymbol{\nu}}(D^2u_\varepsilon)\varphi\right\|_{L^q(B_{2R})}\\
    &\qquad\qquad+ C(n,m,q,R,\|f\|_{S_{loc}(\Omega)},p,\|\nabla u\|_{L^\infty_{loc}(\Omega)}).
\end{split}
\end{equation}
Summing up, using \eqref{stima_N_LOT}, \eqref{m_Der_e_pEq}, \eqref{STIMA_I3}, \eqref{I_1} and \eqref{STIMA_I2}, we arrive to
\begin{equation}\label{stima_Sum}
\begin{split}
    &(1-|p-2|C(n,q))\left\|(\varepsilon+|\nabla u_\varepsilon|^2)^\frac{m-2}{2}D^{\boldsymbol{\nu}}(D^2u_\varepsilon)\varphi\right\|_{L^q(B_{2R})}\\
    &\qquad\qquad \leq \mathcal{C}(n,m,q,\|D^{\boldsymbol{\nu}}f\|_{L^q_{loc}(\Omega)},\|f\|_{S_{loc}(\Omega)},R,p,\|\nabla u\|_{L^\infty_{loc}(\Omega)}).
\end{split}
\end{equation}
For $|p-2|$ small enough and using the properties of $\varphi$, namely $\varphi\equiv1$ on $B_R$ we have:
\begin{equation*}
    \left\|(\varepsilon+|\nabla u_\varepsilon|^2)^\frac{m-2}{2}D^{\boldsymbol{\nu}}(D^2u_\varepsilon)\right\|_{L^q(B_{R})} \leq \left\|(\varepsilon+|\nabla u_\varepsilon|^2)^\frac{m-2}{2}D^{\boldsymbol{\nu}}(D^2u_\varepsilon)\varphi\right\|_{L^q(B_{2R})},
\end{equation*}
and this implies that,
$   \|(\varepsilon+|\nabla u_\varepsilon|^2)^\frac{m-2}{2}D^{\boldsymbol{\nu}}(D^2u_\varepsilon)\|_{L^q(B_{R})} \leq \mathcal{C},
$ for a a positive constant 
 $\mathcal{C}=\mathcal{C}(n,m,q,\|D^{\boldsymbol{\nu}}f\|_{L^q_{loc}(\Omega)},\|f\|_{S_{loc}(\Omega)},R,p,\|\nabla u\|_{L^\infty_{loc}(\Omega)})$ independent of $\varepsilon$.
\end{proof}
We are in position to prove Theorem \ref{teoprinc}.
\begin{proof}[Proof of Theorem \ref{teoprinc}]
Let $B_{2R}\subset\subset \Omega$ and let $u_\varepsilon$ be a solution of \eqref{eq:problregol}. Assume for the moment that $f \in C^{m-2,\beta'}_{loc}(\Omega).$
First of all, we claim that
\begin{equation}\label{CONV}
    \left\|(\varepsilon+|\nabla u_\varepsilon|^2)^\frac{m-2}{2}\nabla u_\varepsilon\right\|_{W^{m-1,q}(B_R)} \leq C,
\end{equation}
where $C>0$ independent of $\varepsilon$. To prove the claim, let $\boldsymbol{\alpha}$ be a multi-index such that $|\boldsymbol{\alpha}|= m-1$ and exploiting the Leibniz rule, we get:
\begin{equation}\label{L_conv}
\begin{split}
    &D^{\boldsymbol{\alpha}} \left((\varepsilon+|\nabla u_\varepsilon|^2)^\frac{m-2}{2}\nabla u_\varepsilon\right)= \sum_{\boldsymbol{\gamma}\leq \boldsymbol{\alpha}} \tbinom{\boldsymbol{\alpha}}{\boldsymbol{\gamma}} D^{\boldsymbol{\gamma}}\left((\varepsilon+|\nabla u_\varepsilon|^2)^\frac{m-2}{2}\right) D^{\boldsymbol{\alpha}-\boldsymbol{\gamma}}\left(\nabla u_\varepsilon\right)\\
    &\qquad=(\varepsilon+|\nabla u_\varepsilon|^2)^\frac{m-2}{2}D^{\boldsymbol{\alpha}}(\nabla u_\varepsilon) + D^{\boldsymbol{\alpha}}\left((\varepsilon+|\nabla u_\varepsilon|^2)^\frac{m-2}{2}\right) \nabla u_\varepsilon\\
    &\qquad\quad+ \sum_{0<\boldsymbol{\gamma}< \boldsymbol{\alpha}} \tbinom{\boldsymbol{\alpha}}{\boldsymbol{\gamma}} D^{\boldsymbol{\gamma}}\left((\varepsilon+|\nabla u_\varepsilon|^2)^\frac{m-2}{2}\right) D^{\boldsymbol{\alpha}-\boldsymbol{\gamma}}\left(\nabla u_\varepsilon\right) = : K_1+K_2+K_3.
\end{split}
\end{equation}
We aim now to show that $K_1$, $K_2$ and $K_3$ are bounded in $L^q(B_R)$. Since $\boldsymbol{\alpha}$ is a multi-index of order $m-1$, by Lemma \ref{sisso}, see \eqref{RESULT}, we can estimate $K_1$ as follows
\begin{equation}\label{K_1}
\big\|K_1\big\|_{L^q(B_R)} = \Big\|(\varepsilon+|\nabla u_\varepsilon|^2)^\frac{m-2}{2}D^{\boldsymbol{\alpha}}(\nabla u_\varepsilon)\Big\|_{L^q(B_R)} \leq C,
\end{equation}
where $C$ is a positive constant independent of $\varepsilon$. To estimate the remaining two terms in \eqref{L_conv}, we first apply the Faà di Bruno formula. In particular, for $K_2$ we have
\begin{equation}\label{K_2_P}
\begin{split}
    &D^{\boldsymbol{\alpha}}\left((\varepsilon+|\nabla u_\varepsilon|^2)^\frac{m-2}{2}\right)\\
    &\qquad= \sum_{r=1}^{|\boldsymbol{\alpha}|} c_1 (\varepsilon+|\nabla u_\varepsilon|^2)^{\frac{m-2}{2} -r} \sum_{s=1}^{|\boldsymbol{\alpha}|} \sum_{p_s(\boldsymbol{\alpha},{r})} \boldsymbol{\alpha}! \prod_{j=1}^s \frac{\left(D^{\boldsymbol{l}_j}(|\nabla u_\varepsilon|^2)\right)^{k_j}}{k_j!(\boldsymbol{l}_j!)^{k_j}}
\end{split}
\end{equation}
where $c_1= \prod_{h=1}^r \frac{m-2r}{2}$ and
$$p_s(\boldsymbol{\alpha},{r})= \{({k}_1,...,{k}_s;\boldsymbol{l}_1,...,\boldsymbol{l}_s):  0\prec \boldsymbol{l}_1\prec\cdots \prec\boldsymbol{l}_s;k_i>0;\sum_{i=1}^s k_i ={r}, \sum_{i=1}^{s} k_i \boldsymbol{l}_i = \boldsymbol{\alpha}\}.$$
In order to identify the leading term, we separate the cases $r=s=1$ from the remaining ones. Thus, \eqref{K_2_P} becomes
\begin{equation}\label{K_2_s}
\begin{split}
    &D^{\boldsymbol{\alpha}}\left((\varepsilon+|\nabla u_\varepsilon|^2)^\frac{m-2}{2}\right)
  = \sum_{r=1}^{|\boldsymbol{\alpha}|} c_1 (\varepsilon+|\nabla u_\varepsilon|^2)^{\frac{m-2}{2} -r} \sum_{p_1(\boldsymbol{\alpha},{r})} \boldsymbol{\alpha}! \frac{\left(D^{\boldsymbol{l}_1}(|\nabla u_\varepsilon|^2)\right)^{k_1}}{k_1!(\boldsymbol{l}_1!)^{k_1}}\\
    &\qquad\quad + \sum_{r=1}^{|\boldsymbol{\alpha}|} c_1 (\varepsilon+|\nabla u_\varepsilon|^2)^{\frac{m-2}{2} -r} \sum_{s=2}^{|\boldsymbol{\alpha}|} \sum_{p_s(\boldsymbol{\alpha},{r})} \boldsymbol{\alpha}! \prod_{j=1}^s \frac{\left(D^{\boldsymbol{l}_j}(|\nabla u_\varepsilon|^2)\right)^{k_j}}{k_j!(\boldsymbol{l}_j!)^{k_j}}\\
    &\qquad = \tfrac{m-2}{2} (\varepsilon+|\nabla u_\varepsilon|^2)^{\frac{m-4}{2}}D^{\boldsymbol{\alpha}}(|\nabla u_\varepsilon|^2) \\
    &\qquad\quad+ \sum_{r=2}^{|\boldsymbol{\alpha}|} c_1 (\varepsilon+|\nabla u_\varepsilon|^2)^{\frac{m-2}{2} -r} \sum_{p_1(\boldsymbol{\alpha},{r})} \boldsymbol{\alpha}! \frac{\left(D^{\boldsymbol{l}_1}(|\nabla u_\varepsilon|^2)\right)^{k_1}}{k_1!(\boldsymbol{l}_1!)^{k_1}}\\
    &\qquad\quad + \sum_{r=1}^{|\boldsymbol{\alpha}|} c_1 (\varepsilon+|\nabla u_\varepsilon|^2)^{\frac{m-2}{2} -r} \sum_{s=2}^{|\boldsymbol{\alpha}|} \sum_{p_s(\boldsymbol{\alpha},{r})} \boldsymbol{\alpha}! \prod_{j=1}^s \frac{\left(D^{\boldsymbol{l}_j}(|\nabla u_\varepsilon|^2)\right)^{k_j}}{k_j!(\boldsymbol{l}_j!)^{k_j}},
\end{split}
\end{equation}
where we have used 
    $p_1(\boldsymbol{\alpha},1)=(1,\boldsymbol{\alpha}).$
Moreover, using again the Leibniz rule, we have
\begin{equation}\label{Leibiniz_norma}
\begin{split}
    D^{\boldsymbol{\alpha}}(|\nabla u_\varepsilon|^2) &= \sum_{\boldsymbol{\mu}\leq \boldsymbol{\alpha}} \tbinom{\boldsymbol{\alpha}}{\boldsymbol{\mu}} \left(D^{\boldsymbol{\mu}}(\nabla u_\varepsilon),D^{\boldsymbol{\alpha}-\boldsymbol{\mu}}(\nabla u_\varepsilon)\right)\\
    &=  2\left(D^{\boldsymbol{\alpha}}(\nabla u_\varepsilon),\nabla u_\varepsilon\right)+\sum_{0<\boldsymbol{\mu}< \boldsymbol{\alpha}} \tbinom{\boldsymbol{\alpha}}{\boldsymbol{\mu}} \left(D^{\boldsymbol{\mu}}(\nabla u_\varepsilon),D^{\boldsymbol{\alpha}-\boldsymbol{\mu}}(\nabla u_\varepsilon)\right).
\end{split}    
\end{equation}
Therefore, using the previous computation, \eqref{K_2_s} can be rewritten as follows
\begin{equation}\label{K_2_sl}
\begin{split}
    &D^{\boldsymbol{\alpha}}\Big((\varepsilon+|\nabla u_\varepsilon|^2)^\frac{m-2}{2}\Big)
     = (m-2)(\varepsilon+|\nabla u_\varepsilon|^2)^{\frac{m-4}{2}}\left(D^{\boldsymbol{\alpha}}(\nabla u_\varepsilon),\nabla u_\varepsilon\right)\\
    &\qquad\quad+ \frac{m-2}{2} (\varepsilon+|\nabla u_\varepsilon|^2)^{\frac{m-4}{2}}\sum_{0<\boldsymbol{\mu}< \boldsymbol{\alpha}} \tbinom{\boldsymbol{\alpha}}{\boldsymbol{\mu}} \left(D^{\boldsymbol{\mu}}(\nabla u_\varepsilon),D^{\boldsymbol{\alpha}-\boldsymbol{\mu}}(\nabla u_\varepsilon)\right)\\
    &\qquad\quad+ \sum_{r=2}^{|\boldsymbol{\alpha}|} c_1 (\varepsilon+|\nabla u_\varepsilon|^2)^{\frac{m-2}{2} -r} \sum_{p_1(\boldsymbol{\alpha},{r})} \boldsymbol{\alpha}! \frac{\left(D^{\boldsymbol{l}_1}(|\nabla u_\varepsilon|^2)\right)^{k_1}}{k_1!(\boldsymbol{l}_1!)^{k_1}}\\
    &\qquad\quad + \sum_{r=1}^{|\boldsymbol{\alpha}|} c_1 (\varepsilon+|\nabla u_\varepsilon|^2)^{\frac{m-2}{2} -r} \sum_{s=2}^{|\boldsymbol{\alpha}|} \sum_{p_s(\boldsymbol{\alpha},{r})} \boldsymbol{\alpha}! \prod_{j=1}^s \frac{\left(D^{\boldsymbol{l}_j}(|\nabla u_\varepsilon|^2)\right)^{k_j}}{k_j!(\boldsymbol{l}_j!)^{k_j}}.
\end{split}
\end{equation}
An easy computation, see \eqref{After_c_1} and recalling  that $m=|\boldsymbol{\alpha}|+1$, shows that \eqref{K_2_sl} becomes
\begin{equation}\label{K_2_sl2}
\begin{split}
    &D^{\boldsymbol{\alpha}}\left((\varepsilon+|\nabla u_\varepsilon|^2)^\frac{m-2}{2}\right) = (m-2)(\varepsilon+|\nabla u_\varepsilon|^2)^{\frac{m-4}{2}}\left(D^{\boldsymbol{\alpha}}(\nabla u_\varepsilon),\nabla u_\varepsilon\right)\\
    &\qquad\quad+ \frac{m-2}{2} (\varepsilon+|\nabla u_\varepsilon|^2)^{-\frac{1}{2}}\sum_{0<\boldsymbol{\mu}< \boldsymbol{\alpha}} \tbinom{\boldsymbol{\alpha}}{\boldsymbol{\mu}} \left((\varepsilon+|\nabla u_\varepsilon|^2)^{\frac{|\boldsymbol{\mu}|-1}{2}}D^{\boldsymbol{\mu}}(\nabla u_\varepsilon),\right.\\
    &\qquad\qquad\left., (\varepsilon+|\nabla u_\varepsilon|^2)^{\frac{|\boldsymbol{\alpha}|-|\boldsymbol{\mu}|-1}{2}}D^{\boldsymbol{\alpha}-\boldsymbol{\mu}}(\nabla u_\varepsilon)\right)\\
    &\qquad\quad+ \sum_{r=2}^{|\boldsymbol{\alpha}|} c_1 (\varepsilon+|\nabla u_\varepsilon|^2)^{-\frac{1}{2}} \sum_{p_1(\boldsymbol{\alpha},{r})} c_2\Big(\sum_{\boldsymbol{\mu}\leq \boldsymbol{l}_1} \tbinom{\boldsymbol{l}_1}{\boldsymbol{\mu}} \Big((\varepsilon+|\nabla u_\varepsilon|^2)^{\frac{|\boldsymbol{\mu}|-1}{2}}D^{\boldsymbol{\mu}}(\nabla u_\varepsilon),\\
    &\qquad\qquad ,(\varepsilon+|\nabla u_\varepsilon|^2)^{\frac{|\boldsymbol{l}_1|-|\boldsymbol{\mu}|-1}{2}}D^{\boldsymbol{l}_1-\boldsymbol{\mu}}(\nabla u_\varepsilon)\Big)\Big)^{k_1}\\
    &\qquad\quad + \sum_{r=1}^{|\boldsymbol{\alpha}|} c_1 (\varepsilon+|\nabla u_\varepsilon|^2)^{-\frac{1}{2}} \sum_{s=2}^{|\boldsymbol{\alpha}|} \sum_{p_s(\boldsymbol{\alpha},{r})} \boldsymbol{\alpha}! \prod_{j=1}^s c_3\\
    &\qquad\qquad \times\Big(\sum_{\boldsymbol{\mu}\leq \boldsymbol{l}_j} \tbinom{\boldsymbol{l}_j}{\boldsymbol{\mu}} \Big((\varepsilon+|\nabla u_\varepsilon|^2)^{\frac{|\boldsymbol{\mu}|-1}{2}}D^{\boldsymbol{\mu}}(\nabla u_\varepsilon) ,(\varepsilon+|\nabla u_\varepsilon|^2)^{\frac{|\boldsymbol{l}_j|-|\boldsymbol{\mu}|-1}{2}}D^{\boldsymbol{l}_j-\boldsymbol{\mu}}(\nabla u_\varepsilon)\Big)\Big)^{k_j},
\end{split}
\end{equation}
where $c_2=\frac{\boldsymbol{\alpha}!}{k_1!(\boldsymbol{l}_1!)^{k_1}}$ and $c_3= \frac{1}{k_j!(\boldsymbol{l}_j!)^{k_j}}$.
 Thus, using \eqref{K_2_sl2}, we can estimate $K_2$ as follows:
\begin{equation}\label{1E_K2}
    \begin{split}
    &|K_2|\leq (m-2)(\varepsilon+|\nabla u_\varepsilon|^2)^{\frac{m-2}{2}}|D^{\boldsymbol{\alpha}}(\nabla u_\varepsilon)|\\
    &\qquad+ \frac{m-2}{2} \sum_{0<\boldsymbol{\mu}< \boldsymbol{\alpha}} \tbinom{\boldsymbol{\alpha}}{\boldsymbol{\mu}} \left|(\varepsilon+|\nabla u_\varepsilon|^2)^{\frac{|\boldsymbol{\mu}|-1}{2}}D^{\boldsymbol{\mu}}(\nabla u_\varepsilon)\right|\\
    &\qquad\qquad \times\left|(\varepsilon+|\nabla u_\varepsilon|^2)^{\frac{|\boldsymbol{\alpha}|-|\boldsymbol{\mu}|-1}{2}}D^{\boldsymbol{\alpha}-\boldsymbol{\mu}}(\nabla u_\varepsilon)\right|\\
    &\qquad+ \sum_{r=2}^{|\boldsymbol{\alpha}|} c_1 \sum_{p_1(\boldsymbol{\alpha},{r})} c_2c(|\boldsymbol{l}_1|)
    \sum_{\boldsymbol{\mu}\leq \boldsymbol{l}_1} \tbinom{\boldsymbol{l}_1}{\boldsymbol{\mu}} \Big|(\varepsilon+|\nabla u_\varepsilon|^2)^{\frac{|\boldsymbol{\mu}|-1}{2}}D^{\boldsymbol{\mu}}(\nabla u_\varepsilon)\Big|^{k_1}\\
    &\qquad\qquad \times\Big|(\varepsilon+|\nabla u_\varepsilon|^2)^{\frac{|\boldsymbol{l}_1|-|\boldsymbol{\mu}|-1}{2}}D^{\boldsymbol{l}_1-\boldsymbol{\mu}}(\nabla u_\varepsilon)\Big|
    ^{k_1}\\
    &\qquad + \sum_{r=1}^{|\boldsymbol{\alpha}|} c_1 \sum_{s=2}^{|\boldsymbol{\alpha}|} \sum_{p_s(\boldsymbol{\alpha},{r})} \boldsymbol{\alpha}! \prod_{j=1}^s c_3c(|\boldsymbol{l}_j|)\\
    &\qquad\qquad \times \sum_{\boldsymbol{\mu}\leq \boldsymbol{l}_j} \tbinom{\boldsymbol{l}_j}{\boldsymbol{\mu}} \Big|(\varepsilon+|\nabla u_\varepsilon|^2)^{\frac{|\boldsymbol{\mu}|-1}{2}}D^{\boldsymbol{\mu}}(\nabla u_\varepsilon)\Big|^{k_j}\Big|(\varepsilon+|\nabla u_\varepsilon|^2)^{\frac{|\boldsymbol{l}_j|-|\boldsymbol{\mu}|-1}{2}}D^{\boldsymbol{l}_j-\boldsymbol{\mu}}(\nabla u_\varepsilon)\Big|^{k_j},
\end{split}
\end{equation}
where $c(|\boldsymbol{l}_1|)= \max\{1,|\boldsymbol{l}_1|^{k_1-1}\}$ and $c(|\boldsymbol{l}_j|)= \max\{1,|\boldsymbol{l}_j|^{k_j-1}\}$.\\
Using iterated H\"older inequalities and Lemma \ref{sisso}, see \eqref{RESULT} and \eqref{Inductive_Step}, we obtain
\begin{equation}\label{K2_C}
    \|K_2\|_{L^q(B_R)} \leq C,
\end{equation}
where $C>0$ independent of $\varepsilon$.
In an analogous way, we can estimate the $L^q$-norm of $K_3$. 
Summing up, by \eqref{K_1} and \eqref{K2_C} we deduce that
\begin{equation}\label{Final_conv}
    \left\|D^{\boldsymbol{\alpha}} \left((\varepsilon+|\nabla u_\varepsilon|^2)^\frac{m-2}{2}\nabla u_\varepsilon\right) \right\|_{L^q(B_R)} \leq C,
\end{equation}
where $C=C(n,m,q,\|D^{\boldsymbol{\nu}}f\|_{L^q_{loc}(\Omega)},\|f\|_{S_{loc}(\Omega)},R,p,\|\nabla u\|_{L^\infty_{loc}(\Omega)})$ is a positive constant independent of $\varepsilon$. The case $|\boldsymbol{\alpha}|<m-1$ follows immediately from the same computations and by \eqref{Inductive_Step}. This proves the claim.

Therefore, by \eqref{CONV}, it follows that  there exists $w \in W^{m-1,q}(B_R)$ such that:
\begin{equation}\label{conv_deb}
    (\varepsilon+|\nabla u_\varepsilon|^2)^\frac{m-2}{2}\nabla u_\varepsilon \rightharpoonup w \quad \text{in } \ W^{m-1,q}(B_R).
\end{equation}
Moreover, the Rellich–Kondrachov Theorem implies
\begin{equation}\label{forte_m}
    (\varepsilon+|\nabla u_\varepsilon|^2)^\frac{m-2}{2}\nabla u_\varepsilon \rightarrow w \quad \text{in } \ W^{m-2,q}(B_R).
\end{equation}
In addition, we know that $\nabla u_\varepsilon \rightarrow \nabla u$ uniformly in $B_R$ (see \cite{Ant1,Anto2,3,DB,L2}), thus we can conclude:
\begin{equation}\label{before_f}
    w\equiv |\nabla u|^{m-2} \nabla u \in W^{m-1,q}(B_R).
\end{equation}
In the final step, we remove assumption $f\in C^{m-2,\beta'}_{loc}$.
By standard density argument one can infer that there exists a sequence $\{f_{l}\} \subset C^\infty(\Omega)$ such that
\begin{equation}\label{nuovo_conv}
    f_l \rightarrow f \ \ \text { in }\ W^{m-2,q}_{loc}(\Omega) \cap W^{m-3,s}_{loc}(\Omega), \ s \geq 1.
\end{equation}
We proceed by considering a sequence $\{u_l\}$ of weak solutions to the following problem
\begin{equation} \label{problu_L}
		\begin{cases}
			-\operatorname{div}\left( |\nabla u_{l}|^{p-2} \, \nabla u_{l} \right) = f_l & \text{in } B_{2R} \\
			u_{l} = u & \text{on } \partial B_{2R}.
		\end{cases}
\end{equation}
By \eqref{before_f}, we deduce that
\begin{equation}\label{ineq_u_l}
    \left\||\nabla u_{l}|^{m-2}\nabla u_l\right\|_{W^{m-1,q}(B_R)} \leq C(n,m,q,R,p,\|\nabla u_l\|_{L^\infty_{loc}(\Omega)},\|f_l\|_{W^{m-2,q}_{loc}(\Omega)},\|f_l\|_{S_{loc}(\Omega)}).
\end{equation}
In addition, we recall that (see \cite[Theorem 1.7]{L2})
\begin{equation}\label{G_L}
    \|\nabla u_l\|_{L^\infty_{loc}(\Omega)} \leq C(p,n,\|f_l\|_{L^s_{loc}(\Omega)}),
\end{equation}
with $s>n$.\\
Observe that, by Lemma \ref{approssimazione}, \eqref{nuovo_conv} and \eqref{G_L}, inequality \eqref{ineq_u_l} would imply the existence of $\bar w \in W^{m-1,q}(B_R)$ such that
\begin{equation}\label{conv:W:m-1}
    |\nabla u_l|^{m-2} \nabla u_l \rightharpoonup \bar w \ \ \text{ in } W^{m-1,q}(B_R).
\end{equation}
Furthermore, since the sequence $\{u_l\}$ satisfies
\begin{equation*}
    \nabla u_l \rightarrow \nabla u \ \ \text{uniformly in } K\subset \subset B_{R},
\end{equation*}
we conclude that
$    \bar w \equiv |\nabla u|^{m-2} \nabla u \in W^{m-1,q}(B_R)$.\\
It remains to prove that
    $|\nabla u|^{m-2}D^2 u \in W^{m-2,q}_{loc}(\Omega).$
Assume $f\in C^{m-2,\beta'}_{loc}(\Omega)$. Arguing as done  in the proof of \eqref{CONV}, we obtain
\begin{equation}\label{CONV2}
    \left\|(\varepsilon+|\nabla u_\varepsilon|^2)^\frac{m-2}{2}D^2 u_\varepsilon\right\|_{W^{m-2,q}(B_R)} \leq C,
\end{equation}
where $C>0$ independent of $\varepsilon$. Therefore, it follows that  there exists $\tilde w \in W^{m-2,q}(B_R)$ such that:
\begin{equation*}
    (\varepsilon+|\nabla u_\varepsilon|^2)^\frac{m-2}{2}D^2 u_\varepsilon \rightharpoonup \tilde w \quad \text{in } \ W^{m-2,q}(B_R).
\end{equation*}
Moreover, the Rellich–Kondrachov Theorem implies
\begin{equation*}
    (\varepsilon+|\nabla u_\varepsilon|^2)^\frac{m-2}{2}D^2 u_\varepsilon \rightarrow \tilde w \quad \text{in } \ W^{m-3,q}(B_R).
\end{equation*}
In addition, by the uniform convergence of the gradients and by Remark \ref{regolarizziamo}, we have
\begin{equation*}
    (\varepsilon+|\nabla u_\varepsilon|^2)^\frac{m-2}{2}D^2 u_\varepsilon \rightharpoonup |\nabla u|^{m-2}D^2 u \quad \text{in } \ L^q(B_R).
\end{equation*}
This  implies that
$\tilde w \equiv |\nabla u|^{m-2}D^2 u \in W^{m-2,q}(B_R).$ The assumption $f\in C^{m-2,\beta '}_{loc}$ can be removed as before. The fact that $
|\nabla u|^{m-2}\nabla u \in C^{m-2,\gamma}_{\mathrm{loc}}(\Omega)$ and 
$|\nabla u|^{m-2} D^2 u \in C^{m-3,\gamma}_{\mathrm{loc}}(\Omega)$
follows by classical Morrey's embedding theorems. Indeed, if we choose $q>n$, \cite[Corollary 9.13]{Brezis} yields $|\nabla u|^{m-2}\nabla u\in C^{m-2,\gamma}_{\mathrm{loc}}(\Omega)$, with $\gamma=1-n/q,$ therefore it is enough to choose $q$ large enough. The same argument has to be applied for $|\nabla u|^{m-2} D^2 u$.
\end{proof}

\section{\texorpdfstring{$L^{\infty}$-type estimates}{L-infinity type estimates}}\label{sec:Linf}
\noindent This section is entirely devoted to the proof of Theorem \ref{INFINITO}.\\
For any $m\in \N$, we consider the $m$-th linearization of the weak formulation of the regularized problem \eqref{eq:problregol}.
First, we state the weak formulation of the problem \eqref{eq:problregol}, that is, \begin{equation}\label{equazione debole3}
\int_{B_{2R}} (\varepsilon + |\nabla u_\varepsilon|^2)^{\frac{p-2}{2}} (\nabla u_\varepsilon, \nabla \zeta) \,dx
= \int_{B_{2R}} f  \zeta  \,dx
\quad \forall \zeta \in C_c^{\infty}(B_{2R}),
\end{equation}
where $B_{2R}:=B_{2R}(x_0)$.
 We consider the following function $\zeta:=D^{\boldsymbol{\nu}}\varphi,$ with $\varphi\in C_c^{\infty}(B_{2R})$ and $|\boldsymbol{\nu}|=m$, and we test the previous equation \eqref{equazione debole3} with $\zeta$, obtaining
\begin{equation*}\label{m_linearizzato1}
    \int_{B_{2R}} \left (D^{\boldsymbol{\nu}}\left((\varepsilon + |\nabla u_\varepsilon|^2)^{\frac{p-2}{2}} \nabla u_\varepsilon\right), \nabla \varphi\right) \,dx
= \int_{B_{2R}} D^{\boldsymbol{\nu}} (f) \varphi  \,dx, \quad \forall \varphi \in C_c^{\infty}(B_{2R}).
\end{equation*}
Employing Leibniz’s rule for derivatives, we get 
\begin{equation}\label{eq:100}
    \begin{split}
        \sum_{\boldsymbol{\gamma}\leq \boldsymbol{\nu}} \int_{B_{2R}} \tbinom{\boldsymbol{\nu}}{\boldsymbol{\gamma}} D^{\boldsymbol{\gamma}}\left((\varepsilon + |\nabla u_\varepsilon|^2)^{\frac{p-2}{2}}\right) (D^{\boldsymbol{\nu}-\boldsymbol{\gamma}}(\nabla u_\varepsilon), \nabla \varphi) \,dx
= \int_{B_{2R}} D^{\boldsymbol{\nu}} (f) \varphi  \,dx.
    \end{split}
\end{equation}
For a fixed multi-index $\boldsymbol{\gamma}$, we apply Faa di Bruno's formula, see \eqref{FFA_s}, to the term 
$\\D^{\boldsymbol{\gamma}}\Bigl((\varepsilon + |\nabla u_\varepsilon|^2)^{\frac{p-2}{2}}\Bigr).$ 
More precisely, we define the functions 
\begin{equation}\label{definizione_di_h}
    h(t) := (\varepsilon + t)^{\frac{p-2}{2}} \quad \text{and} \quad v(x) := |\nabla u_\varepsilon(x)|^2,
\end{equation}
and then apply the formula to $h(v(x))$, yielding 
\begin{equation}\label{eq:101}
\begin{split}
D^{\boldsymbol{\gamma}}\Bigl((\varepsilon + |\nabla u_\varepsilon|^2)^{\frac{p-2}{2}}\Bigr) = \sum_{r=1}^ {|\boldsymbol{\gamma}|} D^{{r}} h \sum_{s=1}^{|\boldsymbol{\gamma}|} \sum_{p_s(\boldsymbol{\gamma},r)}(\boldsymbol{\gamma}!) \prod_{j=1}^{s} \frac{\left(D^{\boldsymbol{l}_j}(|\nabla u_\varepsilon|^2)\right)^{{k}_j}}{k_j!(\boldsymbol{l}_j!)^{k_j}},
    \end{split}
\end{equation}
where \begin{equation}\label{set}
\begin{split}
    p_s(\boldsymbol{\gamma},{r})= \Big\{({k}_1,...,{k}_s;\boldsymbol{l}_1,...,\boldsymbol{l}_s):  0\prec \boldsymbol{l}_1\prec\cdots \prec\boldsymbol{l}_s;k_i>0;\sum_{i=1}^s k_i ={r}, \sum_{i=1}^{s} k_i \boldsymbol{l}_i = \boldsymbol{\gamma}\Big\},
\end{split}
\end{equation}
where each $\boldsymbol{l}_i$ is an $n$-dimensional multi-index and each $k_i$ is a positive scalar, for any $i = 1, \ldots, s$. We recall that the notation $D^{{r}} h$ denotes the ${r}$-th order derivative of the function $h(t)$ with respect to $t$, evaluated at $|\nabla u_\varepsilon|^2$.
By \eqref{eq:100} and \eqref{eq:101}, we obtain 
\begin{equation}\label{eq:mlinearizzato}
    \begin{split}
        &\sum_{\boldsymbol{\gamma}\leq \boldsymbol{\nu}} \int_{B_{2R}} \tbinom{\boldsymbol{\nu}}{\boldsymbol{\gamma}}  \sum_{r=1}^ {|\boldsymbol{\gamma}|} D^{r} h \sum_{s=1}^{|\boldsymbol{\gamma}|} \sum_{p_s(\boldsymbol{\gamma},r)}(\boldsymbol{\gamma}!) \prod_{j=1}^{s} \frac{\left(D^{\boldsymbol{l}_j}(|\nabla u_\varepsilon|^2)\right)^{{k}_j}}{k_j!(\boldsymbol{l_j}!)^{k_j}} (D^{\boldsymbol{\nu}-\boldsymbol{\gamma}}(\nabla u_\varepsilon), \nabla \varphi) \,dx
\\&\qquad= \int_{B_{2R}} D^{\boldsymbol{\nu}} (f) \varphi  \,dx.
    \end{split}
\end{equation}
Let $\psi$ be a nonnegative compactly supported smooth function in $B_{2R}$. For fixed numbers $k\in \R$ and $q\ge 1$, by density, we consider the following test function 
\begin{equation}\label{test1}
    \varphi (x):= (\varepsilon+|\nabla u_\varepsilon|^2)^{\frac{2(m-2+k)q+2-p}{2}}|D^{\boldsymbol\nu}u_{\varepsilon}|^{2q-2}D^{\boldsymbol\nu}u_{\varepsilon}\psi^2\in C_{c}^{1}(B_{2R}). 
\end{equation}
So, we get 
\begin{equation}\label{gradiente_test}
\begin{split}
 \nabla \varphi&= (2q-1)(\varepsilon+|\nabla u_\varepsilon|^2)^{\frac{2(m-2+k)q+2-p}{2}}|D^{\boldsymbol\nu}u_{\varepsilon}|^{2q-2}D^{\boldsymbol\nu}(\nabla u_{\varepsilon})\psi^2 
 \\ &\quad+(2(m-2+k)q+2-p)(\varepsilon+|\nabla u_\varepsilon|^2)^{\frac{2(m-2+k)q-p}{2}}|D^{\boldsymbol\nu}u_{\varepsilon}|^{2q-2}D^{\boldsymbol\nu}u_{\varepsilon}\psi^2 D^2u_\varepsilon\nabla u_\varepsilon 
\\&\quad+2(\varepsilon+|\nabla u_\varepsilon|^2)^{\frac{2(m-2+k)q+2-p}{2}}|D^{\boldsymbol\nu}u_{\varepsilon}|^{2q-2}D^{\boldsymbol\nu}u_{\varepsilon}\psi \nabla \psi:=J_1+J_2+J_3.
\end{split}
\end{equation}
%
%
%

\begin{rem}
    Note that when $\boldsymbol{\gamma}=\boldsymbol{0}$, in the sum of \eqref{eq:mlinearizzato} we obtain a summand containing a derivative of order $m+1$. Now, suppose $\boldsymbol{\gamma}\neq\boldsymbol{0}$, and write $\bigcup_{\boldsymbol{\gamma},s,r} p_s(\boldsymbol{\gamma},r)=p_1(\boldsymbol{\nu},1)\cup p_1(\boldsymbol{\nu},1)^c.$ We claim that for indexes in $p_1(\boldsymbol{\nu},1)^c$, the derivatives appearing in the related summands in \eqref{eq:mlinearizzato} are of order at most $m$, when $\nabla \varphi \rightsquigarrow J_2+J_3$. Indeed indexes belonging to $p_1(\boldsymbol{\nu},1)^c$ verify $\boldsymbol{\gamma}\neq \boldsymbol{\nu}$ or $s\neq 1$ or $r\neq 1$. Suppose, for instance, that $\boldsymbol{\gamma}\neq \boldsymbol{\nu}.$ Then, since $\boldsymbol{\gamma} \leq \boldsymbol{\nu}$, we have $|\boldsymbol{\gamma}|\leq |\boldsymbol{\nu}|-1$ which, in turn, implies that $|\boldsymbol{l_j}|\leq |\boldsymbol{\nu}|-1=m-1$, for any $j$. For the other cases one has to reason in a similar way.
\end{rem}
Owing to the previous remark, in the following lemma we provide an upper bound for the terms appearing in \eqref{eq:mlinearizzato} that involve derivatives of order at most $m$.
\begin{lem}\label{pantusa1}
 Let $\Omega$ be a domain in $\R^n$, $n\ge 2$, and $B_{2R}\subset\subset \Omega$. Fix $\zeta>2$ and let $\hat s:=1/k$, for $0<k<1$. Let $q\geq \hat s$. Fix $|\boldsymbol{\nu}|=m$. Let $u_\varepsilon$ be a solution  of \eqref{eq:problregol}, and consider $J_2$ and $J_3$ defined in \eqref{gradiente_test}, see also \eqref{test1}. There exists  $\mathfrak{C}:=\mathfrak{C}(k,n,m,\zeta)>0$  such that if $|p-2|<\mathfrak{C}$ then the following holds true. 
    \begin{equation}\label{pantusa}
        \begin{split}
    &\int_{B_{2R}}\Big(\sum_{\boldsymbol{0}<\boldsymbol{\gamma}\leq \boldsymbol{\nu}}  \tbinom{\boldsymbol{\nu}}{\boldsymbol{\gamma}}  \sum_{r=1}^ {|\boldsymbol{\gamma}|} D^{r} h \sum_{s=1}^{|\boldsymbol{\gamma}|} \sum_{p_1(\boldsymbol{\nu},1)^c}(\boldsymbol{\gamma}!) \prod_{j=1}^{s} \frac{\left(D^{\boldsymbol{l}_j}(|\nabla u_\varepsilon|^2)\right)^{{k}_j}}{k_j!(\boldsymbol{l_j}!)^{k_j}} D^{\boldsymbol{\nu}-\boldsymbol{\gamma}}(\nabla u_\varepsilon),J_2+J_3\Big)
    \\&\qquad+ \Big(D^{1} h \sum_{\boldsymbol{0}\neq \boldsymbol{\mu}< \boldsymbol{\nu}}  \tbinom{\boldsymbol{\nu}}{\boldsymbol{\mu}}\,  (D^{\boldsymbol{\nu}-\boldsymbol{\mu}}(\nabla u_\varepsilon),D^{\boldsymbol{\mu}}(\nabla u_\varepsilon)) \nabla u_\varepsilon, J_2+J_3\Big)\,dx
    \\&\quad \leq q\mathcal{C}\Big(\int_{B_{2R}}(\varepsilon+|\nabla u_\varepsilon|^2)^\frac{\zeta(m-2+k)(q-\hat s)}{2} |D^{\boldsymbol\nu}u_{\varepsilon}|^{\zeta(q-\hat s)}(\psi^\zeta + |\nabla \psi|^\zeta)\,dx\Big)^{\frac 2\zeta},
        \end{split}
    \end{equation}
    where $\mathcal{C}$ depending on $m,k,n,\zeta,l,R,p,\| \nabla u\|_{L^{\infty}(B_{2R})}$ and $\|f\|_{W^{m-2,l}( B_{2R})}$.
\end{lem}
\begin{proof}
We denote by $R_1(J_2)$, $R_1(J_3)$, $R_2(J_2)$ and $R_2(J_3)$ the terms to be estimated, where $R_1$ are the integrals in the first line of the l.h.s of the thesis and $R_2$ the ones in the second.
We start by estimating $R_1(J_2)$. Recalling that 
\begin{equation}\label{derivate di h}
    D^{r}h \le C(p)(\varepsilon+|\nabla u_\varepsilon|^2)^{\frac{p-2(r+1)}{2}},
\end{equation}
where $C(p)$ is a positive constant, we have
\begin{equation*}
    \begin{split}
R_1(J_2)&:=\int_{B_{2R}}\sum_{\boldsymbol{0}<\boldsymbol{\gamma}\leq \boldsymbol{\nu}}  \tbinom{\boldsymbol{\nu}}{\boldsymbol{\gamma}}  \sum_{r=1}^ {|\boldsymbol{\gamma}|} D^{r} h \sum_{s=1}^{|\boldsymbol{\gamma}|} \sum_{p_1(\boldsymbol{\nu},1)^c}(\boldsymbol{\gamma}!) \prod_{j=1}^{s} \frac{\left(D^{\boldsymbol{l}_j}(|\nabla u_\varepsilon|^2)\right)^{{k}_j}}{k_j!(\boldsymbol{l_j}!)^{k_j}} (D^{\boldsymbol{\nu}-\boldsymbol{\gamma}}(\nabla u_\varepsilon), J_2)\,dx
        \\& \ \leq qC(|\boldsymbol{\nu}|,p,k)\int_{B_{2R}}\sum_{\boldsymbol{0}<\boldsymbol{\gamma}\leq \boldsymbol{\nu}} \sum_{r,s=1}^ {|\boldsymbol{\gamma}|} \sum_{p_1(\boldsymbol{\nu},1)^c}(\varepsilon + |\nabla u_\varepsilon|^2)^{\frac{2(m-2+k)q-1}{2}-r} \\
        &\qquad\quad\times\prod_{j=1}^{s} \left(D^{\boldsymbol{l}_j}(|\nabla u_\varepsilon|^2)\right)^{{k}_j} |D^{\boldsymbol\nu}u_{\varepsilon}|^{2q-1}|D^{\boldsymbol{\nu}-\boldsymbol{\gamma}}(\nabla u_\varepsilon)| |D^2u_\varepsilon|\psi^2\,dx.
    \end{split}
\end{equation*}
Using standard Young inequality we deduce

\begin{equation*}
    \begin{split}
        R_1(J_2)&\lesssim R_{1,1}(J_2)+R_{1,2}(J_2)\\&:=q\int_{B_{2R}} \sum_{\boldsymbol{0}< \boldsymbol{\gamma}\leq \boldsymbol{\nu}}\sum_{r,s=1}^ {|\boldsymbol{\gamma}|}  \sum_{p_1(\boldsymbol{\nu},1)^c} (\varepsilon+|\nabla u_\varepsilon|^2)^{(m-2+k)q-2r}\prod_{j=1}^{s} {\left(D^{\boldsymbol{l}_j}(|\nabla u_\varepsilon|^2)\right)^{2{k}_j}}\\
        &\qquad \qquad \times|D^{\boldsymbol{\nu}-\boldsymbol{\gamma}}(\nabla u_\varepsilon)|^2  |D^{\boldsymbol\nu}u_{\varepsilon}|^{2q-2}\psi^2\,dx\\
        &\qquad +q\int_{B_{2R}}(\varepsilon+|\nabla u_\varepsilon|^2)^{(m-2+k)q-1}|D^{\boldsymbol\nu}u_{\varepsilon}|^{2q}|D^2u_\varepsilon|^2\psi^2\,dx.
    \end{split}
\end{equation*}
 In the integrals above, it appears either a factor $|D^{\boldsymbol{\nu}}u_\varepsilon|^{2q}$ or $|D^{\boldsymbol{\nu}}u_\varepsilon|^{2q-2}$, in the following we shall make them explicit.   
 For reasons that will become clear later, we shall analyze separately the terms containing $D^{\boldsymbol{\nu}}u_\varepsilon$. With this aim, we note that the set $p_1(\boldsymbol{\nu},1)^c= p_{\boldsymbol{\nu}-1} \cup p_{\boldsymbol{\nu}-1}^c$, where $p_{\boldsymbol{\nu}-1}$ denotes all partitions in which there exists an index $j\in\{1,...,s\}$ such that $|\boldsymbol{l}_j|=|\boldsymbol{\nu}|-1$. Obviously, this happens only if $\boldsymbol{\nu}-1 \leq \boldsymbol{\gamma} \leq \boldsymbol{\nu}$.  Moreover, it is straightforward to verify that in each of these partitions that appear in $p_{\boldsymbol{\nu}-1}$, there exists a unique $\boldsymbol{l}_j$ such that $|\boldsymbol{l}_j|=|\boldsymbol{\nu}|-1$ and as a consequence the corresponding integer $k_j$ is equal to $1$. We remark that, since by definition the elements of $p_{\boldsymbol{\nu}-1}^c$ satisfy $|\boldsymbol{l}_j| < |\boldsymbol{\nu}|-1$, it follows that in $p_{\boldsymbol{\nu}-1}^c$ there are no terms of type $D^{\boldsymbol{\nu}}u_\varepsilon$. With this in mind, we can rewrite the term $R_{1,1}(J_2)$ as follows
\begin{equation*}
    \begin{split}
R_{1,1}(J_2)&=q\int_{B_{2R}} \sum_{\boldsymbol{\nu-1}\leq \boldsymbol{\gamma}\leq \boldsymbol{\nu}}\sum_{r,s=1}^ {|\boldsymbol{\gamma}|} \sum_{p_{\boldsymbol{\nu}-1}} (\varepsilon+|\nabla u_\varepsilon|^2)^{(m-2+k)q-2r}|D^{\boldsymbol{\nu}-1}(|\nabla u_\varepsilon|^2)|^2\\
&\qquad\qquad\qquad\times\prod_{j=1}^{s-1} {\left(D^{\boldsymbol{l}_j}(|\nabla u_\varepsilon|^2)\right)^{2{k}_j}} |D^{\boldsymbol{\nu}-\boldsymbol{\gamma}}(\nabla u_\varepsilon)|^2  |D^{\boldsymbol\nu}u_{\varepsilon}|^{2q-2}\psi^2\,dx
        \\&\qquad +q\int_{B_{2R}} \sum_{\boldsymbol{0}< \boldsymbol{\gamma}\leq \boldsymbol{\nu}}\sum_{r,s=1}^ {|\boldsymbol{\gamma}|} \sum_{p^c_{\boldsymbol{\nu}-1}} (\varepsilon+|\nabla u_\varepsilon|^2)^{(m-2+k)q-2r}\prod_{j=1}^{s} {\left(D^{\boldsymbol{l}_j}(|\nabla u_\varepsilon|^2)\right)^{2{k}_j}} \\&\qquad\qquad\qquad\times |D^{\boldsymbol{\nu}-\boldsymbol{\gamma}}(\nabla u_\varepsilon)|^2  |D^{\boldsymbol\nu}u_{\varepsilon}|^{2q-2}\psi^2\,dx.
    \end{split}
\end{equation*}
Observe that the term $D^{\boldsymbol{\nu}-\boldsymbol{\gamma}}(\nabla u_\varepsilon)$, appearing in the second integral on the right-hand side of $R_{1,1}(J_2)$, contains a derivative of order $m$ only if $|\boldsymbol{\gamma}|=1$. In particular, in this case $r=s=1$, $|l_1|=|\gamma|=1$ and consequently $k_1=1$. Therefore, we shall treat this case separately from the others.
Moreover, using the Leibniz rule for the term $D^{\boldsymbol{\nu}-1}(|\nabla u_\varepsilon|^2)$, we bound the term $R_{1,1}(J_2)$ from above by  
\begin{equation*}
    \begin{split}
         &q\int_{B_{2R}} \sum_{\boldsymbol{\nu-1}< \boldsymbol{\gamma}\leq \boldsymbol{\nu}}\sum_{r,s=1}^{|\boldsymbol{\gamma}|} \sum_{p_{\boldsymbol{\nu}-1}} (\varepsilon+|\nabla u_\varepsilon|^2)^{(m-2+k)q-2r}\sum_{\boldsymbol{0}<\boldsymbol{\mu}<\boldsymbol{\nu-1}}|D^{\boldsymbol{\nu}-1-\boldsymbol{\mu}}(\nabla u_\varepsilon)|^2|D^{\boldsymbol{\mu}}(\nabla u_\varepsilon)|^2\\
        &\qquad\qquad\qquad\times\prod_{j=1}^{s-1} {\left(D^{\boldsymbol{l}_j}(|\nabla u_\varepsilon|^2)\right)^{2{k}_j}} |D^{\boldsymbol{\nu}-\boldsymbol{\gamma}}(\nabla u_\varepsilon)|^2  |D^{\boldsymbol\nu}u_{\varepsilon}|^{2q-2}\psi^2\,dx
        \\&\quad+q\int_{B_{2R}} \sum_{\boldsymbol{\nu-1}< \boldsymbol{\gamma}\leq \boldsymbol{\nu}}\sum_{r,s=1}^ {|\boldsymbol{\gamma}|} \sum_{p_{\boldsymbol{\nu}-1}} (\varepsilon+|\nabla u_\varepsilon|^2)^{(m-2+k)q+1-2r}\\
        &\qquad\qquad\qquad\times \prod_{j=1}^{s-1} {\left(D^{\boldsymbol{l}_j}(|\nabla u_\varepsilon|^2)\right)^{2{k}_j}} |D^{\boldsymbol{\nu}-\boldsymbol{\gamma}}(\nabla u_\varepsilon)|^2  |D^{\boldsymbol\nu}u_{\varepsilon}|^{2q-2}|D^{\boldsymbol\nu-1}(\nabla u_{\varepsilon})|^2\psi^2\,dx
        \\&\quad+q\int_{B_{2R}}  (\varepsilon+|\nabla u_\varepsilon|^2)^{(m-2+k)q-1} |D^2 u_\varepsilon|^2   |D^{\boldsymbol\nu}u_{\varepsilon}|^{2q-2}|D^{\boldsymbol\nu-1}(\nabla u_{\varepsilon})|^2\psi^2\,dx
        \\&\quad+q\int_{B_{2R}} \sum_{|\boldsymbol{\gamma}|\neq 0,1}\sum_{r,s=1}^ {|\boldsymbol{\gamma}|} \sum_{p^c_{\boldsymbol{\nu}-1}} (\varepsilon+|\nabla u_\varepsilon|^2)^{(m-2+k)q-2r}\prod_{j=1}^{s} {\left(D^{\boldsymbol{l}_j}(|\nabla u_\varepsilon|^2)\right)^{2{k}_j}} \\&\qquad\qquad\qquad\times |D^{\boldsymbol{\nu}-\boldsymbol{\gamma}}(\nabla u_\varepsilon)|^2  |D^{\boldsymbol\nu}u_{\varepsilon}|^{2q-2}\psi^2\,dx\\
        &=:R_{1,1,1}(J_2)+R_{1,1,2}(J_2)+R_{1,1,3}(J_2)+R_{1,1,4}(J_2).
    \end{split}
\end{equation*}

We prove in detail the estimate for the term $R_{1,1,4}(J_2)$, being the others similar.  Consider $\hat s,q,\zeta$ as in the statement. Now, we write 
 \begin{equation*}
 \begin{split}
     (\varepsilon+|\nabla u_\varepsilon|^2)^{-2r}&= (\varepsilon+|\nabla u_\varepsilon|^2)^{|\gamma|-2r} (\varepsilon+|\nabla u_\varepsilon|^2)^{-|\gamma|}
     \\& =(\varepsilon+|\nabla u_\varepsilon|^2)^{\sum_j(|\boldsymbol{l}_j|-2)k_j} (\varepsilon+|\nabla u_\varepsilon|^2)^{-|\gamma|},
      \end{split}
 \end{equation*}
where we have used $|\boldsymbol{\gamma}|-2r=\sum_j(|\boldsymbol{l}_j|-2)k_j$, see \eqref{set}. We obtain

\begin{equation*}
    \begin{split}
R_{1,1,4}(J_2)&=q\int_{B_{2R}}  (\varepsilon+|\nabla u_\varepsilon|^2)^{(m-2+k)(q-\hat s)}|D^{\boldsymbol\nu}u_{\varepsilon}|^{2q-2\hat s}\\&\qquad \times\sum_{|\boldsymbol{\gamma}|\neq 0,1}\sum_{r,s=1}^ {|\boldsymbol{\gamma}|} \sum_{p_{\boldsymbol{\nu}-1}^c}\prod_{j=1}^{s} (\varepsilon+|\nabla u_\varepsilon|^2)^{(|\boldsymbol{l}_j|-2)k_j}{\left(D^{\boldsymbol{l}_j}(|\nabla u_\varepsilon|^2)\right)^{2{k}_j}}\\&\qquad \times(\varepsilon+|\nabla u_\varepsilon|^2)^{|\boldsymbol{\nu}|-|\boldsymbol{\gamma}|-1} |D^{\boldsymbol{\nu}-\boldsymbol{\gamma}}(\nabla u_\varepsilon)|^2 (\varepsilon+|\nabla u_\varepsilon|^2)^{(|\boldsymbol{\nu}|-2)(\hat s-1)}|D^{\boldsymbol\nu}u_{\varepsilon}|^{2\hat s-2}\psi^2\,dx,
\end{split}
\end{equation*}
where we have used 

\begin{equation*}
    \begin{split}
        &(\varepsilon+|\nabla u_\varepsilon|^2)^{-|\boldsymbol{\gamma}|}|D^{\boldsymbol{\nu}-\boldsymbol{\gamma}}(\nabla u_\varepsilon)|^2  |D^{\boldsymbol\nu}u_{\varepsilon}|^{2q-2}
        \\&\quad=(\varepsilon+|\nabla u_\varepsilon|^2)^{-|\boldsymbol{\gamma}|}|D^{\boldsymbol{\nu}-\boldsymbol{\gamma}}(\nabla u_\varepsilon)|^2  |D^{\boldsymbol\nu}u_{\varepsilon}|^{2q-2\hat s}|D^{\boldsymbol\nu}u_{\varepsilon}|^{2\hat s-2}
        \\&\quad=(\varepsilon+|\nabla u_\varepsilon|^2)^{-(|\boldsymbol{\nu}|-2)(\hat s-1)-|\boldsymbol{\nu}|+1}|D^{\boldsymbol\nu}u_{\varepsilon}|^{2q-2\hat s}(\varepsilon+|\nabla u_\varepsilon|^2)^{|\boldsymbol{\nu}|-|\boldsymbol{\gamma}|-1}|D^{\boldsymbol{\nu}-\boldsymbol{\gamma}}(\nabla u_\varepsilon)|^2\\&\qquad\qquad\times   (\varepsilon+|\nabla u_\varepsilon|^2)^{(|\boldsymbol{\nu}|-2)(\hat s-1)}|D^{\boldsymbol\nu}u_{\varepsilon}|^{2\hat s-2}
    \end{split}
\end{equation*}
and, since $\hat{s}=1/k$ and $|\boldsymbol{\nu}|=m$,
$$-(|\boldsymbol{\nu}|-2)(\hat s-1)-|\boldsymbol{\nu}|+1=-(m-2+k)\hat s.$$
By using H\"older inequality with exponents $(\zeta/2, \zeta/(\zeta-2))$, we deduce

\begin{equation*}
    \begin{split}
R_{1,1,4}(J_2)& \lesssim q\Big(\int_{B_{2R}}  (\varepsilon+|\nabla u_\varepsilon|^2)^{\frac{\zeta (m-2+k)(q-\hat s)}{2}}|D^{\boldsymbol\nu}u_{\varepsilon}|^{\zeta(q-\hat s)}\psi^\zeta\,dx\Big)^{\frac 2\zeta}\\&\ \times\Big(\int_{B_{2R}}\sum_{|\boldsymbol{\gamma}|\neq 0,1}\sum_{r,s=1}^ {|\boldsymbol{\gamma}|} \sum_{p_{\boldsymbol{\nu}-1}^c} \prod_{j=1}^{s} (\varepsilon+|\nabla u_\varepsilon|^2)^{(|\boldsymbol{l}_j|-2)k_j\frac{\zeta}{\zeta-2}}{\left(D^{\boldsymbol{l}_j}(|\nabla u_\varepsilon|^2)\right)^{2{k}_j\frac{\zeta}{\zeta-2}}}\\&\ \times(\varepsilon+|\nabla u_\varepsilon|^2)^{(|\boldsymbol{\nu}|-|\boldsymbol{\gamma}|-1)\frac{\zeta}{\zeta-2}} |D^{\boldsymbol{\nu}-\boldsymbol{\gamma}}(\nabla u_\varepsilon)|^{2\frac{\zeta}{\zeta-2}} \\
&\times (\varepsilon+|\nabla u_\varepsilon|^2)^{(|\boldsymbol{\nu}|-2)(\hat s-1)\frac{\zeta}{\zeta-2}}|D^{\boldsymbol\nu}u_{\varepsilon}|^{(2\hat s-2)\frac{\zeta}{\zeta-2}}\,dx\Big)^{\frac {\zeta-2}{\zeta}},
\end{split}
\end{equation*}
one concludes by using iterative H\"older inequalities and Lemma \ref{sisso}. One has to reason exactly in the same way for the remaining terms $R_1(J_3)$,  $R_2(J_2)$, $R_2(J_3)$ and $R_{1,2}(J_2)$.
\end{proof}

Let $\boldsymbol{\nu}=(\nu_1,...,\nu_n)$ be a multi-index of order $m$. We consider the following function 
\begin{equation}\label{g}
    g_\varepsilon:=(\varepsilon+|\nabla u_\varepsilon|^2)^{\frac{(m-2+k)}{2}}|D^{\boldsymbol{\nu}}u_\varepsilon|.
\end{equation}
The rest of this section is devoted to the proof of the following fundamental proposition.

\begin{prop}\label{dinodino}
   Let $\Omega$ be a domain in $\R^n$, $n\ge 3$, and $B_{2R}\subset\subset \Omega$. Let $u_\varepsilon$ be a weak solution of the regularized problem \eqref{eq:problregol}. Assume $f\in C^{m}(\Omega)$.
   For $0< k\leq  1$ fixed, there exists  $\mathfrak{C}:=\mathfrak{C}(k,l,n,m)>0$  such that if $|p-2|<\mathfrak{C}$ then the following holds true. For any  $q\ge  1/k$ and any $\zeta$ such that $2l/(l-1)<\zeta<2^{*}$, with $l>n/2$, there exists $\mathcal{C}>0$ depending on $k,n,l$,$R,p$, $\| \nabla u\|_{L^{\infty}(B_{2R})}$ and $\|f\|_{W^{m,l}(B_{2R})}$ such that
\begin{align}\label{cane477}
    \begin{split}
    \|g_{\varepsilon}\|_{L^{2^*q}(B_{h'})}^q\leq \mathcal{C}\frac{q}{h-h'}\|g_{\varepsilon}\|_{L^{\zeta(q-\hat{s})}(B_h)}^{q-\hat{s}}, \quad \hat{s}:=1/k,
    \end{split}
\end{align}
for any $0<h'<h<R$  and where $g_{\varepsilon}$ is defined in \eqref{g}.

\end{prop}

\begin{proof}
    
By Sobolev embedding $\|g_{\varepsilon}\|_{L^{2^*q}(B_{h'})}^q\leq C_S (\| \nabla g^q_{\varepsilon}\|_{L^{2}(B_{h'})}+\|  g^q_{\varepsilon}\|_{L^{2}(B_{h'})})$. Therefore for our purposes it is important to estimate the $L^2$-norm of $\nabla g^q_\varepsilon$. We have
\begin{equation}\label{gradiente_di_g}
    \begin{split}
        \nabla g_\varepsilon^q&=q(\varepsilon+|\nabla u_\varepsilon|^2)^{\frac{q(m-2+k)}{2}}|D^{\boldsymbol{\nu}}u_\varepsilon|^{q-2}D^{\boldsymbol{\nu}}u_\varepsilon D^{\boldsymbol{\nu}}(\nabla u_\varepsilon) 
        \\&\qquad+q(m-2+k)(\varepsilon+|\nabla u_\varepsilon|^2)^{\frac{q(m-2+k)-2}{2}}|D^{\boldsymbol{\nu}}u_\varepsilon|^{q}D^2u_\varepsilon\nabla u_\varepsilon.
    \end{split}
\end{equation}
We want to reconstruct the $L^2$-norm of \eqref{gradiente_di_g} from \eqref{eq:mlinearizzato}.
In \eqref{eq:mlinearizzato}, the only summands in which derivatives of order $m+1$ appear quadratically are those involving $J_1$ and given by the cases $\boldsymbol\gamma = 0$ and $\boldsymbol\gamma = \boldsymbol\nu$.  For this reason we split the sum as follows
\begin{equation}\label{santuzzo}
    \begin{split}
         &I_1+I_2+R:=\int_{B_{2R}} (\varepsilon+|\nabla u_\varepsilon|^2)^{\frac{p-2}{2}}(D^{\boldsymbol{\nu}}(\nabla u_\varepsilon), \nabla\varphi)\,dx
         \\&
        \quad \ +\int_{B_{2R}} \sum_{r=1}^ {|\boldsymbol{\nu}|} D^{r} h \sum_{s=1}^{|\boldsymbol{\nu}|} \sum_{p_s(\boldsymbol{\nu},r)}(\boldsymbol{\nu}!) \prod_{j=1}^{s} \frac{\left(D^{\boldsymbol{l}_j}(|\nabla u_\varepsilon|^2)\right)^{{k}_j}}{k_j!(\boldsymbol{l_j}!)^{k_j}} (\nabla u_\varepsilon, \nabla \varphi)\,dx
        \\&\quad \ +\sum_{0\neq \boldsymbol{\gamma}< \boldsymbol{\nu}} \int_{B_{2R}} \tbinom{\boldsymbol{\nu}}{\boldsymbol{\gamma}}  \sum_{r=1}^ {|\boldsymbol{\gamma}|} D^{r} h \sum_{s=1}^{|\boldsymbol{\gamma}|} \sum_{p_s(\boldsymbol{\gamma},r)}(\boldsymbol{\gamma}!) \prod_{j=1}^{s} \frac{\left(D^{\boldsymbol{l}_j}(|\nabla u_\varepsilon|^2)\right)^{{k}_j}}{k_j!(\boldsymbol{l_j}!)^{k_j}} (D^{\boldsymbol{\nu}-\boldsymbol{\gamma}}(\nabla u_\varepsilon), \nabla\varphi)\,dx 
        \\&\qquad\qquad\quad \ = \int_{B_{2R}} D^{\boldsymbol{\nu}} (f) \varphi \,dx.
    \end{split}
\end{equation}
Moreover, by using \eqref{gradiente_test}, we rewrite the term  $I_1$ 
\begin{equation}\label{termineI_1}
    \begin{split}
        &I_1=(2q-1)\int_{B_{2R}} (\varepsilon+|\nabla u_\varepsilon|^2)^{(m-2+k)q}|D^{\boldsymbol\nu} u_{\varepsilon}|^{2q-2}|D^{\boldsymbol\nu}(\nabla u_{\varepsilon})|^{2}\psi^2 \,dx+R_{I_1},
\end{split}
\end{equation}
where 
\begin{equation}\label{termineR_I_1}
    \begin{split}
        R_{I_1}:=\int_{B_{2R}} (\varepsilon+|\nabla u_\varepsilon|^2)^{\frac{p-2}{2}}(D^{\boldsymbol{\nu}}(\nabla u_\varepsilon), J_2+J_3)\,dx.
    \end{split}
\end{equation}
 Consider now the term $I_2$. By separating the cases $s=1$ and $s\neq 1$, we get
\begin{equation}\label{seperare_casi}
    \begin{split}
  I_2&=\int_{B_{2R}} \sum_{r=1}^ {|\boldsymbol{\nu}|} D^{r} h  \sum_{p_1(\boldsymbol{\nu},r)}(\boldsymbol{\nu}!)  \frac{\left(D^{\boldsymbol{l}_1}(|\nabla u_\varepsilon|^2)\right)^{{k}_1}}{k_1!(\boldsymbol{l_1}!)^{k_1}} (\nabla u_\varepsilon, J_1+J_2+J_3)\,dx
  \\&\qquad+\int_{B_{2R}} \sum_{r=1}^ {|\boldsymbol{\nu}|} D^{r} h \sum_{s=2}^{|\boldsymbol{\nu}|} \sum_{p_s(\boldsymbol{\nu},r)}(\boldsymbol{\nu}!) \prod_{j=1}^{s} \frac{\left(D^{\boldsymbol{l}_j}(|\nabla u_\varepsilon|^2)\right)^{{k}_j}}{k_j!(\boldsymbol{l_j}!)^{k_j}} (\nabla u_\varepsilon, J_1+J_2+J_3)\,dx,    
    \end{split}
\end{equation}
where  
\begin{equation*}
    p_1(\boldsymbol{\nu},{r})= \Big\{({k}_1;\boldsymbol{l}_1):  0\prec \boldsymbol{l}_1;\quad k_1>0;\quad k_1 ={r},\quad  k_1 \boldsymbol{l}_1 = \boldsymbol{\nu}\Big\}.
\end{equation*}
Once again, considering separately the case $r = 1$ and the remaining ones in the first term of the previous equation, and since $ p_1(\boldsymbol{\nu},1)=(1,\boldsymbol{\nu})$, we obtain
\begin{equation*}\label{eq:I_31}
    \begin{split}
  I_2&=\int_{B_{2R}}  D^{1} h \cdot\left(D^{\boldsymbol{\nu}}(|\nabla u_\varepsilon|^2)\right) (\nabla u_\varepsilon, J_1+J_2+J_3)\,dx
  \\&\qquad+\int_{B_{2R}} \sum_{r=2}^ {|\boldsymbol{\nu}|} D^{r} h  \sum_{p_1(\boldsymbol{\nu},r)}(\boldsymbol{\nu}!)  \frac{\left(D^{\boldsymbol{l}_1}(|\nabla u_\varepsilon|^2)\right)^{{k}_1}}{k_1!(\boldsymbol{l_1}!)^{k_1}} (\nabla u_\varepsilon, J_1+J_2+J_3)\,dx
  \\&\qquad+\int_{B_{2R}} \sum_{r=1}^ {|\boldsymbol{\nu}|} D^{r} h \sum_{s=2}^{|\boldsymbol{\nu}|} \sum_{p_s(\boldsymbol{\nu},r)}(\boldsymbol{\nu}!) \prod_{j=1}^{s} \frac{\left(D^{\boldsymbol{l}_j}(|\nabla u_\varepsilon|^2)\right)^{{k}_j}}{k_j!(\boldsymbol{l_j}!)^{k_j}} (\nabla u_\varepsilon, J_1+J_2+J_3)\,dx.    
    \end{split}
\end{equation*}
Using the Leibniz rule \eqref{Leibiniz_norma} with $\boldsymbol{\alpha}=\boldsymbol{\nu}$, we infer that

\begin{equation}\label{RestoI_2}
    \begin{split}
  I_2&=2\int_{B_{2R}}  D^{1} h \cdot   (D^{\boldsymbol{\nu}}(\nabla u_\varepsilon),\nabla u_\varepsilon) (\nabla u_\varepsilon, J_1+J_2+J_3)\,dx
  \\&\qquad+\int_{B_{2R}}  D^{1} h \cdot\sum_{\boldsymbol{0}\neq \boldsymbol{\mu}< \boldsymbol{\nu}}  \tbinom{\boldsymbol{\nu}}{\boldsymbol{\mu}}\,  (D^{\boldsymbol{\nu}-\boldsymbol{\mu}}(\nabla u_\varepsilon),D^{\boldsymbol{\mu}}(\nabla u_\varepsilon)) (\nabla u_\varepsilon, J_1+J_2+J_3)\,dx
  \\&\qquad+\int_{B_{2R}} \sum_{r=2}^ {|\boldsymbol{\nu}|} D^{r} h  \sum_{p_1(\boldsymbol{\nu},r)}(\boldsymbol{\nu}!)  \frac{\left(D^{\boldsymbol{l}_1}(|\nabla u_\varepsilon|^2)\right)^{{k}_1}}{k_1!(\boldsymbol{l_1}!)^{k_1}} (\nabla u_\varepsilon, J_1+J_2+J_3)\,dx
  \\&\qquad+\int_{B_{2R}} \sum_{r=1}^ {|\boldsymbol{\nu}|} D^{r} h \sum_{s=2}^{|\boldsymbol{\nu}|} \sum_{p_s(\boldsymbol{\nu},r)}(\boldsymbol{\nu}!) \prod_{j=1}^{s} \frac{\left(D^{\boldsymbol{l}_j}(|\nabla u_\varepsilon|^2)\right)^{{k}_j}}{k_j!(\boldsymbol{l_j}!)^{k_j}} (\nabla u_\varepsilon, J_1+J_2+J_3)\,dx.
    \end{split}
\end{equation}
In order to isolate the addend in which the derivatives of order $m+1$ appear quadratically, we write
\begin{equation}\label{termineI_2}
    I_2=2\int_{B_{2R}}  D^{1} h \cdot   (D^{\boldsymbol{\nu}}(\nabla u_\varepsilon),\nabla u_\varepsilon) (\nabla u_\varepsilon, J_1)\,dx+R_{I_2}^{m+1}+R_{I_2}^{\leq m},
\end{equation}
where we defined the integrals depending linearly on $D^{\boldsymbol{\nu}}(\nabla u_\varepsilon)$
\begin{equation}\label{Restosarracino}
    \begin{split}
  R_{I_2}^{m+1}&:=2\int_{B_{2R}}  D^{1} h \cdot   (D^{\boldsymbol{\nu}}(\nabla u_\varepsilon),\nabla u_\varepsilon) (\nabla u_\varepsilon, J_2+J_3)\,dx
  \\&\qquad+\int_{B_{2R}}  D^{1} h \cdot\sum_{\boldsymbol{0}\neq \boldsymbol{\mu}< \boldsymbol{\nu}}  \tbinom{\boldsymbol{\nu}}{\boldsymbol{\mu}}\,  (D^{\boldsymbol{\nu}-\boldsymbol{\mu}}(\nabla u_\varepsilon),D^{\boldsymbol{\mu}}(\nabla u_\varepsilon)) (\nabla u_\varepsilon, J_1)\,dx
  \\&\qquad+\int_{B_{2R}} \sum_{r=2}^ {|\boldsymbol{\nu}|} D^{r} h  \sum_{p_1(\boldsymbol{\nu},r)}(\boldsymbol{\nu}!)  \frac{\left(D^{\boldsymbol{l}_1}(|\nabla u_\varepsilon|^2)\right)^{{k}_1}}{k_1!(\boldsymbol{l_1}!)^{k_1}} (\nabla u_\varepsilon, J_1)\,dx
  \\&\qquad+\int_{B_{2R}} \sum_{r=1}^ {|\boldsymbol{\nu}|} D^{r} h \sum_{s=2}^{|\boldsymbol{\nu}|} \sum_{p_s(\boldsymbol{\nu},r)}(\boldsymbol{\nu}!) \prod_{j=1}^{s} \frac{\left(D^{\boldsymbol{l}_j}(|\nabla u_\varepsilon|^2)\right)^{{k}_j}}{k_j!(\boldsymbol{l_j}!)^{k_j}} (\nabla u_\varepsilon, J_1)\,dx.
    \end{split}
\end{equation} 
Using \eqref{santuzzo}, together with \eqref{termineI_1}, \eqref{RestoI_2} and \eqref{termineI_2}, we obtain 
\begin{equation*}
    \begin{split}
        &(2q-1)\int_{B_{2R}} (\varepsilon+|\nabla u_\varepsilon|^2)^{(m-2+k)q}|D^{\boldsymbol\nu} u_{\varepsilon}|^{2q-2}|D^{\boldsymbol\nu}(\nabla u_{\varepsilon})|^{2}\psi^2 \,dx
        \\&\quad\quad +2\int_{B_{2R}}  D^{1} h \cdot   (D^{\boldsymbol{\nu}}(\nabla u_\varepsilon),\nabla u_\varepsilon) (\nabla u_\varepsilon, J_1)\,dx
        \\&\qquad\quad \leq \int_{B_{2R}} |D^{\boldsymbol{\nu}} (f) \varphi| \,dx+|R_{I_1}|+|R_{I_2}^{m+1}|+|R_{I_2}^{\le m}|+|R|.
    \end{split}
\end{equation*}
Developing $D^1 h$, see \eqref{definizione_di_h}, using the expression of $J_1$ (see~\eqref{gradiente_test}), we infer that

\begin{equation}\label{stimaprimaimportante}
    \begin{split}
      &(2q-1)\min\{1,(p-1)\}\int_{B_{2R}} (\varepsilon+|\nabla u_\varepsilon|^2)^{(m-2+k)q}|D^{\boldsymbol\nu}u_{\varepsilon}|^{2q-2}| D^{\boldsymbol\nu}(\nabla u_{\varepsilon})|^2\psi^2\,dx 
      \\&\qquad\qquad\leq \int_{B_{2R}} |D^{\boldsymbol{\nu}} (f) \varphi| \,dx+|R_{I_1}|+|R_{I_2}^{m+1}|+|R_{I_2}^{\le m}|+|R|,
    \end{split}
\end{equation}
where we used the fact that for $p > 2$ the second term on the left-hand side of the previous inequality can be neglected, while for $p < 2$ the estimate follows from the Cauchy–Schwarz inequality.\\
Recall the definition of $\nabla \varphi$ in \eqref{gradiente_test}, and set $R:=R(J_1)+R(J_2)+R(J_3)$ in \eqref{santuzzo}. Now, in all the terms in $R_{I_1},$ $R_{I_2}^{m+1}$ and $R(J_1)$  it appears $D^{\boldsymbol{\nu}}(\nabla u_\varepsilon)$ linearly under the sign of integral. We shall estimate from above each one of these contributions by quantities of the form 
\begin{equation*}
    \mathtt a\int_{B_{2R}} (\varepsilon+|\nabla u_\varepsilon|^2)^{(m-2+k)q}|D^{\boldsymbol\nu}u_{\varepsilon}|^{2q-2}| D^{\boldsymbol\nu}(\nabla u_{\varepsilon})|^2\psi^2\,dx+R_{\leq m},
\end{equation*}
where $\mathtt a\ll 1$ and $R_{\le m}$ are some integrals of functions on which it falls a derivative of order at most $m$.\\
We begin with $R_{I_1}$, see \eqref{termineR_I_1}. Recall \eqref{eq:mlinearizzato}, by using a weighted Young inequality we deduce
\begin{equation}\label{termine3piu4}
    \begin{split}
        |R_{I_1}|&=\Big|(2(m-2+k)q+2-p)\int_{B_{2R}}(\varepsilon+|\nabla u_\varepsilon|^2)^{(m-2+k)q-1}\\&\quad\qquad\qquad\qquad\qquad\qquad\qquad\times |D^{\boldsymbol\nu}u_{\varepsilon}|^{2q-2}D^{\boldsymbol\nu}u_{\varepsilon} (D^{\boldsymbol{\nu}}(\nabla u_\varepsilon),D^2u_\varepsilon\nabla u_\varepsilon)\psi^2\,dx
        \\&\qquad+2\int_{B_{2R}}(\varepsilon+|\nabla u_\varepsilon|^2)^{(m-2+k)q}|D^{\boldsymbol\nu}u_{\varepsilon}|^{2q-2}D^{\boldsymbol\nu}u_{\varepsilon} (D^{\boldsymbol{\nu}}(\nabla u_\varepsilon),\nabla \psi)\psi\,dx\Big|
        \\&\leq (q+1)\theta\int_{B_{2R}} (\varepsilon+|\nabla u_\varepsilon|^2)^{(m-2+k)q}|D^{\boldsymbol\nu}u_{\varepsilon}|^{2q-2}| D^{\boldsymbol\nu}(\nabla u_{\varepsilon})|^2\psi^2\,dx+R_{I_1}^{\le m},
        \end{split}
\end{equation}
where $\theta\ll 1$ will be chosen later and 
\begin{equation*}
    \begin{split}
R_{I_1}^{\le m}&:=qC(p,k,m,\theta)\int_{B_{2R}}(\varepsilon+|\nabla u_\varepsilon|^2)^{(m-2+k)q-1}|D^{\boldsymbol\nu}u_{\varepsilon}|^{2q}|D^2u_\varepsilon|^2\psi^2\,dx
        \\&\qquad\quad+C(p,\theta)\int_{B_{2R}}(\varepsilon+|\nabla u_\varepsilon|^2)^{(m-2+k)q}|D^{\boldsymbol\nu}u_{\varepsilon}|^{2q}|\nabla \psi|^2\,dx, \end{split}
\end{equation*}
where $C(p,\theta), C(p,k,m,\theta)$ are positive constants.\\
We now pass to $R_{I_2}^{m+1}$, see \eqref{Restosarracino}. Recalling the definition of $h$, and by the weighted Young inequality, we have
\begin{equation}\label{termine1}
    \begin{split}
        &\left |2\int_{B_{2R}} D^{1} h \cdot(D^{\boldsymbol{\nu}}(\nabla u_\varepsilon),\nabla u_\varepsilon)  (\nabla u_\varepsilon,J_2)\,dx\right|
        \\&\qquad=\Big|(2(m-2+k)q+2-p)(p-2)\int_{B_{2R}}(\varepsilon+|\nabla u_\varepsilon|^2)^{(m-2+k)q-2}|D^{\boldsymbol\nu}u_{\varepsilon}|^{2q-2}D^{\boldsymbol\nu}u_{\varepsilon}\\&\qquad\qquad\qquad\qquad\qquad\qquad\qquad\qquad\qquad\times(D^{\boldsymbol\nu}(\nabla u_\varepsilon),\nabla u_\varepsilon)(\nabla u_\varepsilon,D^2u_\varepsilon\nabla u_\varepsilon)\psi^2\,dx \Big|
        \\& \qquad\le q\theta\int_{B_{2R}} (\varepsilon+|\nabla u_\varepsilon|^2)^{(m-2+k)q}|D^{\boldsymbol\nu}u_{\varepsilon}|^{2q-2}| D^{\boldsymbol\nu}(\nabla u_{\varepsilon})|^2\psi^2\,dx+R_{I_2,1}^{m+1\rightsquigarrow m},
        \end{split}
\end{equation}
where 
\begin{equation*}
    R_{I_2,1}^{m+1\rightsquigarrow m}:=qC(p,k,m,\theta)\int_{B_{2R}}(\varepsilon+|\nabla u_\varepsilon|^2)^{(m-2+k)q-1}|D^{\boldsymbol\nu}u_{\varepsilon}|^{2q}|D^2u_\varepsilon|^2\psi^2\,dx.
\end{equation*}
Similarly, we deduce
\begin{equation}\label{termine2}
    \begin{split}
        &2\int_{B_{2R}} D^{1} h \cdot(D^{\boldsymbol{\nu}}(\nabla u_\varepsilon),\nabla u_\varepsilon)  (\nabla u_\varepsilon,J_3)\,dx
        \\&\quad=2(p-2)\int_{B_{2R}}(\varepsilon+|\nabla u_\varepsilon|^2)^{(m-2+k)q-1}|D^{\boldsymbol\nu}u_{\varepsilon}|^{2q-2}D^{\boldsymbol\nu}u_{\varepsilon}\\&\ \qquad\qquad\qquad\qquad\qquad\times(D^{\boldsymbol\nu}(\nabla u_\varepsilon),\nabla u_\varepsilon)(\nabla u_\varepsilon,\nabla \psi)\psi\,dx 
        \\& \quad\le \theta\int_{B_{2R}} (\varepsilon+|\nabla u_\varepsilon|^2)^{(m-2+k)q}|D^{\boldsymbol\nu}u_{\varepsilon}|^{2q-2}| D^{\boldsymbol\nu}(\nabla u_{\varepsilon})|^2\psi^2\,dx+R_{I_2,2}^{m+1\rightsquigarrow m},
        \end{split}
\end{equation}
where 
\begin{equation*}
    R_{I_2,2}^{m+1\rightsquigarrow m}:=C(p,\theta)\int_{B_{2R}}(\varepsilon+|\nabla u_\varepsilon|^2)^{(m-2+k)q}|D^{\boldsymbol\nu}u_{\varepsilon}|^{2q}|\nabla \psi|^2\,dx.
\end{equation*}
In a similar way, we obtain 
\begin{equation}\label{termine5}
    \begin{split}
        &\int_{B_{2R}}  D^{1} h \cdot\sum_{\boldsymbol{0}\neq \boldsymbol{\mu}< \boldsymbol{\nu}}  \tbinom{\boldsymbol{\nu}}{\boldsymbol{\mu}}\,  (D^{\boldsymbol{\nu}-\boldsymbol{\mu}}(\nabla u_\varepsilon),D^{\boldsymbol{\mu}}(\nabla u_\varepsilon)) (\nabla u_\varepsilon, J_1)\,dx
        \\&\qquad=(2q-1)((p-2)/2) \int_{B_{2R}} (\varepsilon+|\nabla u_\varepsilon|^2)^{(m-2+k)q-1}|D^{\boldsymbol\nu}u_{\varepsilon}|^{2q-2}
        \\&\qquad\qquad\times \sum_{\boldsymbol{0}\neq \boldsymbol{\mu}< \boldsymbol{\nu}}  \tbinom{\boldsymbol{\nu}}{\boldsymbol{\mu}}\,  (D^{\boldsymbol{\nu}-\boldsymbol{\mu}}(\nabla u_\varepsilon),D^{\boldsymbol{\mu}}(\nabla u_\varepsilon))(D^{\boldsymbol{\nu}}(\nabla u_\varepsilon),\nabla u_\varepsilon)\psi^2\,dx
        \\&\qquad \leq q\theta\int_{B_{2R}} (\varepsilon+|\nabla u_\varepsilon|^2)^{(m-2+k)q}|D^{\boldsymbol\nu}u_{\varepsilon}|^{2q-2}| D^{\boldsymbol\nu}(\nabla u_{\varepsilon})|^2\psi^2\,dx+R_{I_2,3}^{m+1\rightsquigarrow m},
        \end{split}
\end{equation}
where
\begin{equation*}
\begin{split}
    R_{I_2,3}^{m+1\rightsquigarrow m}&:=qC(p,|\boldsymbol{\nu}|,\theta)\int_{B_{2R}} (\varepsilon+|\nabla u_\varepsilon|^2)^{(m-2+k)q-1}|D^{\boldsymbol\nu}u_{\varepsilon}|^{2q-2}\\&\qquad\qquad\qquad\times \sum_{\boldsymbol{0}\neq \boldsymbol{\mu}< \boldsymbol{\nu}}  |D^{\boldsymbol{\nu}-\boldsymbol{\mu}}(\nabla u_\varepsilon)|^2|D^{\boldsymbol{\mu}}(\nabla u_\varepsilon))|^2\psi^2\,dx,
    \end{split}
\end{equation*}
where $C(p,|\boldsymbol{\nu}|,\theta)$ is a positive constant. Recalling the definition of $h$, in particular see \eqref{derivate di h}, and proceeding as in other cases, we deduce
\begin{equation}\label{termine8}
\begin{split}
    &\int_{B_{2R}} \sum_{r=2}^ {|\boldsymbol{\nu}|} D^{r} h  \sum_{p_1(\boldsymbol{\nu},r)}(\boldsymbol{\nu}!)  \frac{\left(D^{\boldsymbol{l}_1}(|\nabla u_\varepsilon|^2)\right)^{{k}_1}}{k_1!(\boldsymbol{l_1}!)^{k_1}} (\nabla u_\varepsilon, J_1)\,dx
    \\&\qquad=(2q-1)\int_{B_{2R}} \sum_{r=2}^ {|\boldsymbol{\nu}|} D^{r} h  \sum_{p_1(\boldsymbol{\nu},r)}(\boldsymbol{\nu}!)  \frac{\left(D^{\boldsymbol{l}_1}(|\nabla u_\varepsilon|^2)\right)^{{k}_1}}{k_1!(\boldsymbol{l_1}!)^{k_1}}\\&\qquad\qquad\times (\nabla u_\varepsilon,D^{\boldsymbol\nu}(\nabla u_{\varepsilon})) (\varepsilon+|\nabla u_\varepsilon|^2)^{\frac{2(m-2+k)q+2-p}{2}}|D^{\boldsymbol\nu}u_{\varepsilon}|^{2q-2}\psi^2\,dx
    \\&\qquad\le q\theta\int_{B_{2R}} (\varepsilon+|\nabla u_\varepsilon|^2)^{(m-2+k)q}|D^{\boldsymbol\nu}u_{\varepsilon}|^{2q-2}| D^{\boldsymbol\nu}(\nabla u_{\varepsilon})|^2\psi^2\,dx+R_{I_2,4}^{m+1\rightsquigarrow m},
   \end{split}    
\end{equation}
where
\begin{equation*}
    \begin{split}
R_{I_2,4}^{m+1\rightsquigarrow m}&:=qC(p,|\boldsymbol{\nu}|,\theta)\int_{B_{2R}} \sum_{r=2}^ {|\boldsymbol{\nu}|} (\varepsilon+|\nabla u_\varepsilon|^2)^{(m-2+k)q+1-2r}  \sum_{p_1(\boldsymbol{\nu},r)}  \left(D^{\boldsymbol{l}_1}(|\nabla u_\varepsilon|^2)\right)^{2{k}_1}\\&\quad\qquad\qquad\qquad\times  |D^{\boldsymbol\nu}u_{\varepsilon}|^{2q-2}\psi^2\,dx.
    \end{split}
\end{equation*}
For the last contribution of $R_{I_2}^{m+1}$, see \eqref{Restosarracino}, we have the following 
\begin{equation}\label{termine11}
\begin{split}
    &\int_{B_{2R}} \sum_{r=1}^ {|\boldsymbol{\nu}|} D^{r} h \sum_{s=2}^{|\boldsymbol{\nu}|} \sum_{p_s(\boldsymbol{\nu},r)}(\boldsymbol{\nu}!) \prod_{j=1}^{s} \frac{\left(D^{\boldsymbol{l}_j}(|\nabla u_\varepsilon|^2)\right)^{{k}_j}}{k_j!(\boldsymbol{l_j}!)^{k_j}} (\nabla u_\varepsilon, J_1)\,dx
    \\&\qquad=(2q-1)\int_{B_{2R}} \sum_{r=1}^ {|\boldsymbol{\nu}|} D^{r} h \sum_{s=2}^{|\boldsymbol{\nu}|} \sum_{p_s(\boldsymbol{\nu},r)}(\boldsymbol{\nu}!) \prod_{j=1}^{s} \frac{\left(D^{\boldsymbol{l}_j}(|\nabla u_\varepsilon|^2)\right)^{{k}_j}}{k_j!(\boldsymbol{l_j}!)^{k_j}}\\& \qquad\qquad\qquad\times (\nabla u_\varepsilon, D^{\boldsymbol\nu}(\nabla u_{\varepsilon})) (\varepsilon+|\nabla u_\varepsilon|^2)^{\frac{2(m-2+k)q+2-p}{2}}|D^{\boldsymbol\nu}u_{\varepsilon}|^{2q-2}\psi^2\,dx
    \\&\qquad\le q\theta\int_{B_{2R}} (\varepsilon+|\nabla u_\varepsilon|^2)^{(m-2+k)q}|D^{\boldsymbol\nu}u_{\varepsilon}|^{2q-2}| D^{\boldsymbol\nu}(\nabla u_{\varepsilon})|^2\psi^2\,dx+R_{I_2,5}^{m+1\rightsquigarrow m},
\end{split}    
\end{equation}
where
\begin{equation*}
    \begin{split}
R_{I_2,5}^{m+1\rightsquigarrow m}&:=qC(p,|\boldsymbol{\nu}|,\theta)\int_{B_{2R}} \sum_{r=1}^ {|\boldsymbol{\nu}|} (\varepsilon+|\nabla u_\varepsilon|^2)^{(m-2+k)q+1-2r}  \sum_{s=2}^{|\boldsymbol{\nu}|} \sum_{p_s(\boldsymbol{\nu},r)} \prod_{j=1}^{s} {\left(D^{\boldsymbol{l}_j}(|\nabla u_\varepsilon|^2)\right)^{2{k}_j}}\\&\qquad\qquad\qquad\qquad\times  |D^{\boldsymbol\nu}u_{\varepsilon}|^{2q-2}\psi^2\,dx.
    \end{split}
\end{equation*}
Concerning $R(J_1)$, see \eqref{santuzzo}, we have
\begin{equation}\label{termine14}
    \begin{split}
         &\sum_{0\neq \boldsymbol{\gamma}< \boldsymbol{\nu}} \int_{B_{2R}} \tbinom{\boldsymbol{\nu}}{\boldsymbol{\gamma}}  \sum_{r=1}^ {|\boldsymbol{\gamma}|} D^{r} h \sum_{s=1}^{|\boldsymbol{\gamma}|} \sum_{p_s(\boldsymbol{\gamma},r)}(\boldsymbol{\gamma}!) \prod_{j=1}^{s} \frac{\left(D^{\boldsymbol{l}_j}(|\nabla u_\varepsilon|^2)\right)^{{k}_j}}{k_j!(\boldsymbol{l_j}!)^{k_j}} (D^{\boldsymbol{\nu}-\boldsymbol{\gamma}}(\nabla u_\varepsilon), J_1)\,dx
         \\&\qquad=(2q-1)\sum_{0\neq \boldsymbol{\gamma}< \boldsymbol{\nu}} \int_{B_{2R}} \tbinom{\boldsymbol{\nu}}{\boldsymbol{\gamma}}  \sum_{r=1}^ {|\boldsymbol{\gamma}|} D^{r} h \sum_{s=1}^{|\boldsymbol{\gamma}|} \sum_{p_s(\boldsymbol{\gamma},r)}(\boldsymbol{\gamma}!) \prod_{j=1}^{s} \frac{\left(D^{\boldsymbol{l}_j}(|\nabla u_\varepsilon|^2)\right)^{{k}_j}}{k_j!(\boldsymbol{l_j}!)^{k_j}} \\&\qquad\qquad\qquad\times (D^{\boldsymbol{\nu}-\boldsymbol{\gamma}}(\nabla u_\varepsilon), D^{\boldsymbol\nu}(\nabla u_{\varepsilon})) (\varepsilon+|\nabla u_\varepsilon|^2)^{\frac{2(m-2+k)q+2-p}{2}}|D^{\boldsymbol\nu}u_{\varepsilon}|^{2q-2}\psi^2\,dx
          \\&\qquad\le q\theta\int_{B_{2R}} (\varepsilon+|\nabla u_\varepsilon|^2)^{(m-2+k)q}|D^{\boldsymbol\nu}u_{\varepsilon}|^{2q-2}| D^{\boldsymbol\nu}(\nabla u_{\varepsilon})|^2\psi^2\,dx+R(J_1)^{\le m},
    \end{split}
\end{equation}
where 
\begin{equation*}
\begin{split}
R(J_1)^{\le m}&:=qC(p,|\boldsymbol{\nu}|,\theta)\int_{B_{2R}}\sum_{0\neq \boldsymbol{\gamma}< \boldsymbol{\nu}} \sum_{r=1}^ {|\boldsymbol{\nu}|} (\varepsilon+|\nabla u_\varepsilon|^2)^{(m-2+k)q-2r} \sum_{s=1}^{|\boldsymbol{\gamma}|} \sum_{p_s(\boldsymbol{\gamma},r)} \prod_{j=1}^{s} {\left(D^{\boldsymbol{l}_j}(|\nabla u_\varepsilon|^2)\right)^{2{k}_j}} \\&\qquad\qquad\qquad\qquad\times |D^{\boldsymbol{\nu}-\boldsymbol{\gamma}}(\nabla u_\varepsilon)|^2  |D^{\boldsymbol\nu}u_{\varepsilon}|^{2q-2}\psi^2\,dx.
\end{split}
\end{equation*}
By collecting all these estimates, see  \eqref{termine3piu4}, \eqref{termine1}, \eqref{termine2}, \eqref{termine5}, \eqref{termine8}, \eqref{termine11} and \eqref{termine14}, in inequality \eqref{stimaprimaimportante}, and by choosing $\theta:=\theta(p)$ sufficiently small, we infer that 

\begin{equation}\label{eq:mlinearizzato1}
    \begin{split}
       F&:=q\int_{B_{2R}} (\varepsilon+|\nabla u_\varepsilon|^2)^{(m-2+k)q}|D^{\boldsymbol\nu}u_{\varepsilon}|^{2q-2}| D^{\boldsymbol\nu}(\nabla u_{\varepsilon})|^2\psi^2\,dx
        \\& \ \lesssim R_{I_1}^{\le m}+
          R_{I_2}^{\le m}+R_{I_2}^{m+1\rightsquigarrow m}+R(J_1)^{\le m}+R(J_2)+R(J_3)+\int_{B_{2R}} |D^{\boldsymbol{\nu}} (f) \varphi | \,dx,
    \end{split}
\end{equation}
where we have set $R_{I_2}^{m+1\rightsquigarrow m}=\sum_{i=1}^5R_{I_2,i}^{m+1\rightsquigarrow m}$.
Note that $R_{I_2}^{\le m}+R(J_2)+R(J_3)$ is exactly the l.h.s. of inequality \eqref{pantusa}. By performing similar computation of Lemma \ref{pantusa1}, which we omit, one can prove that $R_{I_1}^{\le m}+R_{I_2}^{m+1\rightsquigarrow m}+R(J_1)^{\le m}$ is bounded from above by the r.h.s. of \eqref{pantusa1}. Assume for the moment that also $\int_{B_{2R}} |D^{\boldsymbol{\nu}} (f) \varphi | \,dx,$ is bounded from above by the r.h.s. of \eqref{pantusa1}. Altogether we proved that 

\begin{equation*}
    \sqrt{F}\lesssim \sqrt{q} \Big(\int_{B_{2R}}(\varepsilon+|\nabla u_\varepsilon|^2)^\frac{\zeta(m-2+k)(q-\hat s)}{2} |D^{\boldsymbol\nu}u_{\varepsilon}|^{\zeta(q-\hat s)}(\psi^\zeta + |\nabla \psi|^\zeta)\,dx\Big)^{\frac 1\zeta}.
\end{equation*}

Let $h$ and $h'$ be real numbers satisfying $h'<h< R$. Let $\psi(x)$ be the standard cut-off function such that $\psi(x)=1$ in the ball $B_{h'}$, $\psi(x) =0$ in the complement of the ball $B_h$, and $|\nabla \psi|\le 2/(h-h')$. Recalling \eqref{gradiente_di_g}, multiplying by $\sqrt{q}$ and summing the same quantity on both sides of the previous inequality, we obtain
\begin{equation}\label{radici2}
\begin{split}
    &\sqrt{q}\sqrt{ F}+q(m-2+k)\left(\int_{B_{h'}}(\varepsilon+|\nabla u_\varepsilon|^2)^{(m-2+k)q-1}|D^{\boldsymbol\nu}u_{\varepsilon}|^{2q}|D^2u_\varepsilon|^2\psi^2\,dx\right)^{\frac 12}+\|g_\varepsilon^q\|_{L^2(B_{h'})}
    \\&\quad\lesssim q\Big(\int_{B_{2R}}(\varepsilon+|\nabla u_\varepsilon|^2)^\frac{\zeta(m-2+k)(q-\hat s)}{2} |D^{\boldsymbol\nu}u_{\varepsilon}|^{\zeta(q-\hat s)}(\psi^\zeta + |\nabla \psi|^\zeta)\,dx\Big)^{\frac 1\zeta}
    \\&\qquad+q(m-2+k)\left(\int_{B_{h'}}(\varepsilon+|\nabla u_\varepsilon|^2)^{(m-2+k)q-1}|D^{\boldsymbol\nu}u_{\varepsilon}|^{2q}|D^2u_\varepsilon|^2\psi^2\,dx\right)^{\frac 12}+\|g_\varepsilon^q\|_{L^2(B_{h'})},
    \end{split}
\end{equation}
hence
\begin{equation}\label{pantusa11}
    \begin{split}
       &\|g_\varepsilon^q\|_{W^{1,2}(B_{h'})} \lesssim q\Big(\int_{B_{2R}}(\varepsilon+|\nabla u_\varepsilon|^2)^\frac{\zeta(m-2+k)(q-\hat s)}{2} |D^{\boldsymbol\nu}u_{\varepsilon}|^{\zeta(q-\hat s)}(\psi^\zeta + |\nabla \psi|^\zeta)\,dx\Big)^{\frac 1\zeta}\\
       &\quad+q(m-2+k)\left(\int_{B_{h'}}(\varepsilon+|\nabla u_\varepsilon|^2)^{(m-2+k)q-1}|D^{\boldsymbol\nu}u_{\varepsilon}|^{2q}|D^2u_\varepsilon|^2\psi^2\,dx\right)^{\frac 12}+\|g_\varepsilon^q\|_{L^2(B_{h'})}.
    \end{split}
\end{equation}
Reasoning as done in Lemma \ref{pantusa1}, we can estimate the second and the third summands in r.h.s. of \eqref{pantusa11} by the first one. Therefore by means of the Sobolev inequality we infer
\begin{equation}\label{stimona1}
    \|g_\varepsilon^q\|_{L^{2^*}(B_{h'})}\lesssim C_Sq\Big(\int_{B_{2R}}(\varepsilon+|\nabla u_\varepsilon|^2)^\frac{\zeta(m-2+k)(q-\hat s)}{2} |D^{\boldsymbol\nu}u_{\varepsilon}|^{\zeta(q-\hat s)}(\psi^\zeta + |\nabla \psi|^\zeta)\,dx\Big)^{\frac 1\zeta},
\end{equation}
where $C_S$ denotes the Sobolev constant.
 Recalling that $|\nabla \psi|\le C/(h-h')$, for some positive constant $C$, $\psi\leq 1$, we deduce \eqref{cane477}.\\
We conclude by estimating the term $\int_{B_{2R}} |D^{\boldsymbol{\nu}} (f) \varphi | \,dx$.
\begin{equation}\label{terms15}
    \begin{split}
\int_{B_{2R}} |D^{\boldsymbol{\nu}} (f) \varphi | \,dx&=\int_{B_{2R}} D^{\boldsymbol{\nu}} (f) (\varepsilon+|\nabla u_\varepsilon|^2)^{\frac{2(m-2+k)q+2-p}{2}}|D^{\boldsymbol\nu}u_{\varepsilon}|^{2q-1}\psi^2  \,dx
\\&=\int_{B_{h}}  (\varepsilon+|\nabla u_\varepsilon|^2)^{(m-2+k)q-|\boldsymbol{\nu}|+1-(|\boldsymbol{\nu}|-2)(\hat s-1)}|D^{\boldsymbol\nu}u_{\varepsilon}|^{2q-2\hat s}\\&\qquad\qquad\times   (\varepsilon+|\nabla u_\varepsilon|^2)^{\frac{|\boldsymbol{\nu}|+2-p}{2}}|D^{\boldsymbol{\nu}}(f)|\\&\qquad\qquad\times (\varepsilon+|\nabla u_\varepsilon|^2)^{{\frac{(|\boldsymbol{\nu}|-2)(2\hat s-1)}{2}}}|D^{\boldsymbol\nu}u_{\varepsilon}|^{2\hat{s}-1}\psi^2\,dx.
    \end{split}
\end{equation}
Note that $(m-2+k)q-|\boldsymbol{\nu}|+1-(|\boldsymbol{\nu}|-2)(\hat s-1)=(m-2+k)(q-\hat s)$.
Since $|\boldsymbol{\nu}|+2-p>0$ and since $\nabla u_{\varepsilon}$ is uniformly bounded in $L^{\infty}$ we infer that $(\varepsilon+|\nabla u_\varepsilon|^2)^{|\boldsymbol{\nu}|+2-p}$ is uniformly bounded in $L^{\infty}$.
From the latter facts, by using Holder's inequalities twice, first with exponents $(l,l/(l-1))$, and  then with exponents $\left(\frac{\zeta(l-1)}{2l},\frac{\zeta(l-1)}{l(\zeta-2)-\zeta}\right)$, we obtain
\begin{equation*}
    \begin{split}
    &\int_{B_{2R}} |D^{\boldsymbol{\nu}} (f) \varphi | \,dx\\
    &\qquad\leq C(p,n,R,\|f\|_{L^s(B_{2R})})
        \Big(\int_{B_h}|D^{\boldsymbol{\nu}}(f)|^l\,dx\Big)^{\frac{1}{l}}
        \\&\qquad\qquad\times\Big(\int_{B_{h}}  (\varepsilon+|\nabla u_\varepsilon|^2)^{(m-2+k)(q-\hat s)\frac{l}{l-1}}|D^{\boldsymbol\nu}u_{\varepsilon}|^{(2q-2\hat s)\frac{l}{l-1}}\\&\quad\qquad\qquad\qquad\times (\varepsilon+|\nabla u_\varepsilon|^2)^{{\frac{(|\boldsymbol{\nu}|-2)(2\hat s-1)}{2}}\frac{l-1}{l}}|D^{\boldsymbol\nu}u_{\varepsilon}|^{(2\hat{s}-1)\frac{l-1}{l}}\psi^{\frac{2l}{l-1}}\,dx\Big)^{\frac {l-1}{l}}
        \\&\qquad\leq C(p,n,R,\|f\|_{W^{m,l}(B_{2R})})\Big(\int_{B_{h}}  (\varepsilon+|\nabla u_\varepsilon|^2)^{\frac{\zeta (m-2+k)(q-\hat s)}{2}}|D^{\boldsymbol\nu}u_{\varepsilon}|^{\zeta(q-\hat s)} \psi^\zeta\,dx\Big)^{\frac 1\zeta}
        \\&\qquad\qquad\times \Big(\int_{B_{h}}(\varepsilon+|\nabla u_\varepsilon|^2)^{{\frac{(|\boldsymbol{\nu}|-2)(2\hat s-1)}{2}\frac{l\zeta}{l(\zeta-2)-\zeta}}}|D^{\boldsymbol\nu}u_{\varepsilon}|^{(2\hat{s}-1)\frac{l\zeta}{l(\zeta-2)-\zeta}}\,dx\Big)^{{\frac {l(\zeta-2)-\zeta}{l\zeta}}},
    \end{split}
\end{equation*}
which is the claimed bound upon using  Lemma \ref{sisso}. This concludes the proof.
\end{proof}

\begin{rem}\label{pluto}
The case $n=2$ can be treated in a completely analogous way. The only difference lies in the use of the Sobolev inequality to estimate from below the left-hand side of \eqref{pantusa11}. In this setting, one can exploit the embedding $W^{1,2}_{\mathrm{loc}}(\Omega)\hookrightarrow L^{r}_{\mathrm{loc}}(\Omega),
\ r>\frac{2l}{l-1}.$
The proof then proceeds by choosing the exponents such that $2<\frac{2l}{l-1}<\nu<r,$
which allows one to conclude as before. In particular, estimate \eqref{cane477} is recovered with the substitution $2^*\rightsquigarrow r$.
\end{rem}

\begin{lem}\label{Moserstorta}
 Let $\Omega\subset\R^n$ be a domain, with $n\ge 2$. Fix $0<R'<R$. Under the hypotheses of Proposition \ref{dinodino}, for any $R'< R$ there exists a positive constant $\mathcal{C}$ depending on $m,k,n,\zeta,l,R,R',p,\| \nabla u\|_{L^{\infty}(B_{2R})}$ and $\|f\|_{W^{m,l}( B_{2R})}$ such that
 \begin{equation}\label{frtm1}
\left\|(\varepsilon+|\nabla u_\varepsilon|^2)^{\frac{(m-2+k)}{2}}|D^{\boldsymbol{\nu}}u_\varepsilon|\right\|_{L^{\infty}(B_{R'})}\le \mathcal{C}.
\end{equation}
\end{lem}
\begin{proof}
    See \cite[Proposition 3.4]{IV} and the subsequent remark.
\end{proof}
We are now in position to prove Theorem \ref{INFINITO}.
\begin{proof}[Proof of Theorem \ref{INFINITO}]
Let us begin with the case $n=3$. Fix $x_0\in\Omega$ and let $R>0$ such that $B_{2R}:=B_{2R}(x_0)\subset\subset\Omega $ and let $R'<R$. Let $u_\varepsilon$ be a solution of the problem \eqref{eq:problregol}. We claim that
\begin{equation}\label{Final_conv254}
    \left\|(\varepsilon+|\nabla u_\varepsilon|^2)^\frac{k}{2}D^{\boldsymbol{\alpha}} \left((\varepsilon+|\nabla u_\varepsilon|^2)^\frac{m-2}{2}\nabla u_\varepsilon\right) \right\|_{L^\infty(B_{R'})} \leq \mathcal{C},
\end{equation}
with $\mathcal{C}$ independent of $\varepsilon.$ To prove the claim, we use an argument developed in the proof of Theorem \ref{teoprinc}.  Let $\boldsymbol{\alpha}$ be a multi-index of order $m-1$. Exploiting the Leibniz rule, we get:
\begin{equation}\label{L_conv250}
\begin{split}
    &(\varepsilon+|\nabla u_\varepsilon|^2)^\frac{k}{2}D^{\boldsymbol{\alpha}} \left((\varepsilon+|\nabla u_\varepsilon|^2)^\frac{m-2}{2}\nabla u_\varepsilon\right)\\
    &\qquad= (\varepsilon+|\nabla u_\varepsilon|^2)^\frac{k}{2}\sum_{\boldsymbol{\gamma}\leq \boldsymbol{\alpha}} \tbinom{\boldsymbol{\alpha}}{\boldsymbol{\gamma}} D^{\boldsymbol{\gamma}}\left((\varepsilon+|\nabla u_\varepsilon|^2)^\frac{m-2}{2}\right) D^{\boldsymbol{\alpha}-\boldsymbol{\gamma}}\left(\nabla u_\varepsilon\right)\\
    &\qquad=(\varepsilon+|\nabla u_\varepsilon|^2)^\frac{k}{2}\left((\varepsilon+|\nabla u_\varepsilon|^2)^\frac{m-2}{2}D^{\boldsymbol{\alpha}}(\nabla u_\varepsilon) + D^{\boldsymbol{\alpha}}\left((\varepsilon+|\nabla u_\varepsilon|^2)^\frac{m-2}{2}\right) \nabla u_\varepsilon\right)\\
    &\qquad\qquad + (\varepsilon+|\nabla u_\varepsilon|^2)^\frac{k}{2}\sum_{0<\boldsymbol{\gamma}< \boldsymbol{\alpha}} \tbinom{\boldsymbol{\alpha}}{\boldsymbol{\gamma}} D^{\boldsymbol{\gamma}}\left((\varepsilon+|\nabla u_\varepsilon|^2)^\frac{m-2}{2}\right) D^{\boldsymbol{\alpha}-\boldsymbol{\gamma}}\left(\nabla u_\varepsilon\right)\\
    &\qquad = : F_1+F_2+F_3,
\end{split}
\end{equation}
where the three cases
$\boldsymbol{\gamma} = \boldsymbol{0}$,
$\boldsymbol{\gamma} \neq \boldsymbol{\alpha}$,
and
$\boldsymbol{\gamma} = \boldsymbol{\alpha}$
are treated separately.\\
Now we prove that there exists a constant $\mathfrak{C}:=\mathfrak{C}(k,l,n,m)$ small enough such that if $|p-2|<\mathfrak{C}$, then the terms $F_1$, $F_2$ and $F_3$ are bounded in $L^{\infty}(B_{R'})$.\\
By Lemma \ref{Moserstorta}, we can estimate $F_1$ in the following way
\begin{equation}\label{F_1stima}
\big\|F_1\big\|_{L^{\infty}(B_{R'})} = \Big\|(\varepsilon+|\nabla u_\varepsilon|^2)^\frac{m-2+k}{2}D^{\boldsymbol{\alpha}}(\nabla u_\varepsilon)\Big\|_{L^{\infty}(B_{R'})} \leq C,
\end{equation}
where $C$ is a positive constant independent of $\varepsilon$. Now we apply the Faà di Bruno formula to estimate the term $F_2$ in the equality \eqref{L_conv250}. Therefore, we obtain
\begin{equation}\label{F_2stima}
\begin{split}
    F_2&=(\varepsilon+|\nabla u_\varepsilon|^2)^\frac{k}{2}D^{\boldsymbol{\alpha}}\left((\varepsilon+|\nabla u_\varepsilon|^2)^\frac{m-2}{2}\right)\nabla u_\varepsilon\\
    &=\sum_{r=1}^{|\boldsymbol{\alpha}|} c_1 (\varepsilon+|\nabla u_\varepsilon|^2)^{\frac{m-2+k}{2} -r} \sum_{s=1}^{|\boldsymbol{\alpha}|} \sum_{p_s(\boldsymbol{\alpha},{r})} \boldsymbol{\alpha}! \prod_{j=1}^s \frac{\left(D^{\boldsymbol{l}_j}(|\nabla u_\varepsilon|^2)\right)^{k_j}}{k_j!(\boldsymbol{l}_j!)^{k_j}}\nabla u_\varepsilon
\end{split}
\end{equation}
where $c_1:= \prod_{h=1}^r \frac{m-2r}{2}$ and
\begin{equation*}
\begin{split}
    p_s(\boldsymbol{\alpha},{r})= \Big\{({k}_1,...,{k}_s;\boldsymbol{l}_1,...,\boldsymbol{l}_s):  0\prec \boldsymbol{l}_1\prec\cdots \prec\boldsymbol{l}_s;k_i>0;\sum_{i=1}^s k_i ={r}, \sum_{i=1}^{s} k_i \boldsymbol{l}_i = \boldsymbol{\alpha}\Big\}.
\end{split}
\end{equation*}
By treating separately the case $r = s = 1$ and the remaining ones,
\eqref{F_2stima} becomes
\begin{equation}\label{F_2_stima2}
\begin{split}
    &(\varepsilon+|\nabla u_\varepsilon|^2)^\frac{k}{2}D^{\boldsymbol{\alpha}}\left((\varepsilon+|\nabla u_\varepsilon|^2)^\frac{m-2}{2}\right)\nabla u_\varepsilon\\
    &\qquad= \sum_{r=1}^{|\boldsymbol{\alpha}|} c_1 (\varepsilon+|\nabla u_\varepsilon|^2)^{\frac{m-2+k}{2} -r} \sum_{p_1(\boldsymbol{\alpha},{r})} \boldsymbol{\alpha}! \frac{\left(D^{\boldsymbol{l}_1}(|\nabla u_\varepsilon|^2)\right)^{k_1}}{k_1!(\boldsymbol{l}_1!)^{k_1}}\nabla u_\varepsilon\\
    &\qquad\quad + \sum_{r=1}^{|\boldsymbol{\alpha}|} c_1 (\varepsilon+|\nabla u_\varepsilon|^2)^{\frac{m-2+k}{2} -r} \sum_{s=2}^{|\boldsymbol{\alpha}|} \sum_{p_s(\boldsymbol{\alpha},{r})} \boldsymbol{\alpha}! \prod_{j=1}^s \frac{\left(D^{\boldsymbol{l}_j}(|\nabla u_\varepsilon|^2)\right)^{k_j}}{k_j!(\boldsymbol{l}_j!)^{k_j}}\nabla u_\varepsilon\\
    &\qquad = \frac{m-2}{2} (\varepsilon+|\nabla u_\varepsilon|^2)^{\frac{m-4+k}{2}}D^{\boldsymbol{\alpha}}(|\nabla u_\varepsilon|^2)\nabla u_\varepsilon \\
    &\qquad\quad+ \sum_{r=2}^{|\boldsymbol{\alpha}|} c_1 (\varepsilon+|\nabla u_\varepsilon|^2)^{\frac{m-2+k}{2} -r} \sum_{p_1(\boldsymbol{\alpha},{r})} \boldsymbol{\alpha}! \frac{\left(D^{\boldsymbol{l}_1}(|\nabla u_\varepsilon|^2)\right)^{k_1}}{k_1!(\boldsymbol{l}_1!)^{k_1}}\nabla u_\varepsilon\\
    &\qquad\quad + \sum_{r=1}^{|\boldsymbol{\alpha}|} c_1 (\varepsilon+|\nabla u_\varepsilon|^2)^{\frac{m-2+k}{2} -r} \sum_{s=2}^{|\boldsymbol{\alpha}|} \sum_{p_s(\boldsymbol{\alpha},{r})} \boldsymbol{\alpha}! \prod_{j=1}^s \frac{\left(D^{\boldsymbol{l}_j}(|\nabla u_\varepsilon|^2)\right)^{k_j}}{k_j!(\boldsymbol{l}_j!)^{k_j}}\nabla u_\varepsilon,
\end{split}
\end{equation}
where
   $ p_1(\boldsymbol{\alpha},{r})= \Big\{({k}_1;\boldsymbol{l}_1):  0\prec \boldsymbol{l}_1;\quad k_1 ={r},\quad  k_1 \boldsymbol{l}_1 = \boldsymbol{\alpha}\Big\}.$
By the Leibniz rule we have the following
\begin{equation*}
    D^{\boldsymbol{\alpha}}(|\nabla u_\varepsilon|^2)= 2 \left(\nabla u_\varepsilon,D^{\boldsymbol{\alpha}}(\nabla u_\varepsilon)\right) +\sum_{0<\boldsymbol{\mu}< \boldsymbol{\alpha}} \tbinom{\boldsymbol{\alpha}}{\boldsymbol{\mu}} \left(D^{\boldsymbol{\mu}}(\nabla u_\varepsilon),D^{\boldsymbol{\alpha}-\boldsymbol{\mu}}(\nabla u_\varepsilon)\right),
\end{equation*}
which, used in \eqref{F_2_stima2}, gives
\begin{equation}\label{F_2_sl}
\begin{split}
    F_2&=(\varepsilon+|\nabla u_\varepsilon|^2)^\frac{k}{2}D^{\boldsymbol{\alpha}}\left((\varepsilon+|\nabla u_\varepsilon|^2)^\frac{m-2}{2}\right)\nabla u_\varepsilon\\
    &= (m-2)(\varepsilon+|\nabla u_\varepsilon|^2)^{\frac{m-4+k}{2}}\left(D^{\boldsymbol{\alpha}}(\nabla u_\varepsilon),\nabla u_\varepsilon\right)\nabla u_\varepsilon\\
    &\quad+ \frac{m-2}{2} (\varepsilon+|\nabla u_\varepsilon|^2)^{\frac{m-4+k}{2}}\sum_{0<\boldsymbol{\mu}< \boldsymbol{\alpha}} \tbinom{\boldsymbol{\alpha}}{\boldsymbol{\mu}} \left(D^{\boldsymbol{\mu}}(\nabla u_\varepsilon),D^{\boldsymbol{\alpha}-\boldsymbol{\mu}}(\nabla u_\varepsilon)\right)\nabla u_\varepsilon\\
    &\quad+ \sum_{r=2}^{|\boldsymbol{\alpha}|} c_1 (\varepsilon+|\nabla u_\varepsilon|^2)^{\frac{m-2+k}{2} -r} \sum_{p_1(\boldsymbol{\alpha},{r})} \boldsymbol{\alpha}! \frac{\left(D^{\boldsymbol{l}_1}(|\nabla u_\varepsilon|^2)\right)^{k_1}}{k_1!(\boldsymbol{l}_1!)^{k_1}}\nabla u_\varepsilon\\
    &\quad + \sum_{r=1}^{|\boldsymbol{\alpha}|} c_1 (\varepsilon+|\nabla u_\varepsilon|^2)^{\frac{m-2+k}{2} -r} \sum_{s=2}^{|\boldsymbol{\alpha}|} \sum_{p_s(\boldsymbol{\alpha},{r})} \boldsymbol{\alpha}! \prod_{j=1}^s \frac{\left(D^{\boldsymbol{l}_j}(|\nabla u_\varepsilon|^2)\right)^{k_j}}{k_j!(\boldsymbol{l}_j!)^{k_j}}\nabla u_\varepsilon
    \\&=:J_1+\cdots +J_4.
\end{split}
\end{equation}
By \eqref{frtm1}, we deduce that the term $J_1$ is uniformly bounded in $L^{\infty}$. For the term $J_2$, since $m-3=(|\boldsymbol{\mu}|-1)+(|\boldsymbol{\alpha}|-|\boldsymbol{\mu}|-1)$, we deduce 

\begin{equation}\label{J_2stima}
    \begin{split}
        |J_2|&\le C(m)\sum_{0<\boldsymbol{\mu}< \boldsymbol{\alpha}}  (\varepsilon+|\nabla u_\varepsilon|^2)^{\frac{|\boldsymbol{\mu}|-1+k/2}{2}}D^{\boldsymbol{\mu}}(\nabla u_\varepsilon) (\varepsilon+|\nabla u_\varepsilon|^2)^{\frac{|\boldsymbol{\alpha}|-|\boldsymbol{\mu}|-1+k/2}{2}}D^{\boldsymbol{\alpha}-\boldsymbol{\mu}}(\nabla u_\varepsilon),
    \end{split}
\end{equation}
with $C(m)$ positive constant. By \eqref{frtm1} we get the uniform boundedness in $L^{\infty}$ for the term $J_2$.
Now we estimate the term $J_3$. Using the Leibniz rule for the term $D^{\boldsymbol{l}_1}(|\nabla u_\varepsilon|^2)$, and since $m-1-2r=k_1(|\boldsymbol{\mu}|-1)+k_1(|\boldsymbol{l}_1|-|\boldsymbol{\mu}|-1)$,  by \eqref{frtm1}, we deduce

\begin{equation}\label{J_3stima}
    \begin{split}
        J_3&= \sum_{r=2}^{|\boldsymbol{\alpha}|} c_1 (\varepsilon+|\nabla u_\varepsilon|^2)^{\frac{m-2+k}{2} -r} \sum_{p_1(\boldsymbol{\alpha},{r})} \boldsymbol{\alpha}! \frac{\left(D^{\boldsymbol{l}_1}(|\nabla u_\varepsilon|^2)\right)^{k_1}}{k_1!(\boldsymbol{l}_1!)^{k_1}}\nabla u_\varepsilon
        \\&\le C(m)(\varepsilon+|\nabla u_\varepsilon|^2)^{\frac{k}{2}}\sum_{r=2}^{|\boldsymbol{\alpha}|}  \sum_{p_1(\boldsymbol{\alpha},{r})} 
    \sum_{\boldsymbol{\mu}\leq \boldsymbol{l}_1}  (\varepsilon+|\nabla u_\varepsilon|^2)^{\frac{k_1(|\boldsymbol{\mu}|-1)}{2}}\Big|D^{\boldsymbol{\mu}}(\nabla u_\varepsilon)\Big|^{k_1}\times
    \\&\qquad\qquad\qquad\qquad \times(\varepsilon+|\nabla u_\varepsilon|^2)^{\frac{k_1(|\boldsymbol{l}_1|-|\boldsymbol{\mu}|-1)}{2}}\Big|D^{\boldsymbol{l}_1-\boldsymbol{\mu}}(\nabla u_\varepsilon)\Big|^{k_1}
   \\& = C(m)\sum_{r=2}^{|\boldsymbol{\alpha}|}  \sum_{p_1(\boldsymbol{\alpha},{r})} 
    \sum_{\boldsymbol{\mu}\leq \boldsymbol{l}_1}  (\varepsilon+|\nabla u_\varepsilon|^2)^{\frac{k_1(|\boldsymbol{\mu}|-1)+k/2}{2}}\Big|D^{\boldsymbol{\mu}}(\nabla u_\varepsilon)\Big|^{k_1}\times
    \\&\qquad\qquad\qquad\qquad \times(\varepsilon+|\nabla u_\varepsilon|^2)^{\frac{k_1(|\boldsymbol{l}_1|-|\boldsymbol{\mu}|-1)+k/2}{2}}\Big|D^{\boldsymbol{l}_1-\boldsymbol{\mu}}(\nabla u_\varepsilon)\Big|^{k_1}\le \mathcal{C},
    \end{split}
\end{equation}
where $\mathcal{C}$ is a constant not depending on $\varepsilon$.
 For the last term $J_4$, in an analogous way, we obtain
 
\begin{equation}\label{Stima J_4}
    \begin{split}
        J_4&=\sum_{r=1}^{|\boldsymbol{\alpha}|} c_1 (\varepsilon+|\nabla u_\varepsilon|^2)^{\frac{m-2+k}{2} -r} \sum_{s=2}^{|\boldsymbol{\alpha}|} \sum_{p_s(\boldsymbol{\alpha},{r})} \boldsymbol{\alpha}! \prod_{j=1}^s \frac{\left(D^{\boldsymbol{l}_j}(|\nabla u_\varepsilon|^2)\right)^{k_j}}{k_j!(\boldsymbol{l}_j!)^{k_j}}\nabla u_\varepsilon
        \\&\le C(m)\sum_{r=1}^{|\boldsymbol{\alpha}|}  \sum_{s=2}^{|\boldsymbol{\alpha}|} \sum_{p_s(\boldsymbol{\alpha},{r})} \prod_{j=1}^s  \sum_{\boldsymbol{\mu}\leq \boldsymbol{l}_j}  (\varepsilon+|\nabla u_\varepsilon|^2)^{\frac{k_j(|\boldsymbol{\mu}|-1)+k/s}{2}}\Big|D^{\boldsymbol{\mu}}(\nabla u_\varepsilon)\Big|^{k_j}\\
    &\qquad\qquad\qquad \times(\varepsilon+|\nabla u_\varepsilon|^2)^{\frac{k_j(|\boldsymbol{l}_j|-|\boldsymbol{\mu}|-1)+k/s}{2}}\Big|D^{\boldsymbol{l}_j-\boldsymbol{\mu}}(\nabla u_\varepsilon)\Big|^{k_j}\le \mathcal{C}.  
    \end{split}
\end{equation}
Collecting the estimates \eqref{J_2stima}, \eqref{J_3stima}, and \eqref{Stima J_4}, we conclude that the term $F_2$ is uniformly bounded in $L^\infty$.
In an analogous way, we can estimate the $L^\infty$-norm of $F_3$. Summing up, we claim \eqref{Final_conv254}.\\
For $\boldsymbol{\alpha}$ multi-index of order $m-1$, we aim to prove that
\[
|\nabla u|^kD^{\boldsymbol{\alpha}} \left(|\nabla u|^{m-2}\nabla u\right) \in L^\infty\bigl(B_{R'}(x_0)\bigr).
\]
Let $u_\varepsilon$ be the solution to \eqref{eq:problregol}. By estimate~\eqref{Final_conv254}, it is sufficient to show that
\[
(\varepsilon+|\nabla u_\varepsilon|^2)^\frac{k}{2}D^{\boldsymbol{\alpha}} \left((\varepsilon+|\nabla u_\varepsilon|^2)^\frac{m-2}{2}\nabla u_\varepsilon\right)
\  \  \text{converges a.e. to} \quad
|\nabla u|^kD^{\boldsymbol{\alpha}} \left(|\nabla u|^{m-2}\nabla u\right).
\]
For the moment, we assume that
$    f \in C^m(\Omega).$
Let $K \subset\subset B_{R'}$ be a compact set. We recall (see \cite{Ant1,Anto2,DB,L2}) that
\begin{equation}\label{evvai655}
u_\varepsilon \to u
\quad \text{in } C^{1,\beta}(K).
\end{equation}
Moreover, by Schauder estimates (see \cite{GT}), for any compact set $\tilde K \subset\subset B_{R'} \setminus Z_u,$
there exists a constant $C>0$, independent of $\varepsilon$, such that
\[
\|u_\varepsilon\|_{C^{m,\beta'}(\tilde K)} \le C.
\]
By the Arzelà-Ascoli theorem, up to a subsequence, the convergence in \eqref{evvai655} improves to
\[
u_\varepsilon \to u \quad \text{in } C^m(\tilde K).
\]
Therefore, using \eqref{Final_conv254} and letting $\varepsilon \to 0$, we obtain
\begin{equation}\label{frtm2}
\bigl\| |\nabla u|^kD^{\boldsymbol{\alpha}} \left(|\nabla u|^{m-2}\nabla u\right) \bigr\|_{L^\infty(B_{R'} \setminus Z_u)} \le \mathcal{C}.
\end{equation}
Furthermore, by Theorem~\ref{teoprinc}, for any $i=1,...,n$, we have $|\nabla u|^{m-2}\partial_{x_i}u \in W^{m-1,\hat q}(B_{2R})$ for some $\hat q > 1$. Hence, after 
$m-1$ applications of Stampacchia’s theorem (see \cite[Theorem~6.19]{S}), it follows that
\[
D^{\boldsymbol{\alpha}}\left(|\nabla u|^{m-2}\partial_{x_i}u\right) = 0 \quad \text{a.e. on the set } \{|\nabla u| = 0\}.
\]
This implies that
\begin{equation}\label{frtm3}
\bigl\| |\nabla u|^kD^{\boldsymbol{\alpha}} \left(|\nabla u|^{m-2}\nabla u\right) \bigr\|_{L^\infty(B_{R'})}
\le \mathcal{C},
\end{equation}
where $\mathcal{C}=\mathcal{C}\bigl(m,k,n,l,\zeta,m,R,R',p,
\|\nabla u\|_{L^\infty(B_{2R})},
\|f\|_{W^{m,l}(B_{2R})}\bigr)$.\\
The last step consists of removing the assumption $f\in C^m$. Let $f \in W^{m,l}_{loc}(\Omega) $. By standard density arguments we can infer that there exists a sequence $\{{f}_s\} \subset C^\infty(\Omega)$ such that
\begin{equation}\label{conv_f}
    {f}_s \rightarrow {f} \quad \text{in } W^{m,l}_{loc}(\Omega).
\end{equation}
To proceed, we consider a sequence $\{{u}_s\}$ of weak solutions to the following problem
\begin{equation}
\label{system_step3}
\begin{cases}
-\operatorname{ div}({|\nabla{ u_s}|}^{p-2}\nabla{ u_s})=   { f_s}& \mbox{in $B_{2R}$}\\
 u_s= u  & \mbox{on  $\partial B_{2R}$}.
\end{cases}
\end{equation}

\noindent By \eqref{frtm3} applied to $u_s$ we have

\begin{equation}\label{stimaimportante}
\left\||\nabla u_s|^kD^{\boldsymbol{\alpha}} \left(|\nabla u_s|^{m-2}\nabla u_s\right)\right\|_{L^{\infty}(B_{R'})}\le \mathcal{C}, 
\end{equation}

\noindent with $\mathcal{C}=\mathcal{C}(k,n,l,\zeta,m,R,R',p,\| \nabla u_s\|_{L^{\infty}(B_{2R})},\|f_s\|_{W^{m,l}( B_{2R})})$.\\
First, we note that, by \cite[Theorem 1.7]{L2}, we have 
\begin{equation}\label{uniformitàdeigradientisum}
    \|\nabla u_s\|_{L^{\infty}(B_{2R'})}\le C,
\end{equation}
where $C(p,R,n,\|f_s\|_{L^r(B_{2R})})$ is a positive constant and $r>n$.
Now, consider a compact set $K\subset \subset B_{R'}$. We recall that, 
\begin{equation}\label{evvai6558}
{u}_s\rightarrow  u \quad \text{in the norm }\|\cdot\|_{C^{1,\beta}(K)}.
\end{equation}
As above, since $f\in W^{m,l}_{loc}(\Omega)\hookrightarrow C^{m-2,\tilde{\beta}}_{loc}(\Omega)$, for some $\tilde{\beta}\in (0,1)$, by Schauder estimates we infer that  
\begin{equation}\label{convergenzaW22}
    u_s\rightarrow u \quad\  \text{in } \ C^m(\tilde K),
\end{equation}
for any compact set $\tilde K\subset\subset (B_{R'}\setminus Z_u)$.
By \eqref{conv_f}, \eqref{uniformitàdeigradientisum}, \eqref{evvai6558} and \eqref{convergenzaW22}, passing to the limit in \eqref{stimaimportante}, we deduce
\begin{equation*}\label{frtm4.0}
\left\||\nabla u|^kD^{\boldsymbol{\alpha}} \left(|\nabla u|^{m-2}\nabla u\right)\right\|_{L^{\infty}(B_{R'}\setminus Z_u)}\le \mathcal{C}, 
\end{equation*}
 with $\mathcal{C}=\mathcal{C}(k,n,l,\zeta,m,R,R',p,\| \nabla u\|_{L^{\infty}(B_{2R})},\|f\|_{W^{m,l}( B_{2R})})$.\\ By applying Stampacchia's theorem and arguing as in the case $f \in C^m$, we obtain the desired conclusion.\\
We conclude the paper by proving that $|\nabla u|^kD^{\boldsymbol{\alpha}-1} \left(|\nabla u|^{m-2}D^2 u\right)\in L^\infty{(B_{R'})}.$ We assume $f\in C^m(\Omega).$ An argument similar to the one employed in the proof of inequality \eqref{Final_conv254} yields
\begin{equation}\label{cosefinali}
    \left\|(\varepsilon+|\nabla u_\varepsilon|^2)^\frac{k}{2}D^{\boldsymbol{\alpha}-1} \left((\varepsilon+|\nabla u_\varepsilon|^2)^\frac{m-2}{2}D^2u_\varepsilon\right) \right\|_{L^\infty(B_{R'})} \leq \mathcal{C},
\end{equation}
with $\mathcal{C}$ independent of $\varepsilon$. The desired conclusion follows by repeating the arguments developed for the stress field $|\nabla u|^kD^{\boldsymbol{\alpha}} \left(|\nabla u|^{m-2}\nabla u\right)$. The case of a more general function $f$ can be treated analogously, and therefore the details are omitted.\\The case $n=2$ follows by Remark \ref{pluto}, reasoning as above. This concludes the proof.
\end{proof}
\section*{Acknowledgements}
The authors are supported by PRIN PNRR P2022YFAJH \emph{Linear and Nonlinear PDEs: New directions and applications.}  Felice Iandoli has been partially supported by \emph{INdAM-GNAMPA Project Stable and unstable phenomena in propagation of Waves in dispersive media} E5324001950001. D. Vuono has been partially supported by \emph{INdAM-GNAMPA Project Regularity and qualitative aspects of nonlinear PDEs via variational and non-variational approaches} E5324001950001.

\textbf{Declarations.} Data sharing is not applicable to this article as no datasets were generated or analyzed during the current study.
Conflicts of interest: The authors have no conflicts of interest to declare.

\end{document}